\documentclass{amsart}

\usepackage{amssymb}
\usepackage{amsfonts}
\usepackage{amsmath}
\usepackage{amsthm}
\usepackage{extarrows}
\usepackage{graphicx}
\usepackage{mathrsfs}
\usepackage{dsfont}
\usepackage{amscd}
\usepackage{multirow}
\usepackage{mathpazo}
\usepackage[all]{xy}
\usepackage[scaled]{helvet} % ss
\usepackage{courier} % tt
\normalfont
\usepackage[T1]{fontenc}
\usepackage{calligra}
\usepackage{verbatim}
\usepackage[usenames,dvipsnames]{color}
\usepackage[colorlinks=true,linkcolor=Blue,citecolor=Violet]{hyperref}

\newcommand{\N}{{\mathds{N}}}
\newcommand{\Z}{{\mathds{Z}}}

\newcommand{\R}{{\mathds{R}}}
\newcommand{\C}{{\mathds{C}}}
\newcommand{\T}{{\mathds{T}}}
\newcommand{\U}{{\mathds{U}}}
\newcommand{\D}{{\mathfrak{D}}}
\newcommand{\A}{{\mathfrak{A}}}
\newcommand{\B}{{\mathfrak{B}}}
\newcommand{\M}{{\mathfrak{M}}}

\newcommand{\G}{{\mathfrak{G}}}

\newcommand{\bigslant}[2]{{\raisebox{.2em}{$#1$}\left/\raisebox{-.2em}{$#2$}\right.}}
\newcommand{\Nbar}{\overline{\N}_\ast}
\newcommand{\Lip}{{\mathsf{L}}}
\newcommand{\Hilbert}{{\mathscr{H}}}
\newcommand{\Diag}[2]{{\operatorname*{Diag}\left[{#1}\middle| {#2} \right]}}
\newcommand{\dist}{{\mathsf{dist}}}

\newcommand{\qpropinquity}[1]{{\mathsf{\Lambda}_{#1}}}
\newcommand{\propinquity}[1]{{\mathsf{\Lambda}^\ast_{#1}}}
\newcommand{\cBall}[4]{{{#1}\left[#3,#4\right]_{#2}}}
\newcommand{\Kantorovich}[1]{{\mathsf{mk}_{#1}}}
\newcommand{\boundedLipschitz}[1]{{\mathsf{bl}_{#1}}}
\newcommand{\Haus}[1]{{\mathsf{Haus}_{#1}}}

\newcommand{\StateSpace}{{\mathscr{S}}}
\newcommand{\unital}[1]{{\mathfrak{u}{#1}}}
\newcommand{\mongekant}{{Mon\-ge-Kan\-to\-ro\-vich metric}}
\newcommand{\qms}{quantum locally compact metric space}

\newcommand{\pqms}{proper quantum metric space}
\newcommand{\pqpms}{pointed proper quantum metric space}

\newcommand{\Lqcms}{{\JLL} quan\-tum com\-pact me\-tric spa\-ce}
\newcommand{\Qqcms}[1]{${#1}$--quasi-Lei\-bniz quan\-tum com\-pact me\-tric spa\-ce}
\newcommand{\gQqcms}{quasi-\-Leib\-niz quan\-tum com\-pact me\-tric spa\-ce}
\newcommand{\qcms}{quantum compact metric space}
\newcommand{\lcqms}{quantum locally compact metric space}

\newcommand{\unit}{1}

\newcommand{\sa}[1]{{\mathfrak{sa}\left({#1}\right)}}

\newcommand{\indicator}[1]{{\chi_{#1}}}
\newcommand{\compacts}[1]{{\mathcal{K}\left({#1}\right)}}

\newcommand{\corner}[2]{{\indicator{#1}{#2}\indicator{#1}}}

\newcommand{\Loc}[3]{{\mathfrak{Loc}\left[ {#1} \middle\vert {#3} \right]_{#2}}}

\newcommand{\Adm}{{\mathrm{Adm}}}

\newcommand{\Moyal}[1]{\mathfrak{C}_{#1}(\R^2)}

\newcommand{\inner}[2]{{\left<{#1},{#2}\right>}}
\newcommand{\JLL}{Lei\-bniz}
\newcommand{\tr}{{\operatorname*{tr}}}
\newcommand{\Mod}[2]{{{\raisebox{.2em}{$#1$}\left/\raisebox{-.2em}{$#2$}\right.}}}

\newcommand{\dom}[1]{{\operatorname*{dom}({#1})}}
\newcommand{\codom}[1]{{\operatorname*{codom}({#1})}}
\newcommand{\diam}[2]{{\mathrm{diam}\left({#1},{#2}\right)}}
\newcommand{\covn}[3]{{\mathrm{cov}_{{#1}}\left({#2}\middle\vert{#3}\right)}}

\newcommand{\tunnelset}[3]{{\text{\calligra Tunnels}\,\left[\left({#1}\right)\stackrel{#3}{\longrightarrow}\left({#2}\right)\right]}}
\newcommand{\journeyset}[3]{{\text{\calligra Journeys}\left[\left({#2}\right)\stackrel{#1}{\longrightarrow}\left({#3}\right)\right]}}

\newcommand{\finitedimclass}[1]{{\text{\calligra FiniteDim}}\left({#1}\right)}

\newcommand{\bridgereach}[2]{{\varrho\left({#1}\middle|{#2}\right)}}
\newcommand{\bridgeheight}[2]{{\varsigma\left({#1}\middle|{#2}\right)}}

\newcommand{\bridgelength}[2]{{\lambda\left({#1}\middle|{#2}\right)}}
\newcommand{\treklength}[1]{{\lambda\left({#1}\right)}}
\newcommand{\trekset}[3]{{\text{\calligra Treks}\left(\left({#1}\right)\stackrel{#3}{\longrightarrow}\left({#2}\right)\right)}}

\newcommand{\bridgenorm}[2]{{\mathsf{bn}_{ {#1}  }\left({#2}\right)}}

\newcommand{\Jordan}[2]{{{#1}\circ{#2}}} %{{{#1}\circledcirc{#2}}}
\newcommand{\Lie}[2]{{\left\{{#1},{#2}\right\}}} %{{\left[{#1},{#2}\right]_\ast}}

\newcommand{\targetsettunnel}[3]{{\mathfrak{t}_{#1}\left({#2}\middle\vert{#3}\right)}}

\newcommand{\liftsettunnel}[3]{{{\mathfrak{l}_{#1}\left({#2}\middle\vert{#3}\right)}}}

\newcommand{\tunneldepth}[1]{{\delta\left(#1\right)}}
\newcommand{\tunnelreach}[1]{{\rho\left(#1\right)}}
\newcommand{\tunnellength}[1]{{\lambda\left(#1\right)}}
\newcommand{\journeylength}[1]{{\lambda\left(#1\right)}}
\newcommand{\tunnelextent}[1]{{\chi\left({#1}\right)}}
\newcommand{\prox}{{\mathrm{prox}}}
\newcommand{\co}[1]{{\overline{\mathrm{co}}\left(#1\right)}}
\newcommand{\alg}[1]{{\mathfrak{#1}}}
\newcommand{\wcs}[1]{{\mathscr{K}\left(\mathscr{S}\left({#1}\right)\right)}}
\newcommand{\opnorm}[1]{{\left|\mkern-1.5mu\left|\mkern-1.5mu\left| {#1} \right|\mkern-1.5mu\right|\mkern-1.5mu\right|}}
\newcommand{\GH}{{\mathrm{GH}}}
\newcommand{\GHl}{{\mathrm{GH}_l}}
\newcommand{\almostsubseteq}[2]{\;{\subseteq^{#1}_{#2}}\;}
\newcommand{\Set}[2]{{\left\{\begin{aligned} #1 \end{aligned} \middle\vert\begin{aligned} #2 \end{aligned}\right\}}}
\newcommand{\set}[2]{{\left\{{#1}\middle\vert{#2}\right\}}}

\newcommand{\sigmaKantorovich}[1]{{\mathsf{mk}^\sigma_{#1}}}
\newcommand{\lsa}[3]{{ \mathfrak{sa}\left[{#1}\middle|{#3}\right]_{#2}  }}
\newcommand{\indmor}[2]{{\underrightarrow{#1^{#2}}}}
\newcommand{\CondExp}[2]{{\mathds{E}\left({#1}\middle\vert{#2}\right)}}

\theoremstyle{plain}
\newtheorem{theorem}{Theorem}[section]

\newtheorem{corollary}[theorem]{Corollary}

\newtheorem{lemma}[theorem]{Lemma}
\newtheorem{proposition}[theorem]{Proposition}

\newtheorem{theorem-definition}[theorem]{Theorem-Definition}

\theoremstyle{definition}
\newtheorem{definition}[theorem]{Definition}

\newtheorem{notation}[theorem]{Notation}

\newtheorem{convention}[theorem]{Convention}

\theoremstyle{remark}

\newtheorem{example}[theorem]{Example}

\newtheorem{remark}[theorem]{Remark}

\renewcommand{\geq}{\geqslant}
\renewcommand{\leq}{\leqslant}

\newcommand{\conv}[1]{{*_{#1}}} %{{\widehat{\star_{#1}}}}

\numberwithin{equation}{section}

\makeatletter
\newcommand\@dotsep{4.5}
\def\@tocline#1#2#3#4#5#6#7{\relax
  \ifnum #1>\c@tocdepth % then omit
  \else
    \par \addpenalty\@secpenalty\addvspace{#2}%
    \begingroup \hyphenpenalty\@M
    \@ifempty{#4}{%
      \@tempdima\csname r@tocindent\number#1\endcsname\relax
    }{%
      \@tempdima#4\relax
    }%
    \parindent\z@ \leftskip#3\relax \advance\leftskip\@tempdima\relax
    \rightskip\@pnumwidth plus1em \parfillskip-\@pnumwidth
    #5\leavevmode\hskip-\@tempdima #6\relax
    \leaders\hbox{$\m@th
      \mkern \@dotsep mu\hbox{.}\mkern \@dotsep mu$}\hfill
    \hbox to\@pnumwidth{\@tocpagenum{#7}}\par
    \nobreak
    \endgroup
  \fi}
\makeatother

\begin{document}

\title{Quantum Metric Spaces and the Gromov-Hausdorff Propinquity}
\author{Fr\'{e}d\'{e}ric Latr\'{e}moli\`{e}re}
\email{frederic@math.du.edu}
\urladdr{http://www.math.du.edu/\symbol{126}frederic}
\address{Department of Mathematics \\ University of Denver \\ Denver CO 80208}

\date{\today}


\begin{abstract}
We present a survey of the dual Gromov-Hausdorff propinquity, a noncommutative analogue of the Gromov-Hausdorff distance which we introduced to provide a framework for the study of the noncommutative metric properties of C*-algebras. We first review the notions of {\lcqms s}, and present various examples of such structures. We then explain the construction of the dual Gromov-Hausdorff propinquity, first in the context of {\gQqcms s}, and then in the context of {\pqpms s}. We include a few new result concerning perturbations of the metrics on {\Lqcms s} in relation with the dual Gromov-Hausdorff propinquity.
\end{abstract}

\maketitle

\tableofcontents

%%%%%%%%%%%%%%%%%%%%%%%%%%%%%%%%%%%%%%%%%%%

\section{Introduction}

Noncommutative metric geometry proposes to study certain classes of noncommutative algebras as generalizations of algebras of Lipschitz functions over metric spaces, so that that methods from metric geometry may be applied to the analysis of such algebras. Quantum physical systems and other problems where noncommutative algebras appear naturally, such as in the study of certain types of singular spaces, provide the motivation for this research. Inspired by the work of Connes \cite{Connes89,Connes}, Rieffel introduced in \cite{Rieffel98a, Rieffel99} the notion of a compact quantum metric space and in \cite{Rieffel00} a generalization of the Gromov-Hausdorff distance \cite{Gromov,Gromov81}, thus providing in \cite{Latremoliere05, Rieffel01} a meaning to many approximations of classical and quantum spaces by matrix algebras found in the physics literature (see for instance \cite{Connes97,Seiberg99,tHooft02,Zumino98}), and pioneering a new set of techniques in the study of the geometry of C*-algebras (a sample of which is \cite{Ozawa05,Rieffel02,Rieffel09,Rieffel10, Rieffel10c}). This document presents some of the metric aspects of noncommutative geometry, and in particular, our understanding of the topologies which can be constructed on various large classes of quantum metric spaces, by means of our noncommutative analogues of the Gromov-Hausdorff distance \cite{Gromov81} called the dual Gromov-Hausdorff propinquity \cite{Latremoliere13,Latremoliere13b,Latremoliere13c,Latremoliere14,Latremoliere14b}.

The pursuit of an applicable theory of quantum metric spaces continues to raise many challenges, two of which have been addressed in our recent research. First, over the past decade, the search for a noncommutative analogue of the Gromov-Hausdorff distance which would be adequate for the study of the behavior of C*-algebraic structures with respect to metric convergence has proven an elusive query \cite{Rieffel10c}. We recently introduced a family of such metrics, called \emph{dual Gromov-Hausdorff propinquities} \cite{Latremoliere13,Latremoliere13b,Latremoliere13c,Latremoliere14}, adapted to the prospective applications of noncommutative metric geometry. Second, the search for a proper notion of \emph{locally compact quantum metric spaces} proved a delicate issue. Our work \cite{Latremoliere05b,Latremoliere12b} is the main contribution to a theory for such spaces, and we have recently added a framework for Gromov-Hausdorff convergence of proper quantum metric spaces \cite{Latremoliere14b}. The current document is build upon these two contributions.

The core concept for our work is the generalization of the {\mongekant} to the setting of C*-algebras, or in other terms, a generalization of the notion of a Lipschitz seminorm. The classical {\mongekant}, introduced by Kantorovich \cite{Kantorovich40} for his study of Monge's transportation problem, induces the weak* topology on the set of regular Borel probability measures on compact metric spaces. This property is, in fact, dual to the property that the set of real valued $1$-Lipschitz maps over a compact metric space is itself compact modulo the constant functions, thanks to Arz{\'e}la-Ascoli Theorem. Rieffel proposed to formalize these two properties and extend this duality to unital C*-algebras. Thus, a noncommutative Lipschitz seminorm, which Rieffel called a \emph{Lip-norm}, encodes a form of uniform equicontinuity and gives us a noncommutative Arz{\'e}la-Ascoli Theorem.

This picture does not extend to the locally compact metric space setting: the {\mongekant} associated to the Lipschitz seminorm from a locally compact metric space is an extended metric and does not typically metrize the weak* topology on the set of regular Borel probability measures. Instead, Dobrushin \cite{Dobrushin70} introduced a notion of ``metrically tights'' sets of probability measures, on which the restriction of the {\mongekant} does induces the weak* topology. Generalizing these ideas is not a straightforward matter, and in fact, there are only two approaches, both from our own work: while in \cite{Latremoliere05b}, we replace the {\mongekant} by the bounded-Lipschitz distance, in \cite{Latremoliere12b} we propose a noncommutative analogue of Dobrushin's tightness and study the {\mongekant}.

Another property of Lipschitz seminorms which connects it to the multiplication of functions is the Leibniz inequality. Yet, the Leibniz property of the Lipschitz seminorm does not seem to play a role in the topological properties of the associated {\mongekant}, and instead, it introduces some difficulties when trying to extend the Gromov-Hausdorff distance to quantum compact metric spaces \cite{kerr02,Rieffel10c}. Yet, current research in noncommutative metric geometry \cite{Rieffel06, Rieffel08, Rieffel09, Rieffel10, Rieffel10c, Rieffel11, Rieffel11b, Rieffel12, Rieffel14, Rieffel15} suggests that the Leibniz property, or at least some variant of this property relating Lip-norms with the C*-algebra multiplicative structure, is a desirable feature. Our recent work thus focused on addressing the challenges of working with Leibniz Lip-norms, and discovered that it actually provides benefits, such as ensuring that our new analogue of the Gromov-Hausdorff distance has the desired coincidence property. Consequently, the Leibniz property occupies a central role in these notes.

A motivation for the study of quantum compact metric spaces is to extend to noncommutative geometry the techniques and idea from metric geometry. In particular, the Gromov-Hausdorff distance introduces an intrinsic topology on the class of compact metric spaces, and a noncommutative analogue would provide a tool to construct approximations of quantum spaces, such as matricial approximations for quantum tori \cite{Connes97, Latremoliere05,Latremoliere13b}. Such approximations are at times found without a clear framework within the mathematical physics literature; yet they could provide a new mean to construct physical theory. Interestingly, the Gromov-Hausdorff distance appeared first in connection with the superspace approach to quantum gravity \cite{Wheeler68} in a proposal by Edwards \cite{Edwards75}.

The first noncommutative analogue of the Gromov-Hausdorff distance was due to Rieffel \cite{Rieffel00}. This distance was however only partially capturing the C*-algebraic structure underlying quantum compact metric spaces; in particular, distance zero did not imply *-isomorphism. Several alternatives to Rieffel's construction were offered \cite{kerr02,kerr09, li03,li05, li06} to address this matter, though none were built around the Leibniz property. Instead, they incorporate some quantum topological information in their analogues of the Gromov-Hausdorff distance, rather than tie together the quantum topological structure contained in C*-algebras with the quantum metric structures provided by Lip-norms.

Thus, we propose a different path for the construction of the dual Gromov-Hausdorff propinquity \cite{Latremoliere13,Latremoliere13b,Latremoliere13c,Latremoliere14}. Our approach relies on connecting the quantum topological structure and the quantum metric structure by requiring that a form of the Leibniz property holds for all Lip-norms considered in our construction. Our construction allows for quite some flexibility in the choice of which form the Leibniz property can take.

The dual Gromov-Hausdorff propinquity induces the same topology in the classical picture as the Gromov-Hausdorff distance, while also allowing to prove that quantum tori or the algebra of continuous functions on the sphere are limits of matrix algebras, with appropriate quantum metric structures. Our metric also solves the coincidence property issue --- *-iso\-mor\-phism is necessary for null distance, while being explicitly compatible with Leibniz Lip-norms, thus solving a decade of difficulties working with such seminorms. Moreover, our metric retains the natural features of Rieffel's original construction.  We thus propose that the dual Gromov-Hausdorff propinquity is the proper tool for the study of C*-algebraic structures under metric convergence, and a step in realizing this project is the recent work by Rieffel \cite{Rieffel08, Rieffel09, Rieffel10c, Rieffel15} on convergence of modules, which both motivated and benefited from our construction.

In this document, we first survey the notion of a {\lcqms}. We start with the class of Leibniz and quasi-Leibniz pairs, which are the basic ingredients of the theory \cite{Connes89} and allow us to define noncommutative versions of the {\mongekant} and bounded-Lipschitz metrics. We provide many examples of such structures. We then turn to the duality between Arz{\'e}la-Ascoli Theorem and the properties of the {\mongekant} and bounded-Lipschitz metrics from the noncommutative perspective. Rieffel pioneered these matters in his work \cite{Rieffel98a,Rieffel99}; our exposition however begins with our own extension of his original result to non-unital C*-algebras, in order to make our presentation less redundant. The compact quantum metric spaces introduced by Rieffel are presented in Section (\ref{compact-sec}).

We then move to our presentation of the dual Gromov-Hausdorff propinquity. The dual Gromov-Hausdorff propinquity is a noncommutative analogue of the Gromov-Hausdorff distance, originally designed to address issues which arose when applying the construction of Rieffel's quantum Gromov-Hausdorff distance while imposing that all involved Lip-norms are Leibniz. This construction indeed leads to an object, called the proximity in \cite{Rieffel10c}, which is not known to be even a pseudo metric, as the triangle inequality may fail. Thus, our Gromov-Hausdorff propinquity is a way to construct an actual metric on Leibniz and, more generally, quasi-Leibniz quantum compact metric spaces in order to address the same problems as the proximity aimed at solving, hence our choice of terminology, as propinquity and proximity are synonymous. It should be emphasized that taking the Leibniz property as a core feature in our construction actually allows us to fix the coincidence property of Rieffel's original metric as well.

We recall the basic properties of the Gromov-Hausdorff propinquity and the overall strategies to establish them. We also introduce an important specialization of the dual Gromov-Hausdorff propinquity, the quantum propinquity, for which several examples of convergence are discussed. The role of this specialized metric is yet to be fully understood, but it seems to be a useful mean to prove convergences for the dual Gromov-Hausdorff propinquity and to discuss convergence for matrix algebras over convergent {\Lqcms s} \cite{Rieffel15}. We also present a generalization of Gromov's compactness theorem for our new metric. We conclude with a section where we summarize our proposal for a topology on the class of {\pqpms s} inspired by the Gromov-Hausdorff convergence for proper metric spaces.

While this work is a survey, we do present a brief new result on perturbations of metrics which applies, for instance, to conformal deformations of spectral triples which give rise to {\Lqcms s}.

\section{Locally Compact Quantum Metric Spaces}

Alain Connes introduced in \cite{Connes89} the idea of a metric on a noncommutative space, motivated in part by the interaction between the notion of growth for a discrete group and the properties of associated unbounded Fredholm modules on the C*-algebras of such groups. This interaction, in turn, is inspired by Gromov's work \cite{Gromov81}. In \cite{Connes89}, the metric of a noncommutative space is a by-product of the central notion of spectral triple, and its original purpose seemed to have been a mean to prove \cite[Proposition 1]{Connes89} that the standard spectral triple constructed from the Dirac operator of a spin Riemannian manifold encodes the metric on the underlying length space.

The general form of Connes' metric for C*-algebras endowed with a spectral triple, given in \cite[Proposition 3]{Connes89}, is naturally interpreted as a noncommutative analogue of the {\mongekant} introduced by Kantorovich in 1940 \cite{Kantorovich40} in his work on the transportation problem of Monge, and since then a very important tool of probability theory \cite{Dudley}, transportation theory \cite{Villani09}, and many other fields such as fractal theory \cite{Barnsley93}.

The {\mongekant}, of course, is defined on any metric space, as long as one is flexible in one's notion of metric: for example, in the case of locally compact metric spaces, the {\mongekant} is, in fact, an extended metric, i.e. it may take the value $\infty$ between two probability measures. While difficulties arise in the non-compact setting, the fundamental nature of the {\mongekant} is a strong motivation to extend its construction, following Connes' initial idea, to C*-algebras. A welcomed consequence of such a generalization is the possibility to import techniques from metric geometry in noncommutative geometry. Inspired by Connes' proposal, the strategy followed by Rieffel \cite{Rieffel98a,Rieffel99} and later on ourselves \cite{Latremoliere05b, Latremoliere12b} is to find a proper analogue of Lipschitz seminorms in noncommutative geometry, of which the {\mongekant} will be the dual. When working in the noncompact setting, we also consider a variant of the {\mongekant}, known as the bounded-Lipschitz distance, which is at times, better behaved.

This section presents the notion of a quantum metric space. We begin with a brief review of the classical picture, to serve as our model. We then isolate, one by one, the properties that Lipschitz seminorms possess and we would wish to keep when working over general C*-algebras. The simplest property is encoded in the notion of a Lipschitz pair, which is the minimal ingredient to define the {\mongekant}. We then note that Lipschitz seminorms are lower semi-continuous, which makes notions of morphisms between Lipschitz pairs easier to work with. A more delicate property, which is easy to state yet at times challenging to use, is the Leibniz property. While most examples possess this property, its actual role took some time to be uncovered, and is more related to our next section, where we will discuss the dual Gromov-Hausdorff propinquity. Last, the essential property of Lipschitz seminorms relate to the Arz{\'e}la-Ascoli theorem --- and through duality, to the topology induced by the {\mongekant}. We begin our exposition on this last property in the non-unital setting, where it is beneficial to first work with the bounded-Lipschitz distance. We eventually provide a full picture of how one may define a {\lcqms}.

\subsection{The Classical Model}

Gel'fand duality \cite{Dixmier, Arveson76, Pedersen79} suggests that the proper mean to algebraically encode the topology of a locally compact Hausdorff space $X$ is to work with the C*-algebra $C_0(X)$ of $\C$-valued continuous functions on $X$, vanishing at infinity (i.e. continuous functions on the one-point compactification of $X$, vanishing at the infinity point). Thus, we seek to encode the metric information given by a locally compact, metric space $(X,\mathrm{d})$ in some manner at the level of the C*-algebra $C_0(X)$. When $X$ is compact, we will denote $C_0(X)$ simply by $C(X)$.

Let $(X,\mathrm{d})$ be a locally compact metric space. A natural dual notion to the metric is given by the Lipschitz seminorm, defined for any function $f : X \rightarrow \C$ by:
\begin{equation}\label{Lipschitz-seminorm-eq}
\mathrm{Lip}(f) = \sup \left\{ \frac{|f(x)-f(y)|}{\mathrm{d}(x,y)} : x, y\in X, x\not= y \right\}\text{.}
\end{equation}
We are thus led to two questions:
\begin{enumerate}
\item Can we recover the metric $\mathrm{d}$ from the Lipschitz seminorm $\mathrm{Lip}$?
\item What properties of the seminorm $\mathrm{Lip}$ remain meaningful in the larger context of noncommutative C*-algebras, yet capture the usefulness of the Lipschitz seminorm as a tool of analysis?
\end{enumerate}

By the Riesz-Markov-Kakutani Theorem , the dual of $C_0(X)$ consists of the regular $\C$-valued Borel measures, and in particular, the state space $\StateSpace(C_0(X))$ of $C_0(X)$ consists of the regular Borel probability measures.

The foundation upon which noncommutative metric geometry is built, and which owes to the study of the Monge transportation problem, consists of the metric induced by the dual seminorm of $\mathrm{Lip}$ on the set of regular probability measures $\StateSpace(C_0(X))$ of $X$ by setting, for all $\varphi, \psi \in \StateSpace(C_0(X))$:
\begin{equation}\label{Kantorovich-metric-eq}
\Kantorovich{\mathrm{Lip}}(\varphi,\psi) = \sup \left\{ \left|\int_X f \,d\varphi - \int_X f \, d\psi \right| : f \in C_0(X), \mathrm{Lip}(f)\leq 1 \right\}\text{.}
\end{equation}
This metric was introduced in 1940 by Kantorovich \cite{Kantorovich40} in his pioneering work on Monge's transportation problem. In his original work, Kantorovich expressed the distance between two probability measures $\varphi$ and $\psi$ over a metric space $(X,\mathrm{d})$ as:
\begin{equation*}
\inf\left\{ \int_X \mathrm{d}(x,y)\,d\pi(x,y) : \text{$\pi$ is a probability measure on $X^2$ with marginals $\varphi,\psi$} \right\}\text{.}
\end{equation*}

The duality relationship between the Monge-Kantorovich metric and the Lipschitz seminorm was first made explicit in 1958 by Kantorovich and Rubinstein \cite{Kantorovich58}, leading to the form of the metric given by Expression (\ref{Kantorovich-metric-eq}), which will serve as the basis for our work. The first occurrence of a noncommutative analogue of the {\mongekant}, in the context of spectral triples, can be found in the work of Connes \cite{Connes89}.

The {\mongekant} is also known as the Wasserstein metric \cite{Wasserstein69}, thus named by Dobrushin \cite{Dobrushin70}, the earth mover metric, the Hutchinson metric, and likely other names. Our choice of terminology attempts to reflect the historical development of this important construction and its original motivation.

The {\mongekant} associated with a \emph{compact} metric space $(X,\mathrm{d})$ possesses two fundamental properties which address the questions raised at the start of this section. First of all, the map $x \in X \mapsto \delta_x \in \StateSpace(C(X))$, where $\delta_x$ is the Dirac probability measure at $x \in X$, is an isometry from $(X,\mathrm{d})$ into $(\StateSpace(C(X)), \Kantorovich{\mathrm{Lip}})$. Since $\{\delta_x : x\in X\}$ endowed with the weak* topology is the Gel'fand spectrum of $C(X)$, the isometry $x\in X\mapsto \delta_x$ is indeed natural.

Moreover, and very importantly, the {\mongekant} $\Kantorovich{\mathrm{Lip}}$ extends $\mathrm{d}$ to the entire state space $\StateSpace(C(X))$ of the compact metric space $(X,\mathrm{d})$, and it metrizes the weak* topology on $\StateSpace(C(X))$. This fundamental property of $\Kantorovich{\mathrm{Lip}}$ is the root cause of its importance in probability theory and related fields, and will serve as the starting point for the theory of quantum metric spaces.

The {\mongekant} associated to a noncompact, locally compact Hausdorff space is a much more complicated object. To begin with, it is not generally true that one may recover the original metric from which the {\mongekant} is constructed. More challenging is the observation that the topology of the {\mongekant} on the space of regular Borel probability measures is not the weak* topology any longer.

Our research unearthed two approaches to handle the noncompact, locally compact quantum metric space theory. Our newest methods \cite{Latremoliere12b} involve extending to the noncommutative realm a result from Dobrushin \cite{Dobrushin70} which identifies a certain type of sets of regular probability measures whose weak* topology is metrized by the {\mongekant}.

Another approach \cite{Latremoliere05b} consists in using a variant of the {\mongekant}, called the bounded-Lipschitz metric and introduced by Fortet and Mourier \cite[Section 5]{Fortet-Mourier-53}. For any two $\varphi, \psi \in \StateSpace(C_0(X))$ and $r > 0$, we thus set:
\begin{equation*}
\boundedLipschitz{\mathrm{Lip},r}(\varphi, \psi) = \sup\left\{ \left| \int_X f \, d\varphi - \int_X g\, d\psi  \right| : f \in C_0(X), \|f\|_{C_0(X)} \leq 1, \mathrm{Lip}(f) \leq r  \right\}\text{.}
\end{equation*}

Whenever $(X,\mathrm{d})$ is a separable locally compact metric space, the bounded-Lipschitz distance $\boundedLipschitz{\mathrm{Lip},r}$ metrizes the weak* topology on $\StateSpace(C_0(X))$. Moreover, if $(X,\mathrm{d})$ is a proper metric space (i.e. all its closed balls are compact), then the map $x \in X \mapsto \delta_x$ is an isometry when restricted to any ball of radius at most $r$. In particular, if $(X,\mathrm{d})$ is in fact compact, then for $r > 0$ larger than the diameter of $(X,\mathrm{d})$, the bounded Lipschitz metric $\boundedLipschitz{\mathrm{Lip}, r}$ agrees with the {\mongekant} $\Kantorovich{\mathrm{Lip}}$. Thus, the bounded-Lipschitz metrics provide a possible alternate approach to quantum metric spaces, which we explored in our research as well.

In this section, we shall describe a framework which generalizes the construction of the {\mongekant} and the bounded-Lipschitz metrics. This framework raises many technical challenges, yet will allow us to later develop noncommutative analogues of the Gromov-Hausdorff distance.

\subsection{Leibniz Pairs}

This section introduces various structures involved in our final definition of a {\qms}. The following notation will be used throughout this document:
\begin{notation}
Let $\A$ be a C*-algebra. The norm of $\A$ is denoted $\|\cdot\|_\A$ and the state space of $\A$ is denoted by $\StateSpace(\A)$. The set of self-adjoint elements of $\A$ is denoted by $\sa{\A}$.
\end{notation}

\subsubsection{Lipschitz Pairs}

At the root of our work is a pair $(\A,\Lip)$ of a C*-algebra and a seminorm $\Lip$ which enjoys various properties. The following definition contains the minimal assumptions we will make on such a pair.

\begin{notation}
Let $\A$ be a C*-algebra. The smallest unital C*-algebra containing $\A$, i.e. either $\A$ if $\A$ is unital, or its standard unitization $\A\oplus\C$ \cite{Pedersen79} otherwise, is denoted by $\unital{\A}$. The unit of $\unital{\A}$ is always denoted by $\unit_{\A}$. Note that $\sa{\unital{\A}} = \sa{\A} \oplus \R\unit_\A$ if $\A$ is not unital. We identify every state of $\A$ with its unique extension as a state of $\unital{\A}$. Under this identification, the state space of $\unital{\A}$ equals to the quasi-state space of $\A$ \cite{Pedersen79}, and the weak* topology $\sigma(\A^\ast,\A)$ on $\StateSpace(\A)$ agrees with the weak* topology $\sigma(\unital{\A}^\ast,\unital{\A})$ restricted to $\StateSpace(\A)$.
\end{notation}

\begin{definition}[\cite{Rieffel98a}, \cite{Latremoliere12b}]\label{Lipschitz-pair-def}
A \emph{Lipschitz pair} $(\A,\Lip)$ is a pair of a C*-algebra $\A$ and a seminorm $\Lip$ on a dense subspace $\dom{\Lip}$ of $\sa{\unital{\A}}$ and such that:
\begin{equation*}
\{ a \in \sa{\unital{\A}} : \Lip(a) = 0 \} = \R\unit_\A\text{.}
\end{equation*}
A unital Lipschitz pair $(\A,\Lip)$ is a Lipschitz pair where $\A$ is unital.
\end{definition}

We wish to emphasize that the C*-algebra $\A$ of a Lipschitz pair $(\A,\Lip)$ may not be unital; if not then $\Lip$ is in fact a norm on some dense subspace of $\sa{\A}$. To ease our notations later on, we will employ the following convention throughout this document:

\begin{convention}
We adopt the usual convention that if $\Lip$ is a seminorm defined on a dense subspace $\dom{\Lip}$ of a topological vector space $V$, and if $a\in V$ is not in the domain of $\Lip$, then $\Lip(a) = \infty$. With this convention, we observe that:
\begin{equation*}
\dom{\Lip} = \left\{ a \in V : \Lip(a) < \infty \right\} \text{.}
\end{equation*}
Note that with this convention, we do not introduce any ambiguity when talking about lower semi-continuous seminorms by exchanging the original seminorm with its extension.

Moreover, with this convention, we set $0\cdot\infty = 0$.
\end{convention}

The central construction of noncommutative metric geometry is the extension of the {\mongekant} \cite{Kantorovich40, Kantorovich58} to any Lipschitz pair:

\begin{definition}[\cite{Kantorovich40}, \cite{Rieffel98a}, \cite{Latremoliere12b}]\label{Kantorovich-def}
The \emph{\mongekant} $\Kantorovich{\Lip}$ associated with a Lipschitz pair $(\A,\Lip)$ is the extended metric on the state space $\StateSpace(\A)$ of $\A$ defined by setting for all $\varphi,\psi \in \StateSpace(\A)$:
\begin{equation*}
\Kantorovich{\Lip}(\varphi,\psi) = \sup \left\{ |\varphi(a) - \psi(a)| : a\in\sa{\A} \text{ and }\Lip(a)\leq 1 \right\}\text{.}
\end{equation*}
\end{definition}

The {\mongekant} is, as defined, an extended metric: Definition (\ref{Lipschitz-pair-def}) ensures that, for any Lipschitz pair $(\A,\Lip)$, and for any $\varphi, \psi\in\StateSpace(\A)$, we have $\Kantorovich{\Lip}(\varphi,\psi)=0$ if and only if $\varphi = \psi$ thanks to the density of the domain of $\Lip$ in $\sa{\A}$, and moreover $\Kantorovich{\Lip}$ is obviously symmetric and satisfies the triangle inequality. However, in general, $\Kantorovich{\Lip}$ may take the value $\infty$. 

\begin{example}[Fundamental Example]\label{fundamental-LP-ex}
Let $(X,\mathrm{d})$ be a locally compact metric space, and let $\mathrm{Lip}$ be the Lipschitz seminorm on $\sa{C_0(X)}$ induced by $\mathrm{d}$ via Expression (\ref{Lipschitz-seminorm-eq}). Then $(C_0(X), \mathrm{Lip})$ is a Lipschitz pair, and $\Kantorovich{\mathrm{Lip}}$ is the original {\mongekant} of Expression (\ref{Kantorovich-metric-eq}). 

If $X = \R$ with its usual metric, in particular, and if we denote the Dirac probability measure at $x\in X$ by $\delta_x$, then we note that:
\begin{equation*}
\Kantorovich{\mathrm{Lip}}\left(\delta_0, \sum_{n\in\N} \frac{1}{2^{n}} \delta_{2^{2n}} \right) = \infty \text{.}
\end{equation*}

However, when $(X,\mathrm{d})$ is bounded, then $\Kantorovich{\mathrm{Lip}}$ is an actual metric. If moreover $(X,\mathrm{d})$ is compact, then for all $y\in X$, the map $f_y : x\in X \mapsto \mathrm{d}(x,y)$ satisfies $\mathrm{Lip}(f_y)\leq 1$ and $f_y \in C(X)$, and thus one easily checks that the map $x\in X\mapsto \delta_x$ is an isometry from $(C(X),\mathrm{d})$ into $(\StateSpace(C_0(X)),\Kantorovich{\mathrm{Lip}})$. More generally, if $(X,\mathrm{d})$ is proper, i.e. all its closed balls are compact, then $x\in X \mapsto \delta_x\in \StateSpace(C_0(X))$ is still an isometry \cite{Latremoliere14b}.

When $(X,\mathrm{d})$ is not proper, the map $x\in X\mapsto \delta_x$ need no longer be an isometry. For instance, for $X = (0,1)$ with its usual metric, since if $f \in C_0(X)$ and $\mathrm{Lip}(f) \leq 1$ then $\|f\|_{C_0(X)}\leq \frac{1}{2}$, and thus two states are at most at distance $\frac{1}{2}$ from each other for the {\mongekant}.
\end{example}

The metric given by Definition (\ref{Kantorovich-def}) has a long history and many names, as we discussed in the introductory section of this chapter. Our formulation is the result of some evolution of the idea of generalizing the {\mongekant} to noncommutative geometry. The first occurrence of such a construction is due to Connes \cite{Connes89}, where the seminorm $\Lip$ was obtained by means of a spectral triple, and the Lipschitz pairs thus constructed are unital.

\begin{notation}
If $\Hilbert$ is a Hilbert space and $T : \Hilbert \rightarrow \Hilbert$ is a linear map, then the operator norm of $T$ is denoted by $\opnorm{T}$.
\end{notation}

\begin{example}[\cite{Connes89}]\label{Connes-LP-ex}
Let $\A$ be a C*-algebra, $\pi$ a faithful *-re\-pre\-sen\-ta\-tion of $\A$ on some Hilbert space $\Hilbert$, and $D$ a self-adjoint, possibly unbounded operator on $\Hilbert$ such that:
\begin{enumerate}
\item $1+D^2$ has a compact inverse,
\item the *-subalgebra:
\begin{equation*}
\{a\in\A: \text{the closure of $[D,\pi(a)]$ is bounded} \}
\end{equation*}
is dense in $\A$, 
\item the set:
\begin{equation*}
\bigslant{\{a\in\A : \opnorm{[D,\pi(a)]} \leq 1\}}{\C\unit_\A}
\end{equation*}
is bounded. 
\end{enumerate}

For all $a\in\sa{\A}$, we define $\Lip(a) = \opnorm{[D,\pi(a)]}$. Then the pair $(\A,\Lip)$ with $\Lip : a\in\sa{\A} \mapsto \|[D,\pi(a)]\|$ is a Lipschitz pair. Indeed, if $\Lip(a) = 0$ for some $a\in\sa{\A}$, then for all $t\in\R$ then $\Lip(ta) = 0 \leq 1$, and by the third condition on our triple $(\A,\Hilbert,D)$, we must conclude that $a\in \R\unit_\A$.

A triple $(\A,\Hilbert,D)$ satisfying the two first conditions above is called an unbounded Fredholm module or a spectral triple \cite{Connes89, Connes}. When constructed from a spectral triple, the {\mongekant} $\Kantorovich{\Lip}$ is at times called the Connes' metric. For the purpose of noncommutative metric geometry, the condition that $D$ must have compact resolvant has yet to find a role; however this notion is essential for the development of noncommutative differential geometry \cite{Connes}.

In \cite{Connes89}, an example of such a structure is given by a compact connected Riemannian spin manifold $M$, with $\Hilbert$ the Hilbert space of square integrable sections of the spin bundle of $M$ associated to the cotangent bundle, and $D$ the Dirac operator of $M$. The {\mongekant} associated with the Lipschitz pair $(C(M), \Lip)$ obtained by the above construction, letting $C(M)$ act by multiplication on $\Hilbert$, is shown to extend the distance function on $M$ induced by the Riemannian metric \cite[Proposition 1]{Connes89}. 

Another example in \cite{Connes89} is given by $\A$ being the reduced C*-algebra of some discrete group $G$, while $\pi$ is the left regular representation on $\ell^2(G)$, and $D$ is the multiplication operator on $\ell^2(G)$ by a length function on $G$. A length function $\ell : G \rightarrow [0,\infty)$ is a map such that, for all $g, g' \in G$:
\begin{enumerate}
\item $\ell(g) = 0$ if and only if $g$ is the unit of $G$,
\item $\ell(gg') \leq \ell(g) + \ell(g')$,
\item $\ell(g^{-1}) = \ell(g)$.
\end{enumerate}
\end{example}

Now, the next step in the evolution of Definition (\ref{Kantorovich-def}) was the introduction by Rieffel \cite{Rieffel98a} of the concept of a quantum compact metric space, allowing for more general choices of seminorms in Lipschitz pairs. An example of central importance to our work, and which is found in the foundational paper \cite{Rieffel98a}, is as follows:

\begin{example}[\cite{Rieffel98a}]\label{ergodic-LP-ex}
Let $\alpha$ be a strongly continuous action of a compact group $G$ by *-automorphisms on a unital C*-algebra $\A$. For any continuous length function $\ell$ on $G$, we may define for all $a\in\sa{\A}$:

\begin{equation*}
\Lip(a) = \sup\left\{ \frac{\|\alpha^g(a)-a\|_\A}{\ell(g)} : g \in G, \text{ $g$ not the unit of $G$} \right\}\text{.}
\end{equation*} 

In \cite{Rieffel98a}, Rieffel proves that $(\A,\Lip)$ is a Lipschitz pair if and only if:
\begin{equation*}
\{a\in\sa{\A} : \forall g \in G \quad \alpha^g(a) = a \} = \R\unit_\A \text{.}
\end{equation*}
An action for which the fixed point C*-subalgebra is thus reduced to the scalars is called an \emph{ergodic action}. As we shall see later, Rieffel showed that in fact, ergodicity implies additional properties on the Lipschitz pair $(\A,\Lip)$.

A very important special case of this construction is given by the quantum tori $\A$ on which the tori acts via the dual action.

We note that a length function $\ell$ on $G$ allows one to define a left-invariant distance on $G$ by setting $d : g,g' \in G\mapsto \ell(g^{-1}g')$, and conversely given a left-invariant distance on $G$, the distance $\ell$ from any element of $G$ to the unit of $G$ is a length function. When $G$ is a compact metrizable group, there always exist a continuous left invariant metric, and thus a continuous length function.
\end{example}

It should be noted that Example (\ref{ergodic-LP-ex}) is not given as a Lipschitz pair from a spectral triple, though in \cite{Rieffel98a}, a related metric from the natural spectral triple on the quantum tori is also constructed. Moreover, for quantum tori, the construction of Example (\ref{Connes-LP-ex}) involving the length function may be applied as well, leading to interesting Lipschitz pairs over the quantum tori as well \cite{Ozawa05}.

Now, the type of objects found in the earlier work of Rieffel \cite{Rieffel98a, Rieffel99, Rieffel00} on compact quantum metric spaces was a bit more general than unital Lipschitz pairs. Indeed, Rieffel worked with pairs $(\A,\Lip)$ of an order-unit space $\A$ together with a seminorm $\Lip$ on $\A$. Of course, order-unit spaces are subspaces of the self-adjoint part of C*-algebras \cite{Alfsen01}, but in general, they do not have to be complete or closed under the Jordan or the Lie product --- in other words, the multiplicative structure is not playing a role. In sight of our Definition (\ref{Kantorovich-def}), one may naturally conclude that the multiplicative structure is not essential in the definition of quantum metric spaces. We will return to this matter in this document. We shall however emphasize that \emph{for our work, the proper setting is indeed given by the Lipschitz pairs, as we specifically focus on studying noncommutative analogues of the Gromov-Hausdorff distance which are well-suited to working with C*-algebras}. 

Another example of a Lipschitz pair is given by the spectral triples constructed in \cite{Dabrowski05} on the quantum groups $SU_q(2)$: Aguilar, one of our PhD student, showed that such spectral triples give rise to Lipschitz pairs \cite{Aguilar}. Another spectral triple on $SU_q(2)$ which gives rise to a Lipschitz pair is given in \cite{Chakraborty03}; in addition, several examples of Lipshitz pairs on quantum groups and associated spaces can be found in \cite{Voigt14, Li04}.

Our interest in the development of a theory of quantum \emph{locally compact} metric spaces, rather led us to the formulation of our Definition (\ref{Kantorovich-def}) in \cite{Latremoliere12b}, as the third step in the evolution of the noncommutative notion of {\mongekant}. In this setting, an important example which we employed in our work is given by another spectral triple, albeit in the non-unital setting. 

\begin{example}[\cite{Varilly04}]\label{Moyal-LP-ex}
A spectral triple on the C*-algebra of compact operators on a separable Hilbert space, seen as the Moyal plane, is constructed in \cite{Varilly04}. We refer to \cite{Folland89, Cagnache11,GraciaBondia88a, GraciaBondia88b, Varilly04} for detailed expositions on the Moyal plane as a noncommutative geometric object.

Fix $\theta > 0$. The Moyal plane $\mathfrak{M}_\theta$ is informally the quantum phase space of the quantum harmonic oscillator. It is a strict quantization of the usual plane $\R^2$ toward the canonical Poisson bracket on $C_0(\R^2)$, re-scaled by a ``Plank constant'' $\theta$.  The C*-algebra of continuous observables on the Moyal plane is the C*-algebra $\mathfrak{M}_\theta = C^\ast (\R^2,\sigma_\theta)$ where:
\begin{equation*}
\sigma : (p_1,q_1),(p_2,q_2) \in \R^2\times \R^2 \longmapsto \exp(2i\pi \theta (p_1q_2 - p_2q_1))
\end{equation*}
is a bicharacter on $\R^2$. This C*-algebra is easily seen to be *-isomorphic to the C*-algebra $\mathfrak{K}$ of compact operators on $L^2(\R)$. However, we follow here the standard presentation of the Moyal plane, which uses a twisted product (rather than a twisted convolution) obtained by conjugating the twisted convolution by the Fourier transform. This formulation provides a representation of $C^\ast(\R^2,\sigma)$ on $L^2(\R^2)$ which then serves as the basis for the construction of a Lipschitz pair.

Let $\mathcal{S}$ be the space of $\C$-valued Schwartz functions on $\R^2$. For any $f,g\in \mathcal{S}$ we define:
\begin{equation}
f\star g : x\in \R^2 \mapsto \frac{1}{\left(\pi\theta\right)^2}\iint_{\R^2\times \R^2} f(x+y)g(x+z)\sigma(y,z) \, dy dz \text{.}
\end{equation}
The pair $(\mathcal{S},\star)$ is an associative *-algebra, and is a *-algebra  which we denote by $\mathcal{S}_\theta$ if one takes complex conjugation as the *-operation. The integral defines a trace on $\mathcal{S}_\theta$.

Let $\pi$ be the representation $f \in \mathcal{S}_\theta \mapsto [g \in L^2(\R^2) \mapsto f\star g]$ --- one checks this is a well-defined *-representation and can be extended to $\mathfrak{M}_\theta$. For any nonzero vector $u \in \R^2$, we write $\frac{\partial}{\partial u}$ for the directional derivative along $u$, seen as as unbounded operator of  $L^2(\R^2)$. Denote by $\partial$ the partial derivative $\frac{\partial}{\partial \left(\frac{\sqrt{2}}{2},\frac{\sqrt{2}}{2}\right)} = \frac{\sqrt{2}}{2}\left(\frac{\partial}{\partial (1,0)} - i\frac{\partial}{\partial (0,1)}\right)$ on $L^2(\R^2)$. 

Then we define the following operators on $L^2(\R^2)\otimes \C^2$:
\begin{equation*}
\forall c \in \mathfrak{M}_\theta \;\; \Pi(c)= \begin{pmatrix}
\pi(c) & 0 \\ 0 & \pi(c) \end{pmatrix} \quad\text{ and }\quad 
D = -i\sqrt{2}\begin{pmatrix} 0 & \overline{\partial} \\ \partial & 0 \end{pmatrix} \text{.}
\end{equation*}

Then by \cite{Varilly04} $(\mathcal{S}_\theta,\Pi,D)$ is a candidate for a spectral triple for the Moyal plane $\mathfrak{M}_\theta$. In particular, $\Pi$ is a *-representation of $\mathfrak{M}_\theta$ on $L^2(\R^2)\otimes \C^2$, and the set $\{ a \in \sa{\mathcal{S}_\theta} : \opnorm{[\Pi(a),D]} < \infty \} = \sa{\mathcal{S}_\theta}$ is norm dense in $\mathfrak{M}_\theta$. Moreover, since $\Pi$ is faithful, one checks that for all $a\in \mathcal{S}_\theta$, if $[\Pi(a),D] = 0$ then $a = 0$ \cite{Varilly04}.

We thus obtain a Lipschitz pair $(\mathfrak{M}_\theta,\Lip_\theta)$ where:
\begin{equation}\label{moyal-lip-eq}
\Lip_\theta : c \in \mathcal{S}_\theta \mapsto \opnorm{[D,\Pi(c)]}\text{.}
\end{equation}

The metric properties of the Moyal plane for this particular spectral triple have been studied in particular in \cite{Cagnache11, Martinetti11, Wallet12}.
\end{example}

We were led in \cite{Latremoliere12b} to the observation that Lipschitz pairs are not quite sufficient to define the notion of a quantum locally compact metric space: another ingredient will be required. Nonetheless, having a Lipschitz pair at least allows for the definition of a {\mongekant} on such examples as the Moyal planes, and provides all we need to study the bounded-Lipschitz distances.

As indicated in the introduction to this section, however, the behavior of the {\mongekant} for non-unital Lipschitz pairs is a complicated issue --- already made clear in the case of commutative, non-unital Lipschitz pairs. In our earlier work on this particular challenge \cite{Latremoliere05b}, we generalized another metric on spaces of probability to the noncommutative setting: the bounded-Lipschitz distance, whose origin seems to go back at least to Fortet and Mourier \cite{Fortet-Mourier-53}. 

For our purpose, the noncommutative definition reads as follows:
\begin{definition}[\cite{Latremoliere05b}, Definition 2.3]\label{bounded-Lipschitz-def}
The \emph{Bounded-Lipschitz distance} $\boundedLipschitz{\Lip,r}$ associated with a Lipschitz pair $(\A,\Lip)$ and some $r \in (0,\infty)$ is defined, for any two states $\varphi,\psi \in \StateSpace(\A)$, by:
\begin{equation*}
\boundedLipschitz{\Lip, r}(\varphi,\psi) = \sup\left\{ |\varphi(a) - \psi(a)| : a \in \sa{\A}, \Lip(a) \leq 1, \|a\|_\A \leq r \right\}\text{.}
\end{equation*}
The real number $r$ is called a \emph{cut-off} of $\boundedLipschitz{\Lip,r}$.
\end{definition}

It is easy to check that, for any two $r , t > 0$, the bounded-Lipschitz distances $\boundedLipschitz{\Lip,r}$ and $\boundedLipschitz{\Lip,t}$ are (bi-Lipschitz) equivalent \cite[Proposition 2.10]{Latremoliere05b}. In the classical picture given by Example (\ref{fundamental-LP-ex}), the bounded-Lipschitz distance with cut-off $r > 0$, associated to a proper metric space $(X,\mathrm{d})$ restricts to $\min\{\mathrm{d}, r\}$ on the pure states. On the other hand, when working with finite diameter spaces, the bounded-Lipschitz distance agrees with the {\mongekant}, for the cut-off at least as large as the diameter:
\begin{proposition}\label{bounded-Kantorovich-prop}
Let $(\A,\Lip)$ be a Lipschitz pair such that $\diam{\StateSpace(\A)}{\Kantorovich{\Lip}} < \infty$. If $r \geq \diam{\StateSpace(\A)}{\Kantorovich{\Lip}}$ then $\boundedLipschitz{\Lip,r } = \Kantorovich{\Lip}$.
\end{proposition}

\begin{proof}
Let us fix $\psi\in\StateSpace(\A)$. If $a\in \A$ with $\Lip(a)\leq 1$ and $\varphi\in\StateSpace(\A)$ then:
\begin{equation*}
|\varphi(a) - \psi(a)| \leq \Kantorovich{\Lip}(\varphi,\psi) \leq \diam{\StateSpace(\A)}{\Kantorovich{\Lip}} \leq r\text{.}
\end{equation*}
Thus $\|a-\psi(a)\unit_\A\|_\A \leq r$. 

We note that for all $t\in\R$ and $a\in\dom{\Lip}$:
\begin{equation*}
\Lip(a+t\unit_\A) \leq \Lip(a) + |t|\Lip(\unit_\A) = \Lip(a) \leq \Lip(a+t\unit_\A) + |-t|\Lip(\unit_\A) = \Lip(a+t\unit_\A)\text{,}
\end{equation*}
so $\Lip(a) = \Lip(a+t\unit_\A)$.

We now simply observe that for all $\varphi,\eta\in\StateSpace(\A)$:
\begin{equation*}
\begin{split}
\Kantorovich{\Lip}(&\varphi,\eta) \\
&= \sup \left\{ |\varphi(a) - \eta(a)| : a\in\sa{\A}, \Lip(a)\leq 1 \right\} \\
&=  \sup \left\{ |\varphi(a - \psi(a)\unit_\A ) - \eta(a - \psi(a)\unit_\A )| : a\in\sa{\A}, \Lip(a)\leq 1 \right\} \\
&=  \sup \left\{ |\varphi(a - \psi(a)\unit_\A ) - \eta(a - \psi(a)\unit_\A )| : a\in\sa{\A}, \Lip(a-\psi(a)\unit_\A)\leq 1 \right\} \\
&=  \sup \left\{ |\varphi(a - \psi(a)\unit_\A ) - \eta(a - \psi(a)\unit_\A )| \middle\vert \begin{array}{l}
a\in\sa{\A}\\ \Lip(a-\psi(a)\unit_\A)\leq 1\\ \|a-\psi(a)\unit_\A\|_\A\leq r \end{array}\right\} \\
&\leq \boundedLipschitz{\Lip,r}(\varphi,\eta)\text{.}
\end{split}
\end{equation*} 
Our proof is completed by noting that $\boundedLipschitz{\Lip,r} \leq \Kantorovich{\Lip}$ for all $r > 0$ by definition.
\end{proof}

The main question which we now wish to investigate concerns the topology induced by the {\mongekant} and the bounded-Lipschitz metrics on the state spaces of C*-algebras. The characterization of \emph{unital} Lipschitz pairs for which the {\mongekant} induces the weak* topology is the subject of \cite{Rieffel98a, Rieffel99}, and is the foundation of the theory of compact quantum metric spaces. For any Lipschitz pair, the question of when the bounded-Lipschitz distance metrizes the weak* topology is the subject of our work in \cite{Latremoliere05b}, which started the study of quantum locally compact metric spaces. We then continued this study with our work in \cite{Latremoliere12b} for the {\mongekant} of non-unital Lipschitz pair. The non-unital problem raises quite a few many interesting challenges, and of course generalize the earlier work of Rieffel, at least within the context of C*-algebras. These matters are addressed in the next few sections. 

Prior to presenting our work on the topological properties of the metrics associated with a Lipschitz pair, however, we discuss two important desirable properties of the classical Lipschitz seminorms which play a central role in our work with noncommutative analogues of the Gromov-Hausdorff distance. The first property is lower-semicontinuity, which makes the notion of morphisms between Lipschitz pairs easier to work with. The second property is the Leibniz property.

\subsubsection{Morphisms of Lipschitz pairs}

There is a natural notion of morphism between Lipschitz pairs, thus defining a category of Lipschitz pairs. The most relevant aspect of this discussion is the notion of isometry between quantum compact metric spaces. We will see that lower semicontinuity of $\Lip$ in a Lipschitz pair $(\A,\Lip)$ is a property tightly connected to the notion of morphisms for Lipschitz pairs.

A \emph{proper} *-morphism between two C*-algebras is a *-morphism which maps approximate units to approximate units. In particular, it is unital if its domain is unital. 

A natural notion of a proper Lipschitz map between quantum compact metric spaces is given by:
\begin{definition}\label{Lipschitz-map-def}
Let $(\A,\Lip_\A)$ and $(\B,\Lip_\B)$ be two Lipschitz pairs. A proper *-morphism $\pi : \A\rightarrow\B$ is \emph{$k$-Lipschitz} for some $k \geq 0$ if the dual map:
\begin{equation*}
\varphi \in \StateSpace(\B) \longmapsto \varphi\circ\pi \in \StateSpace(\A)
\end{equation*}
is $k$-Lipschitz from $(\StateSpace(\B),\Kantorovich{\Lip_\B})$ to $(\StateSpace(\A),\Kantorovich{\Lip_\A})$.

A proper *-morphism $\pi : \A\rightarrow\B$ is called \emph{Lipschitz} when it is $k$-Lipschitz for some $k\geq 0$.
\end{definition}

It is simple to check that taking as objects the Lipschitz pairs, and taking as morphisms proper Lipschitz *-morphisms defines a category. The isomorphism in this category would be given by bi-Lipschitz *-isomorphisms. 

However, as with the category of metric spaces, we will also wish to work with a stronger type of morphisms: isometries. 

McShane's Theorem \cite{McShane34} states that if $X$ is a nonempty subset of a metric space $(Z,\mathrm{d})$ and if $f : X\rightarrow\R$ is a function with Lipschitz seminorm $k \in [0,\infty)$ then there exists a function $g : Z\rightarrow\R$ whose restriction to $X$ is $f$ and with Lipschitz seminorm $k$. In other words, the Lipschitz seminorm on $C_0(X)$ is the quotient of the Lipschitz seminorm on $C_0(Z)$ when $(Z,\mathrm{d})$ is locally compact. More generally, a map $\pi : X\hookrightarrow Z$ between locally compact metric spaces is an isometry if and only if the quotient of Lipschitz seminorm on  $C_0(Z)$ by $\pi$ is the Lipschitz seminorm on $C_0(X)$ by $f \in C_0(Z) \mapsto f\circ\pi \in C_0(X)$ --- which is well-defined since isometries are always proper maps.

Thus, we introduce:
\begin{definition}[\cite{Rieffel00}]\label{isometry-def}
Let $(\A,\Lip_\A)$ and $(\B,\Lip_\B)$ be two Lipschitz pairs. A proper *-morphisms $\pi:\B\twoheadrightarrow\A$ is an \emph{isometric *-epimorphism} from $(\B,\Lip_\B)$ onto $(\A,\Lip_\A)$ when $\pi$ is a *-epimorphism and for all $a\in\sa{\A}$, we have:
\begin{equation*}
\Lip_\A(a) = \inf\left\{ \Lip_\B(b) : b\in\sa{\B}, \pi(b) = a\right\}\text{.}
\end{equation*}
\end{definition}

If $\pi : \A\rightarrow\B$ is an isometric *-epimorphism, where $(\A,\Lip_\A)$ and $(\B,\Lip_\B)$ are Lipschitz pairs, then $\varphi\in\StateSpace(\B) \mapsto \varphi\circ\pi$ is indeed an isometry \cite[Proposition 3.1]{Rieffel00}. Moreover, the composition of isometric *-epimorphisms is again an isometric *-epimorphism \cite[Proposition 3.7]{Rieffel00}. Thus, we have a subcategory of Lipschitz pairs, with morphisms given as isometries.

For this document, this subcategory will be our framework.

Of central importance to us are isometric *-isomorphisms:
\begin{definition}[\cite{Rieffel00}]\label{isometry-isomorphism-def}
Let $(\A,\Lip_\A)$ and $(\B,\Lip_\B)$ be two Lipschitz pairs. A *-iso\-mor\-phism $\pi : \A\rightarrow\B$ is an \emph{isometric *-iso\-mor\-phism} when for all $\varphi, \psi \in \StateSpace(\B)$:
\begin{equation*}
\Kantorovich{\Lip_\A}(\varphi\circ\pi, \psi\circ\pi) = \Kantorovich{\Lip_\B}(\varphi,\psi)\text{.}
\end{equation*}
\end{definition}

We pause for a remark about non-unital Lipschitz pairs. In general, the notion of morphisms between non-unital C*-algebras require some care: for instance, if $\pi : X\rightarrow Y$ is a continuous functions between two locally compact Hausdorff spaces $X$ and $Y$, then given $f \in C_0(Y)$, we may only expect that $f\circ\pi$ belongs to the multiplier $C_b(X)$ of $C_0(X)$, i.e. the C*-algebra of bounded continuous functions on $X$. Thus a common choice of definition for a morphism from a C*-algebra $\A$ to a C*-algebra $\B$ is a *-morphism from $\A$ to the multiplier C*-algebra of $\B$.

However, if $\pi$ is actually proper, then $f\in C_0(Y) \mapsto f\circ\pi \in C_0(X)$ is actually well-defined, and a proper *-morphism. Conversely, a proper *-morphism from $C_0(Y)$ to $C_0(X)$ is always of that form. For our purpose, since the {\mongekant} of a Lipschitz pair is defined on the state space of $\A$ and not its multiplier algebra, we prefer to limit ourselves to working with proper *-morphisms.

In particular, we note that a surjective isometry between metric spaces is always proper, and a *-isomorphism is always a proper *-morphism, thus for our key notion, this choice is not a source of concern.

\bigskip

The state space of a C*-algebra is a rather intricate world \cite{Alfsen01}, so it is desirable to express notions such as Lipschitz morphisms in terms of the immediate data provided by the Lipschitz pair. To this end, a natural question is: to what degree is the Lip-norm of a quantum compact metric space determined by its associated {\mongekant} ?

Let $(\A,\Lip)$ be a Lipschitz pair. We may define on $\sa{\A}$ a new seminorm $\Lip'$ (possibly taking the value $\infty$) by setting for all $a\in\sa{\A}$:
\begin{equation*}
\Lip'(a) = \sup\left\{ \frac{|\varphi(a)-\psi(a)|}{\Kantorovich{\Lip}(\varphi,\psi)} :  \varphi, \psi\in \StateSpace(\A), \varphi\not=\psi  \right\} \text{.}
\end{equation*}
While we always have $\Lip' \leq \Lip$, equality does not hold in general \cite[Example 3.5]{Rieffel99}. A particular observation is that, for all pair $\varphi , \psi \in \StateSpace(\A)$ of states, the map $a\in\sa{\A} \mapsto \frac{|\varphi(a)-\psi(a)|}{\Kantorovich{\Lip}(\varphi,\psi)}$ is continuous. Thus, as the supremum of continuous functions, $\Lip'$ is a lower semicontinuous function over $\sa{\A}$. Therefore, equality between $\Lip$ and $\Lip'$ may only occur if at least, $\Lip$ is lower semicontinuous on $\sa{\A}$. Rieffel showed in \cite{Rieffel99} that this necessary condition is also sufficient:
\begin{theorem}[\cite{Rieffel99}, Theorem 4.1]\label{lower-semicontinuous-thm}
Let $(\A,\Lip)$ be a unital Lipschitz pair. The following assertions are equivalent:
\begin{enumerate}
\item for all $a\in\sa{\A}$ we have:
\begin{equation*}
\Lip(a) = \sup\left\{ \frac{|\varphi(a)-\psi(a)|}{\Kantorovich{\Lip}(\varphi,\psi)} : \varphi, \psi\in \StateSpace(\A), \varphi\not=\psi  \right\}
\end{equation*}
\item $\Lip$ is lower semicontinuous,
\item $\{a\in\sa{\A} : \Lip(a) \leq 1 \}$ is closed in $\A$.
\end{enumerate}
\end{theorem}
We note that \cite[Theorem 4.1]{Rieffel99} is phrased for unital Lipschitz pairs, but its proof extends unchanged to general Lipschitz pairs.

The equivalence between the second and third assertion of Theorem (\ref{lower-semicontinuous-thm}) is immediate since $\Lip$ is a seminorm --- so it is positively homogeneous. 

Another observation in \cite[Proposition 4.4]{Rieffel99} is that, given a Lipschitz seminorm $(\A,\Lip)$ on a unital C*-algebra $\A$, one may always construct a lower semicontinuous seminorm $\Lip^c$ on some dense subset of $\sa{\A}$ containing the domain of $\Lip$ and such that $\Kantorovich{\Lip} = \Kantorovich{\Lip^c}$. Henceforth we will work with lower semi-continuous Lipschitz seminorms.

Now, under the assumption of lower semicontinuity for Lip-norms, it becomes possible to express the notions of Lipschitz *-morphisms and isometric *-mor\-ph\-isms in terms of Lip-norms.

\begin{theorem}
Let $(\A,\Lip_\A)$ and $(\B,\Lip_\B)$ be two quantum compact metric spaces with $\Lip_\A$ and $\Lip_\B$ lower semicontinuous.
\begin{enumerate}
\item A proper *-morphism $\pi : \A\rightarrow\B$ is  $k$-Lipschitz for some $k\geq 0$ if and only if for all $a\in\sa{\A}$ we have $\Lip_\B\circ\pi(a) \leq k \Lip_\A(a)$ for all $a\in\dom{\Lip_\A}$. 
\item A *-isomorphism $\pi$ is an isometric *-isomorphism if and only if:
\begin{equation*}
\Lip_\B\circ\pi(a) = \Lip_\A(a)
\end{equation*}
for all $a\in\dom{\Lip_\A}$.
\end{enumerate}
\end{theorem}

All of the Lipschitz pairs obtained via Examples (\ref{Connes-LP-ex}) and (\ref{ergodic-LP-ex}), provide lower semicontinuous seminorms. 

We conclude with another important subtlety, exposed in \cite{Rieffel99}. Let $(\A,\Lip)$ be a Lipschitz pair. Then $\Kantorovich{\Lip}$ induces a metric on the pure states of $\A$. This metric, in turn, can be used to defined a new Lipschitz pair $(\A,\Lip_e)$ by setting:
\begin{equation*}
\Lip_e : a\in \sa{\A} \longmapsto \sup\left\{ \frac{|\varphi(a) - \psi(a)|}{\Kantorovich{\Lip}(\varphi,\psi)} : \varphi, \psi \text{ distinct pure states } \right\}\text{.}
\end{equation*}
The natural motivation behind the definition of $\Lip_e$, of course, is that if $\A$ is Abelian, then $\Lip_e$ is the usual Lipschitz seminorm associated to the restriction of $\Kantorovich{\Lip}$ to the Gel'fand spectrum of $\A$. Now, as seen for instance in \cite[Example 7.1]{Rieffel99}, we do not have in general that $\Lip = \Lip_e$.

In general, Rieffel characterized Lipschitz seminorms in the classical picture in \cite[Theorem 8.1]{Rieffel99}. This characterization uses the underlying order on the self-adjoint elements, which is not quite as well behaved a tool in the noncommutative setting as in the commutative setting. It should be noted that even if $(\A,\Lip)$ satisfies all the properties we will enounce in Definition (\ref{LQCMS-def}), and if $\A$ is Abelian, it is still not true in general that $\Lip$ is in fact the Lipschitz seminorm for the restriction of $\Kantorovich{\Lip}$ to the Gel'fand spectrum of $\A$. In essence, we will keep the properties of the Lipschitz seminorms from the classical picture which we deem useful for our purpose, rather than try to retain a characterization of these seminorms.

\bigskip

Now, another property of Lipschitz seminorms is the Leibniz inequality, which connects them with the underlying multiplication of functions. It is this property, rather than any connection with the order of the self-adjoint part of C*-algebras, which we will retain for our noncommutative theory. In some fundamental sense, since two C*-algebras may be Jordan isomorphic without being *-isomorphic, it is more natural to connect the multiplicative structure and the quantum metric structure, rather than the order structure.

\subsubsection{The Leibniz property}

The Lipschitz seminorm $\mathrm{Lip}$ defined by a metric space $(X,\mathrm{d})$ via Expression (\ref{Lipschitz-seminorm-eq}) possesses a natural property with respect to the multiplicative structure of $C_0(X)$, namely:
\begin{equation*}
\mathrm{Lip}(f g) \leq \mathrm{Lip}(f) \|g\|_{C_0(X)} + \|f\|_{C_0(X)} \mathrm{Lip}(g)
\end{equation*}
for all $f, g \in C_0(X)$. This inequality is known as the Leibniz property --- though we will use this term in a slightly more general context. 

This property does not play any known role in the topological properties of the {\mongekant}: the characterizations of Lipschitz pairs for which the {\mongekant}, or the bounded-Lipschitz metrics, induces the weak* topology, do not depend on the Leibniz property. Thus, the Leibniz property was not a part of the original work of Rieffel, or our own earlier work \cite{Latremoliere05b, Latremoliere12b}, and in particular, not a part of Definition (\ref{Lipschitz-pair-def}). However, as research in noncommutative metric geometry progressed, the need for a property of the Leibniz type became evident. One reason is to allow for computations in work related to convergence of modules under Gromov-Hausdorff convergence \cite{Rieffel09, Rieffel10, Rieffel10c, Rieffel15}. Another reason is to address the coincidence property for noncommutative analogues of the Gromov-Hausdorff distance, which was one of our key contribution \cite{Latremoliere13,Latremoliere13b,Latremoliere13c,Latremoliere14}. Remarkably, requiring the Leibniz property, properly defined, for Lipschitz pairs, raises some difficulties. The main source of these difficulties is that the quotient of a Leibniz seminorm may not be Leibniz \cite{Blackadar91}. Yet, central notions, such as isometric *-epimorphisms, rely on the notion of quotient of seminorms.

Our own research \cite{Latremoliere15} led us to allow for more general forms of the Leibniz property. Thus, one will be able to choose a specific variant of the Leibniz identity adapted to whatever one's purpose might be, and then use the appropriate form of the dual Gromov-Hausdorff propinquity. We shall see examples of this approach in our section on compactness for the dual Gromov-Hausdorff propinquity and our section on conformal deformations. 

The first generalization of the Leibniz property for our use occurred in \cite{Latremoliere13}: since the seminorms of Lipschitz pairs are only defined on some dense subsets of the self-adjoint part of C*-algebras, and since the product of two self-adjoint elements is generally not self-adjoint, we replaced the product by the Jordan and the Lie product. As a second step in a subsequent work \cite{Latremoliere15}, motivated by our compactness theorem for the dual Gromov-Hausdorff propinquity, we adapted the notion of an $F$-Leibniz seminorm from Kerr \cite{kerr02} to our Jordan-Lie setting. 

Our current formulation of the Leibniz property for Lipschitz pairs is thus given as follows.

\begin{definition}[\cite{Latremoliere15}, Definition 2.4]\label{permissible-def}
We endow $[0,\infty]^4$ with the following order:
\begin{multline*}
\forall x=(x_1,x_2,x_3,x_4), y=(y_1,y_2,y_3,y_4) \\
x\preccurlyeq y \iff \left(\forall j\in\{1,\ldots,4\} \quad x_j\leq y_j\right)\text{.}
\end{multline*}
A function $F : [0,\infty]^4\rightarrow [0,\infty]$ is \emph{permissible} when:
\begin{enumerate}
\item $F$ is non-decreasing from $([0,\infty)^4,\preccurlyeq)$ to $([0,\infty),\leq)$,
\item for all $x,y,l_x,l_y \in [0,\infty)$ we have:
\begin{equation}\label{permissible-eq}
x l_y + y l_x \leq F(x,y,l_x,l_y)\text{.}
\end{equation}
\end{enumerate}
\end{definition}

\begin{notation}
Let $\A$ be a C*-algebra and $a,b \in \A$. The \emph{Jordan product} $\frac{ab+ba}{2}$ of $a,b$ is denoted by $\Jordan{a}{b}$, while the \emph{Lie product} $\frac{ab-ba}{2i}$ of $a,b$ is denoted by $\Lie{a}{b}$. For any $a,b \in \sa{\A}$ we have $\Jordan{a}{b}, \Lie{a}{b} \in \sa{\A}$ so that $\sa{\A}$ is a Jordan-Lie subalgebra of $\A$.
\end{notation}

\begin{definition}[\cite{Latremoliere15}, Definition 2.5]\label{quasi-Leibniz-def}
Let $F$ be a permissible function. A \emph{$F$-quasi-Leibniz pair} $(\A,\Lip)$ is a Lipschitz pair such that:
\begin{enumerate}
\item the domain $\dom{\Lip}$ of $\Lip$ is a dense Jordan-Lie subalgebra of $\sa{\A}$,
\item for all $a,b \in \dom{\Lip}$, we have:
\begin{equation*}
\Lip\left(\Jordan{a}{b}\right) \leq F\left(\|a\|_\A,\|b\|_\A,\Lip(a),\Lip(b)\right)
\end{equation*}
and
\begin{equation*}
\Lip\left(\Lie{a}{b}\right) \leq F\left(\|a\|_\A,\|b\|_\A,\Lip(a),\Lip(b)\right)\text{.}
\end{equation*}
\end{enumerate}
\end{definition}

\begin{definition}[\cite{Latremoliere13}, Definition 2.15]
A \emph{Leibniz pair} $(\A,\Lip)$ is a Lipschitz pair such that:
\begin{enumerate}
\item the domain $\dom{\Lip}$ of $\Lip$ is a Jordan-Lie subalgebra of $\sa{\A}$,
\item for all $a,b \in \dom{\Lip}$, we have:
\begin{equation*}
\Lip\left(\Jordan{a}{b}\right) \leq \|a\|_\A\Lip(b) + \Lip(a)\|b\|_\A
\end{equation*}
and
\begin{equation*}
\Lip\left(\Lie{a}{b}\right) \leq \|a\|_\A\Lip(b) + \Lip(a)\|b\|_\A\text{.}
\end{equation*}
\end{enumerate}
\end{definition}

\begin{remark}
A Leibniz pair is a $F$-quasi Leibniz pair for $F : x,y,l_x,l_y \in [0,\infty) \mapsto x l_y + y l_x$.
\end{remark}

\begin{remark}
Informally, Definition (\ref{permissible-def}) includes the condition that, given a Leibniz pair, the upper bound for the seminorm of a Jordan or Lie product is no worse than the bound given by the Leibniz inequality. The reason for this requirement will become apparent when we discuss the notion of composition of tunnels for the extent-based construction of the dual Gromov-Hausdorff propinquity in the next chapter.
\end{remark}

It is common, albeit not necessary, that Leibniz seminorms $\mathrm{S}$ are defined on some dense *-subalgebra of a C*-algebra and satisfy:
\begin{equation*}
\mathrm{S}(ab) \leq \|a\|_\A \mathrm{S}(b) + \mathrm{S}(a)\|b\|_\A
\end{equation*}
for all $a,b$ in the domain of $\mathrm{S}$; indeed if a seminorm satisfies these properties then its restriction to the self-adjoint space is Leibniz in our sense \cite[Proposition 2.17]{Latremoliere13}.

We note that Examples (\ref{fundamental-LP-ex}), (\ref{Connes-LP-ex}), (\ref{ergodic-LP-ex}), and (\ref{Moyal-LP-ex}), all provide Leibniz pairs. Examples of quasi-Leibniz pairs will occur for finite dimensional approximations of a large class of Leibniz quantum compact metric spaces in \cite{Latremoliere15}. There are some surprising sources of Leibniz pairs.

\begin{example}[Standard Deviation, \cite{Rieffel12}]
Let $\A$ be a C*-algebra, $\mu\in\StateSpace(\A)$. The \emph{standard deviation} of $a\in\sa{\A}$ under the law $\mu$ is computed as:
\begin{equation*}
\mathrm{stddev}_\mu(a) = \sqrt[2]{\mu(a^\ast a) - |\mu(a)|^2}\text{.}
\end{equation*}
Notably, if $\pi$ be the GNS representation of $\A$ from $\mu$ on the Hilbert space $L^2(\A,\mu)$ obtained by completing $\A$ for the norm associated to the inner product $a,b\in\A\mapsto\mu(a^\ast b)$, then $\mathrm{stddev}_\mu(a) = \|\pi(a)-\mu(a)\unit\|_{L^2(\A,\mu)}$ for all $a\in\sa{\A}$.

For non-self-adjoint elements, we propose to extend the definition of the standard deviation as such: if $a\in\A$ then:
\begin{equation*}
\mathrm{stddev}(a) = \max\{ \|a-\mu(a)\|_{L^2(\A,\mu)}, \|a^\ast-\mu(a^\ast)\|_{L^2(\A,\mu)} \}\text{.}
\end{equation*}

Rieffel proved in \cite[Theorem 3.5]{Rieffel12} that for all $a,b \in \sa{\A}$ we have:
\begin{equation*}
\mathrm{stddev}_\mu(ab) \leq \|a\|_\A \mathrm{stddev}_\mu(b) + \mathrm{stddev}_\mu(a)\|b\|_\A\text{.}
\end{equation*}
In particular, if $\mu$ is faithful, then $(\A,\mathrm{stddev}_\mu)$ is a Leibniz pair.

In \cite[Theorem 3.7]{Rieffel12}, it is shown that $\mathrm{stddev}_\mu$ is in fact strong Leibniz.
\end{example}

Another example of a Leibniz seminorm arises from an intriguing construction, though it does not provide a Leibniz pair in general --- we include it as it is nonetheless interesting:

\begin{example}[Quotients norms, \cite{Rieffel11}]
Let $\A$ be a C*-algebra and $\B$ a C*-subalgebra of $\A$ which contains an approximate unit for $\A$. For all $a\in\A$ we define:
\begin{equation*}
\Lip(a) = \inf\{ \|a-b\|_\A : b\in \B \}\text{.}
\end{equation*}
Then Rieffel showed in \cite{Rieffel11} that for all $a,b \in \A$ we have $\Lip(ab) \leq \|a\|_\A\Lip(b) + \|b\|_\A\Lip(a)$. Of course, $(\A,\Lip)$ is not a Lipschitz pair unless $\B = \C\unit_\A$, but we do obtain a Leibniz seminorm --- in fact, a strong Leibniz seminorm.
\end{example}

Another source of quasi-Leibniz seminorms is given by twisted differential calculi.

\begin{example}\label{module-LP-ex}
Let $\A$ be a C*-algebra and let $\Omega$ be a $\A$-$\A$ bimodule. Let $\sigma : \A\rightarrow \A$ be a continuous morphism of norm $\nu$ at least $1$ --- though not necessarily a *-morphism, so the norm of $\sigma$ is may be strictly larger than $1$. Let $\B$ be a dense *-subalgebra of $\A$. Let $d : \B \rightarrow \Omega$ be a linear map such that for all $a,b \in B$:
\begin{equation*}
d(ab) = da\cdot b + \sigma(a) db \text{,}
\end{equation*}
while $\ker d = \C\unit_\A$. 

Last, let $\|\cdot\|_\Omega$ be a bimodule norm on $\Omega$: namely for all $a,b \in \A$ and $\omega \in \Omega$ we have $\|a \omega b \|_\Omega \leq \|a\|_\A \|\omega\|_\Omega \|b\|_\A$.

If $\Lip : a \in \B \mapsto \|da\|_\Omega$ then $(\A,\Lip)$ is a $F$-quasi-Leibniz Lipschitz pair where:
\begin{equation*}
F : a,b,l_a,l_b \in [0,\infty) \mapsto \nu a l_b + b l_a \text{.}
\end{equation*}
\end{example}

The construction in Example (\ref{module-LP-ex}) is actually underlying a lot of our previous examples. For instance, if $(X,\mathrm{d})$ is a compact metric space, and if we let $\Omega = C_b(X\times X \setminus \Delta)$ where $\Delta = \{ (x,x) : x\in X\}$, then $\Omega$ is a $C(X)$-$C(X)$-bimodule via the actions:
\begin{equation*}
f\cdot g (x,y) = f(x)g(x,y) \text{ and }g\cdot f (x,y) = g(x,y)f(y)
\end{equation*}
for all $f\in C(X)$, $g\in \Omega$ and $(x,y) \in X^2 \setminus\Delta$. Moreover, the uniform norm $\|\cdot\|_\Omega$ is a bimodule norm on $\Omega$. Last, we may set for all $f\in C(X)$ and $(x,y) \in X^2\setminus\Delta$:
\begin{equation*}
df (x,y) = \frac{f(x)-f(y)}{d(x,y)}\text{.}
\end{equation*}
Then $(C(X),\Omega,d)$ is  a first order differential calculus and, moreover, if $\mathrm{Lip}$ is the Lipschitz seminorm associated with $\mathrm{d}$ defined by Expression (\ref{Lipschitz-seminorm-eq}) then $\mathrm{Lip}(f) = \|df\|_\Omega$ for all $f\in C(X)$.

Similarly, if $(\A,\Hilbert,D)$ is a spectral triple over a unital C*-algebra $\A$, then the C*-algebra of bounded linear operators on $\Hilbert$ is an $\A$-$\A$ bimodule, and $d : a\in \A\mapsto [D,\pi(a)]$ defines a derivation, so that again, the Lipschitz pair constructed in Example (\ref{Connes-LP-ex}) becomes a special case of the construction in Example (\ref{module-LP-ex}).

Allowing for a twist in Example (\ref{module-LP-ex}) permits us to adapt our setting to twisted spectral triples via a similar construction \cite{Connes08, Ponge14}. We will briefly investigate the special case of conformal deformations later on in this document.

We conclude this section with the following simple observation, which justifies that we may as well work with lower semicontinuous Leibniz pairs:
\begin{proposition}[\cite{Latremoliere15}, Lemma 3.1]
If $(\A,\Lip)$ is a Lipschitz pair and $\Lip$ is $F$-quasi-Leibniz for some \emph{continuous} permissible function $F$, and if $\Lip'$ is the closure of $\Lip$, then $(\A,\Lip')$ is also $F$-quasi-Leibniz.
\end{proposition}

\begin{proof}
By definition, $\Lip'$ is the Minkowsky functional for the convex $\overline{\mathfrak{S}}$ where $\mathfrak{S} = \{a\in\sa{\A}:\Lip(a)\leq 1\}$. We now use \cite[Lemma 3.1]{Latremoliere15}. Note that the necessary condition of \cite[Lemma 3.1]{Latremoliere15} does not require that the convex is closed, hence since $\Lip$ is $F$-quasi-Leibniz:
\begin{equation*}
\Jordan{a}{b}, \Lie{a}{b} \in F(\|a\|_\A,\|b\|_\B,1, 1) \mathfrak{S}\text{,}
\end{equation*}
for all $a,b \in \sa{\A}$ with $\Lip(a),\Lip(b)\leq 1$. Now, since $F$ is continuous, we conclude that if $a,b \in \overline{\mathfrak{S}}$ then:
\begin{equation*}
\Jordan{a}{b}, \Lie{a}{b} \in F(\|a\|_\A,\|b\|_\B,1, 1) \overline{\mathfrak{S}}\text{.}
\end{equation*}
We can apply \cite[Lemma 3.1]{Latremoliere15} again to conclude.
\end{proof}

\bigskip

We thus have presented many examples of quasi-Leibniz pairs. We now study the topological properties of the associated metrics on the state space. The first such study is due to Rieffel \cite{Rieffel98a}. We also point out the interesting work of Pavlovic \cite{Pavlovic98}. Both these initial studies were carried out in the case of unital Lipschitz pairs.

Our own research begun with the study of the bounded-Lipschitz metrics, motivated by the challenge of working with non-unital Lipschitz pairs. We thus begin with our study of bounded-Lipschitz metrics and explore what is the fundamental property we shall require of Lipschitz pairs to be considered quantum metric spaces.

\subsection{Bounded-Lipschitz Metrics}

\subsubsection{Metrizing the weak* topology}

The bounded-Lipschitz metric associated with a Lipschitz pair $(\A,\Lip)$ is, indeed, a metric, and it is easy to check that it always endow the state space $\StateSpace(\A)$ with a finer topology than the weak* topology and a coarser topology than the norm topology.

There are several reasons why the weak* topology is, indeed, the topology which one desires the bounded-Lipschitz distance to metrize. To begin with, the classical model suggests that the a core property of Lipschitz seminorms is precisely that the associated bounded-Lipschitz distances are metrics for the weak* topology on Borel regular probability measure. The importance of the weak* topology in this context need little advertisement: it is of course the proper notion for convergence in law, a central concept of probability theory, illustrated for example by its role in the central limit theorems.

In the case of unital Lipschitz pairs, it is also quite natural to desire that the state space be compact for the topology given by the bounded-Lipschitz distance. Since the topology of a metric is Hausdorff, and since the bounded-Lipschitz topology is finer than the weak* topology, this desirable feature implies that the bounded-Lipschitz metric must metrize the weak* topology.

Moreover, the restriction of the weak* topology on the pure state space is the topology chosen for the Gel'fand spectrum of Abelian C*-algebras. Thus, when working in noncommutative geometry, one could consider that the requirement for the bounded-Lipschitz metric to metrize the weak* topology at least on the pure state space is a natural leftover from the classical picture. It is but a small step to then consider that the bounded-Lipschitz distance should metrize the weak* topology of the entire state space. 

Last, there is a natural physical motivation to work with the weak* topology on the state space: it is, after all, the topology of pointwise convergence for states; as such it is physically natural. Indeed, one obtain a base of neighborhoods for this topology by, physically speaking, taking all states which agree, within some given error $\varepsilon>0$, on some \emph{finite} set of observables.

Motivated by all these considerations, our purpose is thus to determine when the bounded-Lipschitz distance metrizes the weak* topology restricted to the state space.

A rather pleasant picture emerged from our study in \cite{Latremoliere05b}:

\begin{theorem}[\cite{Latremoliere05b}, Theorem 4.1]\label{bounded-Lipschitz-thm}
Let $(\A,\Lip)$ be a Lipschitz pair where $\A$ is separable. The bounded-Lipschitz distance on the state space $\StateSpace(\A)$, as per Definition (\ref{bounded-Lipschitz-def}), is given for any two $\varphi, \psi \in \StateSpace(\A)$, by:
\begin{equation*}
\boundedLipschitz{\Lip,1}(\varphi, \psi) = \sup\left\{ |\varphi(a) - \psi(a)| : a\in\sa{\A}, \Lip(a)\leq 1, \|a\|_\A \leq 1 \right\}\text{.}
\end{equation*}

The following assertions are equivalent:
\begin{enumerate}
\item the bounded-Lipschitz distance $\boundedLipschitz{\Lip,1}$ metrizes the restriction of the weak* topology to $\StateSpace(\A)$,
\item there exists a completely positive element $h \in \sa{\A}$ such that the set:
\begin{equation*}
\left\{ hah : a\in\sa{\A}, \Lip(a)\leq 1, \|a\|_\A\leq 1 \right\}
\end{equation*}
is precompact for $\|\cdot\|_\A$,
\item for all completely positive $h\in \sa{\A}$, the set:
\begin{equation*}
\left\{ hah : a\in\sa{\A}, \Lip(a)\leq 1, \|a\|_\A\leq 1 \right\}
\end{equation*}
is precompact for $\|\cdot\|_\A$.
\end{enumerate}
\end{theorem}

We note that since all the bounded-Lipschitz metrics associated to a Lipschitz pair, for various cut-off values, are equivalent, Theorem (\ref{bounded-Lipschitz-thm}) applies to any such metrics as well.

By Proposition (\ref{bounded-Kantorovich-prop}), if $(\A,\Lip)$ is a Lipschitz pair such that $\diam{\StateSpace(\A)}{\Kantorovich{\Lip}}<\infty$, then the {\mongekant} $\Kantorovich{\Lip}$ and some bounded-Lipschitz distance agree. Thus, there is no ambiguity in which of the various natural metrics associated with $(\A,\Lip)$ to choose when working in this context, and this led us to define:

\begin{definition}[\cite{Latremoliere05b}]
A Lipschitz pair $(\A,\Lip)$ is a \emph{bounded quantum locally compact metric space} when:
\begin{enumerate}
\item $\diam{\StateSpace(\A)}{\Kantorovich{\Lip}} < \infty$,
\item the bounded-Lipschitz metric $\boundedLipschitz{\Lip,1}$ defined by $(\A,\Lip)$ metrizes the weak* topology on $\StateSpace(\A)$.
\end{enumerate}
\end{definition}

The important special case of unital Lipschitz pairs will be discussed in the section on the {\mongekant} later in this document, where we will see several examples. We include here three simple non-unital examples.

\begin{example}
If $(X,\mathrm{d})$ be a bounded, locally compact metric space. The Leibniz pair $(C(X),\mathrm{Lip})$ of Example (\ref{fundamental-LP-ex}), where $\mathrm{Lip}$ is the Lipschitz seminorm associated with $\mathrm{d}$, is a bounded quantum locally compact metric space.
\end{example}

\begin{example}[\cite{Latremoliere12b}, Section 4]
The pair $(\Moyal{\theta}, \Lip_\theta)$ of the Moyal plane with the seminorm $\Lip_\theta$ constructed in Example (\ref{Moyal-LP-ex}) satisfies the hypothesis of Theorem (\ref{bounded-Lipschitz-thm}), although its diameter for the {\mongekant} is not finite.
\end{example}

\begin{example}[\cite{Latremoliere05b}, Proposition 4.4]
If $(X,\mathds{d})$ is a locally compact metric space and $\alg{K}$ is an algebra of compact operators, then $C(X,\alg{K}) = C(X)\otimes \alg{K}$, endowed with the seminorm:
\begin{equation*}
\Lip : a\in C(X,\alg{K}) \longmapsto \sup\left\{ \frac{\|a(x) - a(y)\|_{\alg{K}}}{\mathrm{d}(x,y)} : x,y \in X, x\not= y \right\},
\end{equation*}
satisfies the assumption of Theorem (\ref{bounded-Lipschitz-thm}). If $(X,\mathrm{d})$ is bounded, then $(C(X,\alg{K}),\Lip)$ is a bounded quantum locally compact metric space.

This example may be adjusted to show that, for instance, many type I C*-crossed-products can be made into quantum locally compact metric spaces --- however, this particular choice of a metric structure is somewhat ad-hoc.
\end{example}

The proof of Theorem (\ref{bounded-Lipschitz-thm}) relies on an interesting new topology defined on C*-algebras, which we introduced in \cite{Latremoliere05b}. Notable among its properties is that this topology is weaker, and often strictly so, than the strict topology, while stronger than the weak topology, and usually strictly so. Definitions of unbounded Fredholm modules and spectral triples in the non-unital setting often involve conditions which borrow from the construction of the strict topology; our own work challenges this idea by proposing a new topology which emerged from metric considerations. We shall present this topology in our next section.

\bigskip

There has been quite a few applications of our work on bounded-Lipschitz distances in the literature. Indeed, until our own work on the {\mongekant} for non-unital Lipschitz pairs, the work in \cite{Latremoliere05b} was the only source to work with non-compact, quantum compact metric spaces. Thus, our work played a role, for instance, in mathematical physics \cite{Cagnache11, DAndrea10, Wallet12, Kellendock12, Hinz13}.

Another intriguing application can be found in the work of Bellissard, Marcolli and Reihani \cite{Bellissard10}, where our work on bounded-Lipschitz metrics is the corner stone for the construction of metrics on C*-crossed-products. The idea of \cite{Bellissard10} is that, for actions on quantum metric spaces which are not quasi-isometric, one should replace the original quantum space by a noncommutative analogue of the metric bundle, in the spirit of Connes and Moscovici's work, where the proper lift of the action will become isometric. This bundle is noncompact in general, hence the need to work with non-unital Lipschitz pairs --- and invoke our results described in this section. A follow-up of \cite{Bellissard10} using our bounded-Lipschitz metric can be found in \cite{Paterson14}.

\subsubsection{The weakly uniform topology on C*-algebras}

Theorem (\ref{bounded-Lipschitz-thm}) expresses that the bounded-Lipschitz metric distance for a Lipschitz pair $(\A,\Lip)$ metrizes the weak* topology on the state space $\StateSpace(\A)$ if and only if the unit ball for $\max\{\|\cdot\|_\A,\Lip\}$ is totally bounded for some topology, which, as it turns out, is metrizable on bounded subsets of $\A$. We now present this topology in greater detail and some consequences of its study, following \cite{Latremoliere05b}.

We shall need the following notation:

\begin{notation}
Let $\A$ be a C*-algebra. The class of all weak* compact subsets of $\StateSpace(\A)$ is denoted by $\wcs{\A}$.
\end{notation}

\begin{definition}[\cite{Latremoliere05b}, Definition 2.5]\label{wu}
The \emph{weakly uniform topology} $\mathrm{wu}$ on a C*-algebra $\A$ is the locally convex topology generated by the family of seminorms $(p_K)_{K\in\wcs{\A}}$, where for all $K\in \wcs{\A}$:
\begin{equation*}
\forall a \in \A \quad p_K(a) = \sup\{ |\varphi(a)| : \varphi\in K\} \text{.}
\end{equation*}
\end{definition}

The weakly uniform topology compares to the usual topologies on C*-algebras:
\begin{theorem}[\cite{Latremoliere05b}, Lemma 3.2, Proposition 3.3]
If $\A$ is a separable C*-algebra, then for all bounded $\B \subseteq \A$, the weakly uniform topology restricted to $\B$ is coarser than the strict topology restricted to $\B$ and finer than the weak topology restricted to $\B$.

If $\A$ is unital, the weakly uniform topology agrees with the norm topology (since $\StateSpace(\A)$ is then weak* compact).
\end{theorem}

In particular, the weakly uniform topology is Hausdorff.

We pause to mention that while we introduced the weakly uniform topology $\mathrm{wu}$ for the study of the topological properties of the bounded-Lipschitz metric for Lipschitz pairs, $\mathrm{wu}$ itself is defined for any C*-algebra regardless of any Lipschitz seminorm. We will see that the situation is somewhat similar for the study of the {\mongekant}.

Now, by means of the Arz{\'e}la-Ascoli Theorem and Kadisson functional representation Theorem \cite{Kadisson51}, we were able to show in \cite{Latremoliere05b} that:

\begin{theorem}[\cite{Latremoliere05b}, Theorem 2.6]\label{bounded-Lipschitz-thm2}
Let $(\A,\Lip)$ be a separable Lipschitz pair. The following assertions are equivalent:
\begin{enumerate}
\item the bounded-Lipschitz distance $\boundedLipschitz{\Lip,1}$ metrizes the weak* topology restricted to $\StateSpace(\A)$,
\item the set:
\begin{equation*}
\left\{ a\in\sa{\A} : \Lip(a)\leq 1, \|a\|_\A\leq 1 \right\}
\end{equation*}
is totally bounded in the weakly uniform topology.
\end{enumerate}
\end{theorem}

Theorem (\ref{bounded-Lipschitz-thm2}) contains the important observation that our topology is indeed the proper one to consider in the study of the metric properties of the bounded-Lipschitz distance; yet the weakly uniform topology, as defined, would seem difficult to use, and thus Theorem (\ref{bounded-Lipschitz-thm2}) may seem hard to apply. The next main step of \cite{Latremoliere05b}, which in fact occupies most of that paper, is to study the metrizability property of the weakly uniform topology on bounded subsets of $\A$. We thus obtain:

\begin{theorem}[\cite{Latremoliere05b}, Theorem 3.17]\label{bounded-Lipschitz-thm3}
Let $\A$ be a separable C*-algebra and $\B\subseteq \A$ be a bounded subset of $\A$. The weakly uniform topology on $\A$ restricted to $\B$ is metrizable, and moreover, for any strictly positive element $h\in\sa{\A}$, a metric is given by:
\begin{equation*}
a,b \in \B \longmapsto \|h(b-a)h\|_\A\text{.}
\end{equation*}
\end{theorem}

Now, putting the metrizability property of the weakly uniform topology in Theorem (\ref{bounded-Lipschitz-thm3}) with the characterization of bounded quantum locally compact quantum metric spaces given by Theorem (\ref{bounded-Lipschitz-thm2}), we obtain our Theorem (\ref{bounded-Lipschitz-thm}).

In particular, if $\A$ is unital, then $\unit_\A\in\sa{\A}$ is a strictly positive element of $\A$ and thus we recover that the norm topology and the weakly uniform topology agree on bounded subsets of $\A$ (of course, we observe from the definition that these two topologies agree on all of $\A$). In general, we note (since the square of a strictly positive element is strictly positive):

\begin{corollary}[\cite{Latremoliere05b}, Proposition 3.22]
If $\A$ is a separable C*-algebra and there exists a strictly positive \emph{central} element $h \in \sa{\A}$ then the strict topology of $\A$ and the weakly uniform topology of $\A$ agree on bounded subsets of $\A$.
\end{corollary}

In particular, for Abelian C*-algebras, the weakly uniform topology and the strict topology agree on bounded subsets.

There are however examples which show that the weakly uniform topology is at times strictly coarser than the strict topology, even on bounded subset:

\begin{example}
The strict topology is strictly finer than the weakly uniform topology on the unit ball of the C*-algebra $\mathscr{K}$ of compact operators on a separable, infinite dimensional Hilbert space $\Hilbert$. Indeed, let $(\xi_n)_{n\in\N}$ be a Hilbert basis for $\Hilbert$ and for all $n\in\N$, let $p_n$ be the projection on $\C\xi_n$. Then $h = \sum_{n\in\N} \frac{1}{2^{n+1}} p_n$ is a strictly positive element in $\mathscr{K}$.

Let $t_n = \inner{\cdot}{\xi_0}\xi_n$ for all $n\in\N$, where $\inner{\cdot}{\cdot}$ is the inner product on $\Hilbert$. Then $\|t_n\|_{\mathscr{K}} =  1$ and $\|t_{n+1}-t_{n}\|_{\mathscr{K}} = \sqrt{2}$ for all $n\in\N$. Moreover, $h t_n h = \frac{1}{2^{n+2}}t_n $. Thus $(h t_n h)_{n\in\N}$ converges to $0$ in norm. On the other hand, $t_n = 2 t_n h$ for all $n\in\N$ and thus $(t_n h)_{n\in\N}$ does not converge, i.e. $(t_n)_{n\in\N}$ does not converge for the strict topology.
\end{example}

Thus, the weakly uniform topology was hidden because in both the Abelian world and the compact world, it agrees with two of the standard topologies of C*-algebras. Our work suggests, however, that the weakly uniform topology is more natural to consider in the study of metric properties of noncompact noncommutative geometries, in addition to the strict topology.

In fact, a natural question which arises from our work is to compare the strict topology and the weakly uniform topology by giving a state-space description of the strict topology. We proposed such a description in \cite{Latremoliere05b}. 

We begin by introducing another topology on C*-algebras:

\begin{definition}[\cite{Latremoliere05b}, Definition 3.19]
The \emph{strongly uniform topology} on a C*-algebra $\A$ is the locally convex topology generated by the family of seminorms
\begin{equation*}
(q_K)_{K\in\wcs{\A}}\text{,}
\end{equation*}
where for all $K\in\wcs{\A}$:
\begin{equation*}
\forall a \in \A \quad q_K(a) = \sup \left\{ \sqrt[2]{\varphi(a^\ast a)}, \sqrt[2]{\varphi(a a^\ast)} : \varphi\in K \right\} \text{.}
\end{equation*}
\end{definition}

We then have our state-space description of the strict topology in terms of the strongly uniform topology:

\begin{theorem}[\cite{Latremoliere05b}, Theorem 3.21]
Let $\A$ be a separable C*-algebra and $\B\subseteq \A$ be a bounded subset of $\A$. The strongly uniform topology on $\A$ restricted to $\B$ is metrizable, and moreover, for any strictly positive element $h\in\sa{\A}$, a metric is given by:
\begin{equation*}
a,b \in \B \longmapsto \max\{ \|h(b-a)\|_\A, \|(b-a)h\|_\A \} \text{.}
\end{equation*}
In particular, the strongly uniform topology and the strict topology agree on bounded subsets of $\A$.
\end{theorem}

When quantum metric spaces are not, in a natural manner, of bounded diameter, it is natural to wonder what can be said of the behavior of the {\mongekant}. This matter will occupy most of our next section. We however begin this next section with the compact case, which was understood by Rieffel in \cite{Rieffel98a} and to which Theorem (\ref{bounded-Lipschitz-thm}) applies as well.

\subsection{The Monge-Kantorovich Distance}

\subsubsection{Quantum Compact Metric Spaces}\label{compact-sec}

The notion of a quantum compact metric space is the foundation of noncommutative metric geometry, with its origins in \cite{Connes89} and its formalization in \cite{Rieffel98a,Rieffel99}. For our purpose, we shall focus on the C*-algebraic theory. However, it should be noted that Rieffel's definition and framework \cite{Rieffel98a,Rieffel99, Rieffel00} is more general, and involves order-unit spaces in place of C*-algebras. 

As we discussed in the section on the bounded-Lipschitz distance, the core property of Lipschitz seminorms which we keep in the noncommutative world is that the associated {\mongekant} metrizes the weak* topology on the state space. Thus, Rieffel proposed \cite{Rieffel98a}:

\begin{definition}[\cite{Rieffel98a}]\label{Rieffel-def}
A \emph{quantum compact metric space} $(\A,\Lip)$ is a unital Lipschitz pair whose {\mongekant}:
\begin{equation*}
\varphi,\psi \in \StateSpace(\A) \longmapsto \Kantorovich{\Lip}(\varphi,\psi) = \sup \left\{ |\varphi(a)-\psi(a)| : a\in\sa{\A},\Lip(a)\leq 1 \right\}
\end{equation*}
metrizes the weak* topology restricted to the state space $\StateSpace(\A)$ of $\A$.

When a Lipschitz pair $(\A,\Lip)$ is a quantum compact metric space, the seminorm $\Lip$ is referred to as a \emph{Lip-norm}.
\end{definition}

In \cite{Rieffel98a}, Rieffel characterized Lip-norms in term of the total-boundedness of their unit ball modulo scalars. Rieffel proposed several formulation of this theorem later on \cite{Rieffel99, Ozawa05}. The following summarizes his characterization of Lip-norms:

\begin{theorem}[\cite{Rieffel98a}, Theorem 1.9, and \cite{Rieffel99,Ozawa05}]\label{Rieffel-thm}
Let $(\A,\Lip)$ be a Lipschitz pair. The following assertions are equivalent:
\begin{enumerate}
\item $(\A,\Lip)$ is a quantum compact metric space,
\item the set:
\begin{equation*}
\{ \dot a : a \in \sa{\A}, \Lip(a) \leq 1 \}
\end{equation*}
is norm precompact in $\left(\bigslant{\sa{\A}}{\R\unit_\A}\right)$, where $a\in\sa{\A} \mapsto {\dot a}$ is the canonical surjection from $\sa{\A}$ onto $\left(\bigslant{\sa{\A}}{\R\unit_\A}\right)$,
\item there exists a state $\mu\in\StateSpace(\A)$ such that the set:
\begin{equation*}
\left\{ a\in\sa{\A} : \Lip(a)\leq 1, \mu(a) = 0 \right\}
\end{equation*}
is norm precompact in $\A$,
\item for all states $\mu\in\StateSpace(\A)$, the set:
\begin{equation*}
\left\{ a\in\sa{\A} : \Lip(a)\leq 1, \mu(a) = 0 \right\}
\end{equation*}
is norm precompact in $\A$,
\item the set:
\begin{equation*}
\left\{ a\in\sa{\A} : \Lip(a)\leq 1, \|a\|_\A \leq 1 \right\}
\end{equation*}
is norm precompact in $\A$ and $\diam{\StateSpace(\A)}{\Kantorovich{\Lip}} < \infty$.
\end{enumerate}
\end{theorem}

Now, as discussed in Proposition (\ref{bounded-Kantorovich-prop}), the bounded-Lipschitz distances and the {\mongekant} agree when the state space has finite diameter for the {\mongekant}. Thus Theorem (\ref{bounded-Lipschitz-thm}) implies the equivalence between (1) and (5) in Theorem (\ref{Rieffel-thm}). The other equivalences can then be recovered fairly quickly. Consequently, our work on the bounded-Lipschitz distance did extend the work of Rieffel from the unital to the general case of Lipschitz pairs.

Definition (\ref{Rieffel-def}) does not require that Lip-norms be lower semi-continuous and quasi-Leibniz. These two additional assumptions, as we have discussed, are useful to our work (we note that lower semi-continuity is a convenience while the quasi-Leibniz property will prove crucial). We can now define the objects which will be of central interest to us:

\begin{definition}[\cite{Latremoliere13,Latremoliere15}]\label{LQCMS-def}
Let $F$ be an permissible function from $[0,\infty)^4\rightarrow [0,\infty)$ (see Definition (\ref{permissible-def})). A unital Lipschitz pair $(\A,\Lip)$ is an \emph{\Qqcms{F}} when:
\begin{enumerate}
\item $(\A,\Lip)$ is a compact quantum metric space,
\item $\Lip$ is lower semicontinuous,
\item $(\A,\Lip)$ is an $F$-quasi-Leibniz pair.
\end{enumerate}
\end{definition}

In particular:

\begin{definition}[\cite{Latremoliere13}]
A unital Lipschitz pair $(\A,\Lip)$  is a {\Lqcms} when it is a Leibniz pair and $\Lip$ is a lower semicontinuous Lip-norm.
\end{definition}

\begin{remark}
Now, when $(\A,\Lip)$ is a quantum compact metric space, then Assertion (3) of Theorem (\ref{lower-semicontinuous-thm}) is equivalent to $\{a\in\sa{\A} :\Lip(a)\leq 1\}$ is compact in norm in $\sa{\A}$, since it is a totally bounded and closed subset of $\A$, which is complete.
\end{remark}

We note that other restrictions may be put on Lip-norms, besides lower semicontinuity or the Leibniz property. Rieffel introduced compact C*-metric spaces in \cite{Rieffel10c}, for instance, which are quantum compact metric spaces with Lip-norms satisfying the so-called strong Leibniz property. As we will see, many such additional properties can be incorporated in our construction of the dual Gromov-Hausdorff propinquity, which was built with this flexibility in mind. Compact C*-metric spaces are, in particular, {\Lqcms s}.

We now show that many of the unital Lipschitz pairs which we discussed in our first section are, in fact, {\Lqcms s}. It is notable that proving a Lipschitz pair is a quantum compact metric space is, typically, hard.

We begin with the original example from Rieffel in \cite{Rieffel98a}, which shows that all Lipschitz pairs in Example (\ref{ergodic-LP-ex}) are indeed quantum compact metric spaces.

\begin{theorem}[\cite{Rieffel98a}, Theorem 2.3]\label{Rieffel-ergo-thm}
Let $G$ be a compact group with unit $e$ and endowed with a continuous length function $\ell$, and let $\A$ be  a unital C*-algebra equipped with a strongly continuous action $\alpha$ of $G$ by *-automorphisms. For all $a\in\sa{\A}$ we define:
\begin{equation*}
\Lip(a) = \sup\left\{ \frac{\|\alpha^g(a)-a\|_\A}{\ell(g)} : g \in G\setminus\{e\} \right\}\text{.}
\end{equation*}

The following assertions are equivalent:
\begin{enumerate}
\item $(\A,\Lip)$ is a {\Lqcms},
\item  $(\A,\Lip)$ is a Lipschitz pair,
\item $\alpha$ is ergodic, i.e.:
\begin{equation*}
\left\{ a \in \A : \forall g \in G\quad \alpha^g(a) = a \right\} = \R\unit_\A\text{.}
\end{equation*}
\end{enumerate}
\end{theorem}

Thus, Example (\ref{ergodic-LP-ex}) provide a good source of quantum compact metric spaces. In particular, Noncommutative tori and noncommutative solenoids \cite{Latremoliere11c} thus provide examples of quantum compact metric spaces using Theorem (\ref{Rieffel-ergo-thm}) for the dual actions of, respectively, the tori and the product of two solenoid groups.

Spectral triples provide a source of Lipschitz pairs, yet one has to prove that a given spectral triple gives rise to a quantum compact metric space case by case. Of course, Lipschitz pairs constructed from spectral triples are always Leibniz. Moreover, Rieffel showed in \cite[Proposition 3.7]{Rieffel99} that spectral triples give Leibniz pairs with lower semicontinuous seminorms. The difficulty, of course, is to show that the associated {\mongekant} metrizes the weak* topology, using Theorem (\ref{Rieffel-thm}).
 
Ozawa and Rieffel proves that one of the first examples (\ref{Connes-LP-ex}) of a Lipschitz pair from \cite{Connes89} from Hyperbolic groups was indeed a quantum compact metric space in \cite{Ozawa05}:

\begin{theorem}[\cite{Ozawa05}, Theorem 1.2]
Let $G$ be a hyperbolic group and $l$ be the length function associated to some finite generating set of $G$. Let $\A$ be the reduced C*-algebra of $G$, $\pi$ the left regular representation of $G$ on $\ell^2(G)$, and $D$ be the multiplication operator by $l$ on $\ell^2(G)$. If we set $\Lip(a) = \opnorm{[D,\pi(a)]}$ for all $a\in\sa{\A}$ (accepting that $\Lip$ takes the value $\infty$), then $(\A,\Lip)$ is an {\Lqcms}.
\end{theorem}

Another example of a Dirac operator from length functions on groups, for the quantum tori, is given by Rieffel in \cite{Rieffel02}, and preceded the previous result on Hyperbolic group C*-algebras.

\begin{theorem}[\cite{Rieffel02}, Theorem 0.1]
Let $l$ be a length function on $\Z^d$ which is either the word-length function for some finite set of generators of $\Z^d$, or which is the restriction of some norm on $\R^d$. Let $\sigma$ be a skew bicharacter of $\Z^d$. Let $\pi$ be the left regular representation of $C^\ast(\Z^d,\sigma)$ on $\ell^2(\Z^d)$ and $D$ be the operator of pointwise multiplication by $l$ on $\ell^2(\Z^d)$.

If, for all $a \in C^\ast(\Z^d,\sigma)$, we set:
\begin{equation*}
\Lip(a) = \opnorm{ [D,\pi(a) ] }
\end{equation*}
then $(C^\ast(\Z^d),\Lip)$ is an {\Lqcms}.
\end{theorem}

Other examples of spectral triples giving quantum compact metric spaces can be found in \cite{li03}, where Connes-Landi spheres are shown to be compact quantum metric spaces for their natural spectral triples. In a different direction, quantum Heisenberg manifolds are proven to be quantum compact metric spaces by H. Li in \cite{li09}.

Yet another example is given by $AF$ algebras. In the work of Antonescu and Christensen, the following construction is proposed:

\begin{theorem}[\cite{Antonescu04}, Theorem 2.1]]\label{AF-thm}
Let $\A$ be a unital AF C*-algebra, and write $\A = \overline{\bigcup_{n\in\N} \A_n}$ with $\A_n$ a finite dimensional C*-algebra for all $n\in\N$. Let $\varphi\in\StateSpace(\A)$ be faithful and denote by $\pi$ the GNS faithful representation of $\A$ the Hilbert space $\Hilbert$ obtained by completing $\A$ for the inner product $(a,b) \in \A\mapsto \inner{a}{b} = \varphi(b^\ast a)$. 

Thus $\A_n$ can be seen as a Hilbert subspace of $\Hilbert$ (since $\A_n$ is finite dimensional hence closed in $\Hilbert$). Let $Q_n$ be the projection onto $\A_{n+1}\cap \A_{n}^\perp$ for all $n \in \N$. 

There exists a sequence $(\alpha_n)_{n\in\N}$ of real numbers such that, if we set:
\begin{equation*}
D = \sum_{n\in\N} \alpha_n Q_n
\end{equation*}
and $\Lip:a\in\sa{\A} \mapsto \opnorm{[D,\pi(a)]}$, then $(\A,\Lip)$ is an {\Lqcms}.
\end{theorem}

With K. Aguilar, the author actually proposes a different construction for Lip-norms on $AF$-algebras with a faitful tracial state:

\begin{notation}
Let $\mathcal{I} = (\A_n,\alpha_n)_{n\in\N}$ be an inductive sequence with limit $\A=\varinjlim \mathcal{I}$. We denote the canonical *-morphisms $\A_n \rightarrow\A$ by $\indmor{\alpha}{n}$ for all $n\in\N$.
\end{notation}

\begin{theorem}[\cite{Latremoliere15d}]\label{AF-lip-norms-thm}
Let $\A$ be an AF algebra endowed with a faithful tracial state $\mu$. Let $\mathcal{I} = (\A_n,\alpha_n)_{n\in\N}$ be an inductive sequence of finite dimensional C*-algebras with C*-inductive limit $\A$, with $\A_0 = \C$ and where $\alpha_n$ is unital and injective for all $n\in\N$.

Let $\pi$ be the GNS representation of $\A$ constructed from $\mu$ on the space $L^2(\A,\mu)$.

For all $n\in\N$, let:
\begin{equation*}
\CondExp{\cdot}{\indmor{\alpha}{n}(\A_n)} : \A\rightarrow\A
\end{equation*}
be the unique conditional expectation of $\A$ onto the canonical image $\indmor{\alpha}{n}\left(\A_n\right)$ of $\A_n$ in $\A$, and such that $\mu\circ\CondExp{\cdot}{\indmor{\alpha}{n}(\A_n)} = \mu$.

Let $\beta: \N\rightarrow (0,\infty)$ have limit $0$ at infinity. If, for all $a\in\sa{\A}$, we set:
\begin{equation*}
\Lip_{\mathcal{I},\mu}^\beta(a) = \sup\left\{\frac{\left\|a - \CondExp{a}{\indmor{\alpha}{n}(\A_n)}\right\|_\A}{\beta(n)} : n \in \N \right\}
\end{equation*}
then $\left(\A,\Lip_{\mathcal{I},\mu}^\beta\right)$ is a {\Qqcms{2}}. 
\end{theorem}

The advantage of the quantum metrics presented in Theorem (\ref{AF-lip-norms-thm}) is that they recover the usual ultrametrics on Cantor sets and moreover, they allow to construct natural continuous surjections from the Baire space to UHF algebras \cite{Glimm60} and Effr{\"o}s-Shen AF algebras \cite{Effros80}. Moreover, AF algebras equipped with such a Lip-norm are limits of the finite dimensional algebras of the chosen inductive sequence, for the propinquity. We refer the reader to \cite{Latremoliere15d} for these results.

\bigskip

Other examples of quantum compact metric spaces can be found in the literature dealing with quantum groups, in particular in \cite{Voigt14, Li04}.

Now, we turn to the question of how to define a quantum locally compact metric space, based upon a similar intuition as for quantum compact metric spaces. Bounded-Lipschitz distances offer one avenue for exploring such a notion, though it requires us to work with infinitely many metrics as soon as the state space does not have a finite diameter for the {\mongekant}. It is natural to ask what can be said about the {\mongekant} for general Lipschitz pairs. This became the subject of our own research, presented in the next subsection.

\subsubsection{Quantum locally compact metric spaces}

The {\mongekant} associated to a general Lipschitz pair is not quite as well behaved as for unital Lipschitz pairs:
\begin{enumerate}
\item the {\mongekant} is in fact, an extended metric in general, i.e. it may take the value $\infty$ (see Example (\ref{fundamental-LP-ex})),
\item the topology generated by the {\mongekant} is usually not the weak* topology, even when restricted to closed balls: for instance, denoting the Dirac probability measure at $x\in\R$ by $\delta_x$, the sequence $\left( \frac{n}{n+1}\delta_0 + \frac{1}{n+1}\delta_n \right)_{n\in\N}$ weak* converge to $\delta_0$ yet $\Kantorovich{\Lip}\left(\delta_0, \frac{n}{n+1}\delta_0 + \frac{1}{n+1}\delta_n\right) = 1$, where $\Lip$ is the Lipschitz seminorm from the usual metric on $\R$,
\item the weak* topology, on the other hand, is neither compact nor even locally compact on the state space of non-unital C*-algebras.
\end{enumerate}

We thus must revise our approach to quantum metric spaces if we wish to extend our theory to noncompact, quantum locally compact metric spaces. If at least to satisfy a natural curiosity, we are thus led to the question of what property the {\mongekant} possesses in general which may be meaningful in the noncommutative context. 

The difficulties with the {\mongekant} arise because of the problem of escape to infinity. A method to control the behavior of set of probability measures at infinity is suggested by a useful characterization of weak* compact sets of probability measures. Indeed, a set of probability measures $\mathscr{S}$ over some locally compact Hausdorff space $X$ is weak* precompact if and only if it is uniformly tight, i.e. when for any $\varepsilon > 0$, there exists a compact subset $K$ of $X$ such that:
\begin{equation*}
\sup\left\{ \mu\left(K^\complement\right)  : \mu \in \mathscr{S} \right\} < \varepsilon\text{.}
\end{equation*}
This notion is of course topological, not metric, yet the behavior of the {\mongekant} is also controlled by a form of tightness, albeit one which must involve explicitly the metric of the underlying space. Dobrushin \cite{Dobrushin70} made this crucial observation:

\begin{theorem}[Dobrushin, \cite{Dobrushin70}]\label{Dobrushin-thm}
Let $(X,\mathrm{d})$ be a locally compact metric space, and let $\Lip$ be the Lipschitz seminorm associated with $\mathrm{d}$. If $\mathscr{S}$ is a subset of the state space $\StateSpace(C_0(X))$ such that, for $x_0 \in X$:
\begin{equation}\label{D-tight}
\lim_{r\rightarrow\infty} \sup\left\{ \int_{x \in X : \mathrm{d}(x,x_0) > r} d(x,x_0) \, d\mu(x) : \mu \in \mathscr{S} \right\} = 0
\end{equation}
then the topology induced on $\mathscr{S}$ by $\Kantorovich{\Lip}$ is the weak* topology restricted to $\mathscr{S}$.
\end{theorem}

The condition expressed by Equation (\ref{D-tight}) will be labeled \emph{Dobrushin tightness}. This condition does not depend on the choice of the base point named $x_0$ in Equation (\ref{D-tight}) thanks to the triangle inequality. For proper metric spaces, Dobrushin's tightness is a strengthening of the notion of uniform tightness. For more general metric spaces, Dobrushin's tightness may be quite unrelated to tightness --- for instance, every subset of a finite diameter locally compact metric space is Dobrushin's tight, though they are not always tight. Related to this observation, we note that our work \cite{Latremoliere05b} on the bounded Lipschitz distances already addressed the notion of finite diameter quantum locally compact metric spaces without any recourse to some notion of Dobrushin tightness.

We took up in \cite{Latremoliere12b} the challenge to use Dobrushin's Theorem (\ref{Dobrushin-thm}) as the basis for a theory of quantum locally compact metric spaces.

There are several difficulties to overcome for this program. To begin with, Dobrushin's Theorem (\ref{Dobrushin-thm}) invokes the distance function itself, or rather the distance from a given point function. Unfortunately, even in the classical case, for metric spaces of infinite radius, such a function does not lie in the C*-algebra of continuous functions vanishing at infinity --- not even the C*-algebra of bounded continuous functions. In fact, the basic structure of quantum metric spaces is given by a generalized Lipschitz seminorm, and its associated {\mongekant}: we have no candidate for the function distance even in the compact case. To be even more specific, if $(\A,\Lip)$ is a unital Lipschitz pair, then for any fixed $\varphi \in \StateSpace(\A)$, the map $\psi \in \StateSpace(\A) \mapsto \Kantorovich{\Lip}(\varphi,\psi)$ is weak* continuous and convex, but not affine in general --- hence it does not correspond, via Kadisson functional calculus, to any element in $\sa{\A}$.

The second difficulty for our program is that Dobrushin's tightness involves a notion of taking the limit at infinity, which is encoded by taking integrals on the complement of closed balls --- it thus relies on a notion of locality and its dual notion of being, so to speak, far away. Of course, one may argue that the fundamental difference between the commutative and noncommutative world is precisely that locality becomes ill-defined in the noncommutative world. In fact, studies on the question of limits at infinity within general C*-algebras \cite{Akemann89} reveals that such notions are not canonical.

Motivated by all these observations, our idea in \cite{Latremoliere12b} is to introduce a mean to define limits at infinity within C*-algebras, accepting that this mean is an additional choice in our definition of quantum metric spaces. A natural approach is to choose a set of commuting ``observables'', for which the notions of locally and going to infinity are well-defined. We thus introduced in \cite{Latremoliere12b}:

\begin{definition}[\cite{Latremoliere12b}, Definition 2.15]
A \emph{topography} on a C*-algebra $\A$ is an Abelian C*-subalgebra $\M$ of $\A$ containing an approximate identity for $\A$. A \emph{topographic quantum space} $(\A,\M)$ is an ordered pair of a C*-algebra $\A$ and a topography $\M$ on $\A$.
\end{definition}

Our terminology is inspired by the notion of a topographic map. For separable C*-algebras, we shall see that a natural choice of topography for our purpose is of the form $C^\ast(h)$ with $h$ a strictly positive element, which one may regard as an ``altitude'' function, whose ``level sets'' will play a central role in our work. The requirement that a topography on a C*-algebra $\A$ contains an approximate unit for $\A$ is desired to make sense of the notion of going to $\infty$ in $\A$, and not just in the topography. It should be noted that in practice, a natural mean to define topography is via Abelian approximate units, and the Abelian C*-algebras they generate (see \cite[Lemma 2.20]{Latremoliere12b}).

Topographies are commutative C*-algebras, so we take advantage of Gel'fand duality; we will thus use the following notation:

\begin{notation}
Let $(\A,\M)$ be a topographic quantum space. The Gel'fand spectrum of the Abelian C*-algebra $\M$ is denoted by $\M^\sigma$. Moreover, the set of all compact subsets of $\M^\sigma$ is denoted by $\compacts{\M}$ and is ordered by \emph{reverse} inclusion $\succ$. As such $(\compacts{\M},\succ)$ is a directed set.
\end{notation}

\begin{notation}
We regard $\A^{\ast\ast}$ as the universal enveloping Von Neumann algebra of $\A$ (see \cite{Pedersen79}). Let $(\A,\M)$ be a topographic quantum space. If $K\in\compacts{M}$ then we denote by $\chi_X$ the projection in $\A^{\ast\ast}$ defined as the indicator function of $K$ in $\M$.

We also note that every state $\varphi$ on $\A$ trivially extends to a state on $\A^{\ast\ast}$, which will denote as $\varphi$. 
\end{notation}

A first application of the notion of a topographic quantum space is that we may define a notion of (uniformly) tight set of states, in analogy with the classical case:

\begin{definition}[\cite{Latremoliere12b}, Definition 2.21]\label{tight-def}
Let $(\A,\M)$ be a topographic quantum space. A subset $\mathscr{S}$ of $\StateSpace(\A)$ is \emph{tight} when:
\begin{equation*}
\lim_{K \in \compacts{\M}} \sup \left\{ |\varphi(1-\chi_K)| : \varphi \in \mathscr{S} \right\} = 0\text{.}
\end{equation*}
\end{definition}

Tightness characterizes weak* precompact sets of states of a topographic quantum space, as desired:

\begin{theorem}[\cite{Latremoliere12b}, Theorem 2.22]
Let $(\A,\M)$ be a topographic quantum space. The weak* closure of a subset $\mathscr{S}$ of $\StateSpace(\A)$ is weak* compact if, and only if $\mathscr{S}$ is tight.
\end{theorem}

An important structure provided by a topography on a C*-algebra is the \emph{local state space}, consisting of the states which are indeed locally supported in the sense of the topography. These states will play an important role in our work.

\begin{definition}[\cite{Latremoliere12b}, Definition 2.23]
Let $(\A,\M)$ be a topographic quantum space. A state $\varphi\in\StateSpace(\A)$ is \emph{local} when there exists $K\in \compacts{\M}$ such that $\varphi(\chi_K) = 1$. The set of all local states of $(\A,\M)$, denoted by $\StateSpace(\A|\M)$, is called the \emph{local state space} of $(\A,\M)$.
\end{definition}

We note that \cite[Proposition 2.24]{Latremoliere12b} shows that the local state space is norm dense in the state space of a topographic quantum space.

Our insight toward a theory of a quantum locally compact metric spaces was to propose an extension to the notion of a Lipschitz pair which includes a topography:

\begin{definition}[\cite{Latremoliere12b}, Definition 2.27]\label{Lipschitz-triple-def}
A \emph{Lipschitz triple} $(\A,\Lip,\M)$ is a triple where $(\A,\Lip)$ is a Lipschitz pair and $\
M$ is a topography on $\A$.
\end{definition}

The notion of a Lipschitz triple did not occur in the compact or Abelian cases because there is a canonical topography in each of these cases, and it occurred implicitly in the bounded quantum compact metric spaces, though Theorem (\ref{bounded-Lipschitz-thm}) is not affected by which topography is chosen --- a situation which differs greatly from the general picture we now describe.

\begin{example}[\cite{Latremoliere12b}]
If $(\A,\Lip)$ is a unital Lipschitz pair, then $(\A,\Lip,\C\unit_\A)$ is a Lipschitz triple.
\end{example}

\begin{example}[\cite{Latremoliere12b}]
If $(\A,\Lip)$ is a Lipschitz pair with $\A$ Abelian, then $(\A,\Lip,\A)$ is a Lipschitz triple.
\end{example}

\begin{example}[\cite{Latremoliere12b}]
If $(\A,\Lip)$ is a Lipschitz pair with $\A$ separable and
\begin{equation*}
\diam{\StateSpace(\A)}{\Kantorovich{\Lip}} < \infty\text{,}
\end{equation*}
then $(\A,\Lip,C^\ast(h))$ is a Lipschitz triple for any strictly positive $h \in \sa{\A}$. 
\end{example}

We will see that in general, the topography is an important part of the theory of quantum locally compact metric spaces, and different choices of topographies lead to different situations (for instance, given a Lipschitz pair $(\A,\Lip)$, there may be a topography $\M$ such that $(\A,\Lip,\M)$ is a quantum locally compact metric space, yet another topography $\alg{N}$ such that $(\A,\Lip,\alg{N})$ is not a quantum locally compact metric space).

\bigskip

We can now present the core notion of our theory of {\lcqms s}. Putting together the topography and the quantum metric structure from a Lipschitz triple, we introduce in \cite{Latremoliere12b} an analogue to the notion of Dobrushin tightness.

\begin{definition}[\cite{Latremoliere12b}, Definition 2.28]
Let $(\A,\Lip,\M)$ be a Lipschitz triple. A subset $\mathscr{S}$ is \emph{tame} when for some local state $\mu \in \StateSpace(\A|\M)$:
\begin{equation*}
\lim_{K \in \compacts{\M}} \sup\left\{ |\varphi\left(a - \corner{K}{a}\right)| : \varphi \in \mathscr{S}, a\in\sa{\A}, \Lip(a)\leq 1, \mu(a) = 0 \right\} = 0\text{.}
\end{equation*}
\end{definition}

The difference between our notion of a tame subset of states and the notion of Dobrushin's tightness is that tameness implies tightness:

\begin{theorem}[\cite{Latremoliere12b}, Theorem 2.30]
Let $(\A,\Lip,\M)$ be a Lipschitz triple. A tame subset of $\StateSpace(\A)$ is tight.
\end{theorem}

In general, a tame set of regular probability measures is Dobrushin tight, but the converse only holds for proper metric spaces. However, with our sights firmly turned toward a generalized Gromov-Hausdorff topology for {\pqms s}, and since all the bounded quantum locally compact metric spaces will fit our general framework as well (and are only proper when compact), this distinction seems to raise no difficulty.

Now, an important source of tame sets of sets is given by the following example:

\begin{example}[\cite{Latremoliere12b}, Proposition 2.29]\label{local-tame-set-ex}
Let $(\A,\Lip,\M)$ be a Lipschitz triple. Let $K \in \compacts{\M}$. The set:
\begin{equation*}
\StateSpace(\A|K) = \left\{ \varphi \in \StateSpace(\A) : \varphi(\chi_K) = 1 \right\}
\end{equation*}
is tame. 
\end{example}
While the notion of a tame set explicitly involves the quantum metric structure, we note that the tame sets given by Example (\ref{local-tame-set-ex}) are ``universally tame'': they would be tame no matter what the seminorm $\Lip$ is (but for the same topography). Of course, not all tame sets are so nicely behaved.

Example (\ref{local-tame-set-ex}) is tightly related to the role of the local state space since $\StateSpace(\A|\M) = \bigcup_{K\in\compacts{\M}}\StateSpace(\A|\M)$.

Since tame sets of states are weak* precompact, if we ever wish the weak* topology to be metrized by the {\mongekant} on tame sets, then tame sets must have finite diameter for the {\mongekant}. We thus define:

\begin{definition}[\cite{Latremoliere12b}, Definition 2.31]
A Lipschitz triple $(\A,\Lip,\M)$ is \emph{regular} when, for all $K\in\compacts{\M}$:
\begin{equation*}
\diam{\StateSpace(\A|K)}{\Kantorovich{\Lip}} < \infty\text{.}
\end{equation*}
\end{definition}

Regularity is a similar condition as the finiteness of the diameter of the state space for the {\mongekant} in the original work of Rieffel on compact quantum metric spaces \cite{Rieffel98a}, as seen in Assertion (5) of Theorem (\ref{Rieffel-thm}) for instance. While regularity involves only tame sets of the form given in Example (\ref{local-tame-set-ex}), it implies that all tame sets are bounded for the {\mongekant}:
\begin{proposition}[\cite{Latremoliere12b}, Proposition 2.36]
Let $(\A,\Lip,\M)$ be a regular Lipschitz triple. If $\mu\in\StateSpace(\A|\M)$ and $\mathscr{K}$ is tame, then $\mathscr{K}$ is contained in a closed ball of center $\mu$ for $\Kantorovich{\Lip}$.
\end{proposition}

A nice consequence of regularity of Lipschitz triples is that testing if a set of states is tight can be done using any local state:
\begin{theorem}[\cite{Latremoliere12b}, Theorem 2.35]
Let $(\A,\Lip,\M)$ be a regular Lipschitz triple. A subset $\mathscr{S}$ of $\StateSpace(\A)$ is tame if and only if for every local state $\mu \in \StateSpace(\A|\M)$ we have:
\begin{equation*}
\lim_{K \in \compacts{\M}} \sup\left\{ |\varphi\left(a - \corner{K}{a}\right)| : \varphi \in \mathscr{S}, a\in\sa{\A}, \Lip(a)\leq 1, \mu(a) = 0 \right\} = 0\text{.}
\end{equation*}
\end{theorem}

Another consequence of --- in fact, an equivalence with --- regularity is:
\begin{proposition}[\cite{Latremoliere12b}, Proposition 2.34]\label{norm-bound-prop}
Let $(\A,\Lip,\M)$ be a Lipschitz triple and $\mu \in \StateSpace(\A|\M)$. The following assertions are equivalent:
\begin{enumerate}
\item $(\A,\Lip,\M)$ is regular,
\item for all $K\in\compacts{\M}$, there exists $r_K in (0,\infty)$ such that for all $a\in\sa{\A}$ with $\Lip(a)\leq 1$, we have:
\begin{equation*}
\|\corner{K}{a}\|_{\A^{\ast\ast}} \leq r_K\text{.}
\end{equation*}
\end{enumerate}
\end{proposition}

By Proposition (\ref{norm-bound-prop}), for a regular Lipschitz triple $(\A,\Lip,\M)$, the sets $\{ a\in\sa{\A} : \Lip(a)\leq 1, \mu(a) = 0\}$ for some local state $\mu$ and any $K\in\compacts{\M}$, is bounded for a certain locally convex topology. We shall see this topology and these sets again in Theorem (\ref{lcqms-char-thm}).
\bigskip

We now have the necessary ingredients to give the main definition of our work \cite{Latremoliere12b}.

\begin{definition}[\cite{Latremoliere12b}, Definition 3.1]
A Lipschitz triple $(\A,\Lip,\M)$ is a \emph{\lcqms} when the associated {\mongekant} $\Kantorovich{\Lip}$ metrizes the weak* topology restricted to any tame subset of $\StateSpace(\A)$.
\end{definition}

Remarkably, the topography of a {\lcqms} carries a natural metric structure:
\begin{theorem}[\cite{Latremoliere12b}, Theorem 3.2]\label{induced-metric-thm}
Let $(\A,\Lip,\M)$ be a {\lcqms}. Let $\M^\sigma$ be the Gel'fand spectrum of $\M$. For any two states $\rho, \omega$ of $\M$, we set:
\begin{equation*}
\mathrm{d}(\omega,\rho) = \inf \left\{ \Kantorovich{\Lip}(\varphi, \psi) : \varphi, \psi \in \StateSpace(\A), [\varphi]_\M = \omega, [\psi]_\M = \rho \right\}
\end{equation*}
where $[\cdot]_\M$ is meant for the restriction to $\M$.

Then $\mathrm{d}$ is an extended metric on $\StateSpace(\M)$, such that for all $K\in\compacts{\M}$, the topology induced by $\mathrm{d}$ on $\StateSpace(\M|K)$ is the weak* topology. Moreover, $(\M^\sigma, \mathrm{d})$ is a locally compact metric space whose topology agrees with the weak* topology on $\M^\sigma$.
\end{theorem}

\bigskip

The key result in \cite{Latremoliere12b} is our characterization of {\lcqms s} in terms of the Lipschitz triple data, in the spirit of our work on the bounded-Lipschitz distance \cite{Latremoliere05b}:

\begin{theorem}[\cite{Latremoliere12b}, Theorem 3.10]\label{lcqms-char-thm}
Let $(\A,\Lip,\M)$ be a Lipschitz triple. The following assertions are equivalent:
\begin{enumerate}
\item $(\A,\Lip,\M)$ is a {\lcqms},
\item for all $s, t \in \M$ compactly supported and for all local state $\mu$ of $\A$, the set:
\begin{equation*}
\left\{ s a t : a\in\sa{\unital{\A}}, \Lip(a)\leq 1, \mu(a) = 0 \right\}
\end{equation*}
is precompact for $\|\cdot\|_\A$,
\item for all $s, t \in \M$ compactly supported and for some local state $\mu$ of $\A$, the set:
\begin{equation*}
\left\{ s a t : a\in\sa{\unital{\A}}, \Lip(a)\leq 1, \mu(a) = 0 \right\}
\end{equation*}
is precompact for $\|\cdot\|_\A$.
\end{enumerate}
\end{theorem}

When working with separable C*-algebras, we obtain a somewhat more practical version of Theorem (\ref{lcqms-char-thm}):

\begin{theorem}[\cite{Latremoliere12b}, Theorem 3.11]\label{separable-lcqms-thm}
Let $(\A,\Lip,\M)$ be a Lipschitz triple. If $\A$ is separable, then the following assertions are equivalent:
\begin{enumerate}
\item $(\A,\Lip,\M)$ is a {\lcqms},
\item there exists a strictly positive $h \in \sa{\M}$ and a local state $\mu$ of $\A$ such that the set:
\begin{equation*}
\left\{ h a h : a\in\sa{\unital{\A}}, \Lip(a)\leq 1, \mu(a) = 0 \right\}
\end{equation*}
is precompact for $\|\cdot\|_\A$,
\item there exists a strictly positive $h \in \sa{\M}$ such that, for all local states $\mu$ of $\A$, the set:
\begin{equation*}
\left\{ h a h : a\in\sa{\unital{\A}}, \Lip(a)\leq 1, \mu(a) = 0 \right\}
\end{equation*}
is precompact for $\|\cdot\|_\A$.
\end{enumerate}
\end{theorem}

We may apply Theorem (\ref{lcqms-char-thm}) to establish several interesting examples of {\lcqms s}.

\begin{example}[\cite{Latremoliere12b}, Theorem 4.1]
If $(X,\mathrm{d})$ is a locally compact metric space and $\Lip$ is the Lipschitz seminorm associated with $\mathrm{d}$, as in Example (\ref{fundamental-LP-ex}), then:
\begin{equation*}
(C_0(X), \Lip, C_0(X))
\end{equation*}
is a {\lcqms}.
\end{example}

\begin{example}[\cite{Latremoliere12b}, Theorem 4.2]
Let $(\A,\Lip)$ be a unital Lipschitz pair. Then $\M\subseteq \A$ is a topography for $\A$ if and only if $\M$ is a unital Abelian C*-subalgebra of $\A$, and moreover the following are equivalent:
\begin{enumerate}
\item $(\A,\Lip)$ is a quantum compact metric space,
\item $(\A,\Lip,\C\unit_\A)$ is a {\lcqms},
\item $(\A,\Lip,\M)$ is a {\lcqms} for some topography $\M$ of $\A$,
\item $(\A,\Lip,\M)$ is a {\lcqms} for all topographies $\M$ of $\A$.
\end{enumerate}

Thus, all quantum compact metric spaces are indeed {\lcqms s}.
\end{example}

\begin{example}[\cite{Latremoliere12b}, Theorem 4.6]
Let $(\A,\Lip)$ be a separable Lipschitz pair. Then $(\A,\Lip)$ is a bounded quantum locally compact metric space if and only if, for some (and hence, for all) strictly positive element $h\in\sa{\A}$:
\begin{enumerate}
\item the Lipschitz triple $(\A,\Lip,C^\ast(h))$ is a {\lcqms},
\item $\diam{\StateSpace(\A)}{\Kantorovich{\Lip}} < \infty$. 
\end{enumerate}

This statement is reassuring: it states that bounded quantum locally compact metric spaces are, well, bounded quantum locally compact metric spaces --- with two distinct definitions of these words, although \cite[Theorem 4.6]{Latremoliere12b} states that these definitions agree after all.
\end{example}

\begin{example}[\cite{Latremoliere12b}, Theorem 4.9]
The Moyal plane, as discussed in Example (\ref{Moyal-LP-ex}), is a {\lcqms}.
\end{example}

We also note that in \cite[Section 4.4]{Latremoliere12b}, we give another example of {\lcqms} constructed over the algebra of compact operators, which show that the choice of topography matters when working with infinite-diameter {\lcqms s}. This contrasts with \cite[Theorem 4.6]{Latremoliere12b} where all topographies will do when working with bounded --- and in particular, compact --- {\lcqms}.

\bigskip

Our strategy to prove Theorem (\ref{lcqms-char-thm}) follows a similar path to our work in \cite{Latremoliere05b}, although the techniques are more involved. The key is to introduce a new topology on topographic quantum spaces:

\begin{definition}[\cite{Latremoliere12b}, Definition 3.5]\label{topographic-top-def}
Let $(\A,\M)$ be a topographic quantum space. The \emph{topographic topology} on $\A$ is the locally convex topology generated by the seminorms:
\begin{equation*}
n_K : a\in\sa{\A} \longmapsto \sup\left\{|\varphi(a)| : \varphi\in\StateSpace(\A|K) \right\}
\end{equation*}
for all $K \in \compacts{\M}$.
\end{definition}

This topology differs from the weakly uniform topology since it only involves seminorms associated with certain tame sets, rather than all the tight sets. However, on bounded sets, these two topologies agree:

\begin{proposition}[\cite{Latremoliere12b}, Proposition 3.8]\label{bounded-agree-prop}
Let $(\A,\M)$ be a topographic quantum space and $\B\subseteq \A$ be a bounded subset of $\A$. The weakly uniform topology and the topographic topology agree on $\B$.
\end{proposition}
Proposition (\ref{bounded-agree-prop}) shows why, when working with the bounded-Lipschitz distance, the weakly uniform topology was used with no reference to any topography.

\bigskip

The connection between the topographic topology and the {\mongekant} is reminiscent of the relation between the weakly uniform topology and the bounded-Lipschitz metric:
\begin{theorem}[\cite{Latremoliere12b}, Theorem 3.9]\label{lcqms-char0-thm}
Let $(\A,\Lip,\M)$ be a regular Lipschitz triple. The following assertions are equivalent:
\begin{enumerate}
\item $(\A,\Lip,\M)$ is a {\lcqms},
\item for all $\mu \in \StateSpace(\A|\M)$, the set:
\begin{equation*}
\{ a\in\sa{\A} : \Lip(a)\leq 1, \mu(a) = 0 \}
\end{equation*}
is totally bounded for the topographic topology,
\item for some $\mu\in\StateSpace(\A|\M)$, the set:
\begin{equation*}
\{ a\in\sa{\A} : \Lip(a)\leq 1, \mu(a) = 0 \}
\end{equation*}
is totally bounded for the topographic topology.
\end{enumerate}
\end{theorem}

Now, additional effort can be applied to make Theorem (\ref{lcqms-char0-thm}) more amenable to applications, by stating conditions in terms of the basic ingredients of a Lipschitz triple, we proved Theorems (\ref{lcqms-char-thm}) and (\ref{separable-lcqms-thm}).

\bigskip

We now have defined notions of {\lcqms s}, and we wish to define an analogue of the Gromov-Hausdorff distance on them. We will focus on our own construction of such an analogue in the rest of this document, starting with the compact framework.

\section{The Gromov-Hausdorff Propinquity}

As an informal motivation for our work and the introduction of the quantum dual Gromov-Hausdorff propinquity, we begin this section with the problem which stimulated much of our research. For all $n\in\N$, let us be given a complex numbers $\rho_n$ such that $\rho_n^n = 1$, and let us define the two $n\times n$ unitary matrices:
\begin{equation*}
U_n = \begin{pmatrix}
0 & 0 & \hdots &          0 &1 \\
1 & 0      & \hdots &  0 & 0 \\
0 & 1 & \ddots &  \vdots & \vdots\\
\vdots & \ddots & \ddots  & 0 & 0 \\
0  & \hdots       & \hdots       &1& 0\\ 
\end{pmatrix}
\text{ and }
V_n = \begin{pmatrix}
1 &      & & & &\\
  & \rho_n & & & &\\
  & & \rho_n^2 & & &\\
  & & & \ddots & & \\
  & & & & & \rho_n^{n-1}
\end{pmatrix}\text{.}
\end{equation*}
By construction, $U_nV_n = \rho_n V_n U_n$. Such pairs of matrices appear in the literature in mathematical physics as well as quantum information theory, among others. The C*-algebras $C^\ast(U_n,V_n)$ are sometimes called fuzzy tori. Often, a desirable outcome of some computations carried out over fuzzy tori is that one can obtain interesting results when $n$ goes to infinity under the condition that the sequence $(\rho_n)_{n\in\N}$ converges --- examples of such situations are found in the mathematical physics literature, for instance \cite{Connes97,Seiberg99,tHooft02,Zumino98}, to cite but a few. 

Informally, one would expect that the limit of the fuzzy tori $C^\ast(U_n,V_n)$ would be the universal C*-algebra  $C^\ast(U,V)$ generated by two unitaries $U$ and $V$ subject to the relation $UV = \rho VU$ where $\rho = \lim_{n\rightarrow\infty} \rho_n$, i.e. a quantum torus. Yet, as quantum tori are not AF --- for instance, their $K_1$ groups are nontrivial --- making sense of such a limiting process is challenging. 

Rieffel proposed \cite{Rieffel00} to start investigating such problems by finding a noncommutative analogue of the Gromov-Hausdorff distance, based upon the metric geometry of the state space provided by the structure of quantum metric spaces described in the previous section.  Rieffel's quantum Gromov-Hausdorff distance \cite{Rieffel00} provides a first framework in which such a limit can be justified, and we proved that indeed, fuzzy tori converge to the quantum tori in our first paper \cite{Latremoliere05}. However, Rieffel's distance may be null between *-isomorphic C*-algebras: in other words, it does not capture the C*-algebraic structure fully. 

This relative lack of connection between the C*-algebraic structure and the first noncommutative analogue of the Gromov-Hausdorff distance sparked quite a lot of research, in an effort to obtain at least the desired coincidence property that distance zero implies *-isomorphism. Many papers were written using a first approach to this problem \cite{kerr02,li03,li05,kerr09}: encapsulate additional C*-algebraic information directly in the construction of a quantum version of the Gromov-Hausdorff distance. In other words, the quantum metric structure and the quantum topology are not connected; instead the Gromov-Hausdorff analogue tries to include a measure on how both are close.

However, recent research in noncommutative metric geometry has made apparent that the natural connection between the quantum metric structure, provided by Lip-norms, and the quantum topological structure, provided by the C*-algebras, is the Leibniz property (in some form), and that this connection is a key tool if one wishes to explore how C*-algebraic structures behave with respect to convergence --- an important example of such a research project is Rieffel's work on convergence of modules \cite{Rieffel08, Rieffel09, Rieffel10c, Rieffel11}. Yet, as seen in \cite{kerr02, Rieffel10c}, the construction of a noncommutative Gromov-Hausdorff distance within the realm of {\Lqcms} proved challenging --- for instance, the proximity of Rieffel \cite{Rieffel10c} is not known to satisfy the triangle inequality. It is largely the triangle inequality property of a prospective noncommutative Gromov-Hausdorff distance within the category of {\Lqcms s} which raises difficulties.

We proposed in \cite{Latremoliere13,Latremoliere13b, Latremoliere13c, Latremoliere14, Latremoliere14b, Latremoliere15}  that a second route to create a noncommutative analogue of the Gromov-Hausdorff distance adapted to C*-algebras is to embrace the Leibniz property. We call our new metric the dual Gromov-Hausdorff propinquity. It has the desired coincidence property --- *-isomorphism is necessary for null distance --- and provides a framework where all objects under considerations are {\Lqcms s}, or more generally {\Qqcms{F}s} for some a priori choice of a permissible function $F$, i.e. a form of the Leibniz identity. Moreover, the dual Gromov-Hausdorff propinquity gives the same topology as the Gromov-Hausdorff distance when restricted to the classical picture, and is a complete metric. Thus, our effort answered the challenge of constructing such a metric, and addresses the coincidence property by tying together the quantum metric structure and quantum topology.

In fact, our construction may be applied to various subcategories of {\gQqcms s}, allowing one to choose which properties of quantum metric spaces one may need. We refer to the metrics thus obtained as specialization of the dual Gromov-Hausdorff propinquity. A particularly relevant such specialization is the quantum propinquity, which we introduced in \cite{Latremoliere13}. This metric dominates the more general dual Gromov-Hausdorff propinquity, and provides a tool to establish examples of convergence, such as the convergence of fuzzy tori to the quantum tori discussed in this introduction \cite{Latremoliere13b}. Notably, this form of the dual Gromov-Hausdorff propinquity plays a role in Rieffel's research on module convergence \cite{Rieffel15}, where, paired with ideas from Wu \cite{Wu05, Wu06a, Wu06b} on Lip-norms for operator spaces, it allows to study convergence of matrix algebras over {\Lqcms s}.

We begin our chapter with a brief overview of the Gromov-Hausdorff distance, and then proceed to describe our new family of metrics. We then introduce a special form which plays an important role in current research, and was our original construction. We then discuss the convergence of fuzzy tori to the quantum tori, and then discuss some notions of perturbations of the metric on {\Lqcms s}. 

We note that our section will focus on the dual Gromov-Hausdorff propinquity, and we refer readers to the above mentioned references for an exposition on earlier proposals for noncommutative analogues of the Gromov-Hausdorff distance. In addition to the original version in \cite{Rieffel00}, we also refer to the survey \cite{Rieffel08}.

\subsection{The Gromov-Hausdorff distance}

The Gromov-Hausdorff distance is a metric between arbitrary compact metric spaces introduced in \cite{Gromov81} by Gromov in his study of the problem of relating growth of groups to some of their structure. More specifically, Gromov proved that Cayley graphs of groups with polynomial growths converge, in a proper sense, to certain manifolds, and was able to infer from this convergence that such groups are virtually nilpotent (i.e. contain a nilpotent subgroup of finite index).

The original Gromov-Hausdorff distance \cite{Gromov81} was introduced in the context of locally compact metric spaces. We will discuss a noncommutative analogue of Gromov's construction for quantum proper metric spaces in a latter chapter. In this chapter, we shall focus on the restriction of the Gromov-Hausdorff distance to the class of compact metric spaces. Interestingly, for compact metric spaces, this metric was already presented by Edwards \cite{Edwards75}, motivated by Wheeler's superspace approach to quantum gravitation \cite{Wheeler68}.

We begin our summary with the first notion of a metric on compact subsets of a metric space, due to F. Hausdorff \cite{Hausdorff}:

\begin{definition}[p. 293, \cite{Hausdorff}]
Let $(X,\mathrm{d})$ be a metric space and let $\mathscr{K}(X,\mathrm{d})$ be the set of all nonempty compact subsets of $(X,\mathrm{d})$. For any two $A, B \in \mathscr{K}(X,\mathrm{d})$, we set:
\begin{equation*}
\Haus{\mathrm{d}}(A,B) = \max\left\{ \sup_{x\in A} \mathrm{d}(x,B), \sup_{x\in B} \mathrm{d}(x,A) \right\}
\end{equation*}
where:
\begin{equation*}
d(x,C) = \inf\{ \mathrm{d}(x,y) : y \in C \}
\end{equation*}
for all $x\in X$ and $\emptyset \not= C \subseteq X$. 
\end{definition}

\begin{theorem}[\cite{Hausdorff}]
Let $(X,\mathrm{d})$ be a metric space. Then $\Haus{\mathrm{d}}$ is a metric on the set of all nonempty compact subsets $\mathscr{K}(X,\mathrm{d})$ of $(X,\mathrm{d})$. Moreover, if $(X,\mathrm{d})$ is complete, then so is $(\mathscr{K}(X),\Haus{\mathrm{d}})$ and if $(X,\mathrm{d})$ is compact, then so is $(\mathscr{K}(X),\Haus{\mathrm{d}})$.
\end{theorem}

Gromov proposes an intrinsic form of the Hausdorff distance, defined between arbitrary compact metric spaces. By intrinsic, we mean that Gromov's distance does not depend on a particular ambient space in which the two metric spaces live.

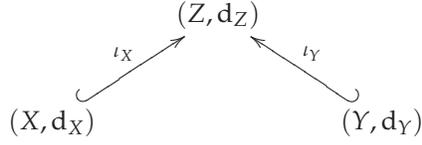
\begin{figure}[h]
\begin{equation*}
\xymatrix{
 & (Z,\mathrm{d}_Z)  & \\
(X,\mathrm{d}_X) \ar@{^{(}->}[ur]^{\iota_X} & & (Y,\mathrm{d}_Y)  \ar@{_{(}->}[ul]_{\iota_Y}
}
\end{equation*}
\caption{Gromov-Hausdorff Isometric Embeddings}\label{GH-fig}
\end{figure}

\begin{definition}[\cite{Edwards75, Gromov81}]\label{GH-def}
Let $(X,\mathrm{d}_X)$ and $(Y,\mathrm{d}_Y)$ be two compact metric spaces. We define:
\begin{multline*}
\GH((X,\mathrm{d}_X),(Y,\mathrm{d}_Y)) = \\
\inf\left\{ \Haus{\mathrm{d_Z}}(\iota_X(X),\iota_Y(Y)) : \exists (Z,\mathrm{d}_Z) \quad \exists \iota_X : X\hookrightarrow Z, \iota_Y:Y\hookrightarrow Z\quad \right. \\ \left. \text{$\iota_X$, $\iota_Y$ are isometries into the compact metric space $(Z,\mathrm{d}_Z)$}  \right\} \text{.}
\end{multline*}
\end{definition}

Thus, the Gromov-Hausdorff distance is the infimum of the Hausdorff distance between two given compact metric spaces for all possible ambient compact metric spaces, as in Figure (\ref{GH-fig}). Gromov established:

\begin{theorem}[\cite{Gromov81}]
The Gromov-Hausdorff distance $\GH$ is well-defined, and satisfies, for any compact metric spaces $(X,\mathrm{d}_X)$, $(Y,\mathrm{d}_Y)$ and $(Z,\mathrm{d}_Z)$:
\begin{enumerate}
\item $\GH((X,\mathrm{d}_X),(Y,\mathrm{d}_Y)) \leq \diam{X}{\mathrm{d}_X} + \diam{Y}{\mathrm{d}_Y}$,
\item $\GH((X,\mathrm{d}_X), (Y,\mathrm{d}_Y)) = 0$ if and only if there exists an isometry from $(X,\mathrm{d}_X)$ onto $(Y,\mathrm{d}_Y)$,
\item $\GH((X,\mathrm{d}_X), (Z,\mathrm{d}_Z)) \leq \GH((X,\mathrm{d}_X), (Y,\mathrm{d}_Y)) + \GH((Y,\mathrm{d}_Y), (Z,\mathrm{d}_Z))$,
\item $\GH((X,\mathrm{d}_X), (Y,\mathrm{d}_Y)) = \GH((Y,\mathrm{d}_Y), (X,\mathrm{d}_X))$,
\item $\GH$ is complete.
\end{enumerate}
\end{theorem}

We note that, while there is no set of all possible metric spaces containing isometric copies of two given compact metric spaces within ZF, the definition of the Gromov-Hausdorff distance does not raise any difficulty within this same axiomatic: indeed, for any set $E$, and any predicate $P$, the axiom of selection in ZF implies that $\{ x\in E : P(x) \}$ is a set. Note that consequently, $\GH$ is indeed the infimum of a set of real numbers, satisfying some predicate. Moreover, this set of real numbers is not empty --- one may construct an easy metric with the desired properties on the disjoint union $X\coprod Y$ --- and is bounded below by $0$. Thus, $\GH$ is certainly well-defined.

$\GH$ is a complete distance on the class of all compact metric spaces, up to isometry.  Note that since every compact metric is separable, one may in fact consider the Gromov-Hausdorff distance as a metric over the set consisting of all possible metrics over $\N$ with compact completion. We will however not need this description.

There is, moreover, a very natural compactness criterion for classes of compact metric spaces for the Gromov-Hausdorff distance \cite{Gromov81, Gromov}. This result was key to the original application of this metric in group theory. We will discuss this theorem when presenting its analogue for the dual Gromov-Hausdorff propinquity.

We refer to Gromov's book \cite{Gromov} and Burago and al \cite{burago01} for an exposition of properties and applications of the Gromov-Hausdorff distance in geometry. Our purpose is to generalize this metric to the realm of {\gQqcms s}.

\subsection{The Dual Gromov-Hausdorff Propinquity}

\subsubsection{Tunnels}

The dual of Figure (\ref{GH-fig}), thanks to our discussion around Definition (\ref{isometry-def}), is given naturally by Figure (\ref{tunnel-fig}), and the following definition:

\begin{definition}[\cite{Latremoliere13c}, Definition 3.1, \cite{Latremoliere15}, Definition 2.15]\label{tunnel-def}
Let $F$ be a permissible function, and let $(\A,\Lip_\A)$ and $(\B,\Lip_\B)$ be two {\Qqcms{F}s}. An $F$-\emph{tunnel} $\tau = (\D,\Lip_\D,\pi_\A,\pi_\B)$ from $(\A,\Lip_\A)$ to $(\B,\Lip_\B)$ is a quadruple where:
\begin{enumerate}
\item $(\D,\Lip_\D)$ is a {\Qqcms{F}}, 
\item $\pi_\A$ and $\pi_\B$ are isometric epimorphisms from $(\D,\Lip_\D)$ onto, respectively, $(\A,\Lip_\A)$ and $(\B,\Lip_\B)$.
\end{enumerate}
We call $(\A,\Lip_\A)$ the \emph{domain} $\dom{\tau}$ of $\tau$ and we call $(\B,\Lip_\B)$ the \emph{codomain} $\codom{\tau}$ of $\tau$. 
\end{definition}

Our original definition \cite[Definition 3.1]{Latremoliere13c} was made in the context of {\Lqcms s} only.

\begin{remark}
In his original construction, Rieffel \cite{Rieffel00} defined, for any two quantum compact metric spaces $(\A,\Lip_\A)$ and $(\B,\Lip_\B)$, an \emph{admissible Lip-norm} $\Lip$ as a Lip-norm on $\A\oplus\B$ whose quotient to $\A$ and $\B$ are respectively given as $\Lip_\A$ and $\Lip_\B$, \emph{without any quasi-Leibniz condition}. Thus $(\A\oplus\B, \Lip, \pi_\A,\pi_\B)$, with $\pi_\A:\A\oplus\B\twoheadrightarrow\A$ and $\pi_\B: \A\oplus\B\twoheadrightarrow\B$ the canonical surjections, is a tunnel when $\Lip$ is admissible and satisfy an appropriate Leibniz property. 

Rieffel defined a quantity associated with admissible Lip-norms, akin to our reach for tunnels, defined below. The infimum of this quantity over all admissible Lip-norms for two given quantum compact metric spaces is the quantum Gromov-Hausdorff distance between these spaces \cite{Rieffel00}. As we mentioned, it may be null even if the underlying C*-algebras are not *-isomorphic. Moreover, admissible Lip-norms do not need to possess any relation with the multiplicative structure --- in fact, Rieffel's theory is developed for order-unit spaces instead of C*-algebras. Thus, distance zero leads to an isomorphism of order-unit space.

If ones wishes to be able to carry out computations with the admissible Lip-norms which give a good estimate on the quantum Gromov-Hausdorff distance, then one may desire to impose that admissible Lip-norms be Leibniz, for instance \cite{Rieffel10c}. However, doing so without modifying Rieffel's construction otherwise leads to an object called the proximity, which may not satisfy the triangle inequality. Thus our work in this section resolves this apparent trade-off.
\end{remark}

\begin{figure}[h]
\begin{equation*}
\xymatrix{
 & (\D,\Lip_\D)  \ar@{->>}[dl]^{\pi_\A} \ar@{->>}[dr]_{\pi_\B} & \\
(\A,\Lip_\A)  & & (\B,\Lip_\B)
}
\end{equation*}
\caption{A tunnel}\label{tunnel-fig}
\end{figure}
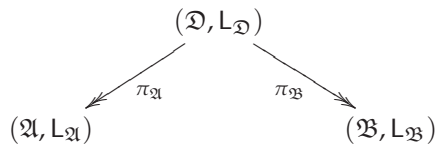

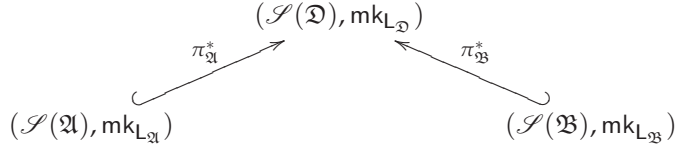
\begin{figure}[h]
\begin{equation*}
\xymatrix{
 & (\StateSpace(\D),\Kantorovich{\Lip_\D})  & \\
(\StateSpace(\A),\Kantorovich{\Lip_\A}) \ar@{^{(}->}[ur]^{\pi_\A^\ast} & & (\StateSpace(\B),\Kantorovich{\Lip_\B})  \ar@{_{(}->}[ul]_{\pi_\B^\ast}}
\end{equation*}
\caption{The dual of a tunnel}\label{dual-tunnel-fig}
\end{figure}

\begin{notation}
Let $\pi : \A\rightarrow\B$ be a unital *-morphism. We denote by $\pi^\ast$ the restriction of the dual map of $\pi$ to $\StateSpace(\B)$, i.e.
\begin{equation*}
\pi^\ast : \varphi \in\StateSpace(\B)\rightarrow \varphi\circ\pi \in \StateSpace(\A)\text{.}
\end{equation*}
\end{notation}

As observed in our section of morphisms of Lipschitz pairs, if $(\D,\Lip_\D,\pi,\rho)$ is a tunnel from $(\A,\Lip_\A)$ to $(\B,\Lip_\B)$, then $\pi^\ast$ and $\rho^\ast$ are isometries from, respectively, $(\StateSpace(\A),\Kantorovich{\Lip_\A})$ and $(\StateSpace(\B),\Kantorovich{\Lip_\B})$ into $(\StateSpace(\D),\Kantorovich{\Lip_\D})$. Thus we obtain Figure (\ref{dual-tunnel-fig}), which is naturally reminiscent of Figure (\ref{GH-fig}).

\bigskip

The construction of the dual Gromov-Hausdorff propinquity begins by associating numerical values to a tunnel, meant to measure how far apart the domain and codomain of a tunnel are. 

The first such numerical value, introduced in \cite{Latremoliere13c}, is a natural analogue of the Hausdorff distance between spaces embedded in a metric space, and is called the reach of the tunnel.

\begin{definition}[\cite{Latremoliere13c}, Definition 3.4]\label{tunnel-reach}
Let $F$ be a permissible function. Let $(\A,\Lip_\A)$ and $(\B,\Lip_\B)$ be two {\Qqcms{F}s}, and let $\tau = (\D,\Lip_\D,\pi_\A,\pi_\B)$ be an $F$-tunnel from $(\A,\Lip_\A)$ to $(\B,\Lip_\B)$. The \emph{reach} $\tunnelreach{\tau}$ of $\tau$ is the non-negative real number:
\begin{equation*}
\tunnelreach{\tau} = \Haus{\Kantorovich{\Lip_\D}}\left(\pi_\A^\ast\left(\StateSpace(\A)\right), \pi_\B^\ast\left(\StateSpace(\B)\right) \right)\text{.}
\end{equation*}
\end{definition}

The second value introduced in \cite{Latremoliere13c} is new to our construction, and has no direct equivalent in the classical picture. Indeed, McShane Theorem \cite{McShane34} can be strengthened easily by noticing that since the pointwise maximum and minimum of two $k$-Lipschitz functions is again $k$-Lipschitz, if $X\subseteq Z$, with $(Z,\mathrm{d})$ a metric space, and $f : X\rightarrow\R$ is $k$-Lipschitz, then there exists a $k$-Lipschitz extension $g : Z\rightarrow\R$ of $f$ with the same uniform norm as $f$. 

In the noncommutative world, we may not expect, in general, that if $(\D, \Lip_\D, \allowbreak \pi_\A, \pi_\B)$ is a tunnel from $(\A,\Lip_\A)$ to $(\B,\Lip_\B)$, and if $a\in\sa{\A}$ and $\varepsilon > 0$, then there exists $d\in\pi_\A^{-1}(\{a\})$ with not only $\Lip_\A(a)\leq \Lip_\D(d)\leq \Lip_\A(a)+\varepsilon$ but also $\|a\|_\A\leq\|d\|_\D\leq\|a\|_\A+\varepsilon$: the truncation argument used in the classical setting does not carry to the general framework of quantum compact metric spaces.

In order to obtain some information on the norms of lifts of elements with finite Lip-norms in a tunnel, we are thus led to the following definition:

\begin{definition}[\cite{Latremoliere13c}, Definition 3.7]\label{tunnel-depth}
Let $F$ be a permissible function. Let $(\A,\Lip_\A)$ and $(\B,\Lip_\B)$ be two {\Qqcms{F}s}, and let $\tau = (\D,\Lip_\D,\pi_\A,\pi_\B)$ be an $F$-tunnel from $(\A,\Lip_\A)$ to $(\B,\Lip_\B)$. The \emph{depth} $\tunneldepth{\tau}$ of $\tau$ is the non-negative real number:
\begin{equation*}
\tunnelreach{\tau} = \Haus{\Kantorovich{\Lip_\D}}\left(\StateSpace(\D),\co{\pi_\A^\ast\left(\StateSpace(\A)\right)\cup\pi_\B^\ast\left(\StateSpace(\B)\right)} \right)\text{,}
\end{equation*}
where $\co{E}$ is the weak* closure of the convex envelope of $E$ for any $E\subseteq \D^\ast$.
\end{definition}

As we discussed, the depth does not enter in the classical picture, or in fact in any construction of analogues of the Gromov-Hausdorff distance prior to the dual propinquity. There is actually two very important reasons for this.

First, it is easy to check that for any two compact metric spaces $(X,\mathrm{d}_X)$ and $(Y,\mathrm{d}_Y)$, we have the identity:
\begin{equation*}
\GH((X,\mathrm{d}_X),(Y,\mathrm{d}_Y)) = \inf\left\{ \Haus{\mathds{d}}(X,Y) : \mathrm{d} \in \Adm\{\mathrm{d_X},\mathrm{d_Y}\} \right\}
\end{equation*}
where $X\coprod Y$ is the disjoint union of $X$ and $Y$ and $\Adm\{\mathrm{d}_X,\mathrm{d}_Y\}$ is the set of all compact metrics on $X\coprod Y$ whose restriction to $X$ and $Y$ are, respectively, given by $\mathrm{d}_X$ and $\mathrm{d}_Y$.

Now, $C(X\coprod Y)$ is *-isomorphic to $C(X)\oplus C(Y)$, and trivially, for any tunnel of the form $(C(X)\oplus C(Y),\Lip, \pi_X, \pi_Y)$, where $\pi_X:C(X)\oplus C(Y)\twoheadrightarrow C(X)$ and $\pi_X:C(X)\oplus C(Y)\twoheadrightarrow C(X)$ are the canonical surjections, the depth is null. So, informally speaking, it is always possible to compute the Gromov-Hausdorff distance between two classical metric spaces using only tunnels with zero depth. This property does not extend to the noncommutative setting.

The second reason for the importance of the depth of the tunnel is that its purpose is to control the norm of lifts of elements of finite Lip-norm, which is essential if we wish to apply the quasi-Leibniz property. The entire purpose of our construction of the propinquity is indeed to be compatible with the Leibniz property, but also put it to use: it will be crucial to obtain *-isomorphisms between {\gQqcms s} at distance zero for our propinquity. No other noncommutative analogue of the Gromov-Hausdorff distance relies on the Leibniz property. One could  thus argue that the depth is a core contribution from the construction of our propinquity.

Now, we can create two natural synthetic numerical values for tunnels, which capture both the reach and depth, and allow for the construction of a metric. Originally, we propose the length \cite{Latremoliere13c}:

\begin{definition}[\cite{Latremoliere13c}, Definition 3.9]\label{tunnel-length}
Let $F$ be a permissible function. Let $(\A,\Lip_\A)$ and $(\B,\Lip_\B)$ be two {\Qqcms{F}s}, and let $\tau = (\D,\Lip_\D,\pi_\A,\pi_\B)$ be an $F$-tunnel from $(\A,\Lip_\A)$ to $(\B,\Lip_\B)$. The \emph{length} $\tunnellength{\tau}$ of $\tau$ is the non-negative real number:
\begin{equation*}
\tunnellength{\tau} = \max\left\{ \tunnelreach{\tau}, \tunneldepth{\tau} \right\}\text{.}
\end{equation*}
\end{definition}

In our later work \cite{Latremoliere14}, we noticed that an equivalent quantity to the length was theoretically quite useful, in particular in providing a nice proof that the propinquity satisfies the triangle inequality --- a nontrivial fact using the length, and a challenge in general when working with Leibniz Lip-norms \cite{kerr02, Rieffel10c}. In practice, the length may seem a bit more tractable, although time will tell which of the length and the extent is easiest to use. 

\begin{definition}[\cite{Latremoliere14}, Definition 2.11] \label{tunnel-extent}
Let $F$ be a permissible function. Let $(\A,\Lip_\A)$ and $(\B,\Lip_\B)$ be two {\Qqcms{F}s}, and let $\tau = (\D,\Lip_\D,\pi_\A,\pi_\B)$ be an $F$-tunnel from $(\A,\Lip_\A)$ to $(\B,\Lip_\B)$. The \emph{extent} $\tunnelextent{\tau}$ of $\tau$ is the non-negative real number:
\begin{multline*}
\tunnelextent{\tau} = \max\left\{ \Haus{\Kantorovich{\Lip_\D}}\left(\StateSpace(\D),\pi_\A^\ast\left(\StateSpace(\A)\right)\right)),\right.\\
\left. \Haus{\Kantorovich{\Lip_\D}}\left(\StateSpace(\D), \pi_\B^\ast\left(\StateSpace(\B)\right)\right)\right\}\text{.}
\end{multline*}
\end{definition}

The relationship between length and extent is described by:

\begin{proposition}[\cite{Latremoliere14}, Proposition 2.12]\label{length-extent-eq-prop}
For any permissible function $F$, any {\Qqcms{F}s} $(\A,\Lip_\A)$ and $(\B,\Lip_\B)$ and any $F$-tunnel $\tau$ from $(\A,\Lip_\A)$ to $(\B,\Lip_\B)$, we have:
\begin{equation*}
\tunnellength{\tau} \leq \tunnelextent{\tau} \leq 2\tunnellength{\tau}\text{.}
\end{equation*}
\end{proposition}

We thus have two constructions for the dual Gromov-Hausdorff propinquity, one using length and one using extent, though Proposition (\ref{length-extent-eq-prop}) suggests that both constructions would lead to equivalent metrics. Informally, one may expect that the dual Gromov-Hausdorff propinquity is the infimum of the length, or the extent, of all $F$-tunnels between any two given {\Qqcms{F}s} --- once a particular permissible function $F$ has been fixed. This informal approach actually only works for the extent.

There is however, a small subtlety to consider. Once we have fixed a particular permissible function $F$, there is still quite a lot of  choices one may consider regarding the collection of $F$-tunnels which may desirable to work with. The point to emphasize is that computing the dual Gromov-Hausdorff propinquity between two {\Qqcms{F}s} involves working within the {\Qqcms{F}s} coming from tunnels, and one might wish to have more structure than just the $F$-quasi-Leibniz property. A particularly relevant example comes from \cite{Rieffel10c}, where one may want to work with so-called \emph{strong Leibniz} Lip-norms, i.e. Leibniz Lip-norms $\Lip$ defined on a dense subspace of the whole C*-algebra $\A$ and such that for all invertible $a\in \A$ we have $\Lip(a^{-1})\leq\|a\|^2 \Lip(a)$. 

Our construction of the dual Gromov-Hausdorff propinquity allows for various constraints on tunnels, within some reasonable conditions to ensure the resulting object is indeed a metric. These conditions, however, depend slightly on whether we use the length of the extent for our construction. 

We will  begin our exposition with the extent. We will explain how to use the length instead when discussing the quantum Gromov-Hausdorff propinquity, which is not compatible with the extent construction, but is a special case of the length construction of the dual Gromov-Hausdorff propinquity. 

\subsubsection{A First Construction of the dual Gromov-Hausdorff Propinquity and The triangle Inequality}

This section proposes our construction of the dual Gromov-Hausdorff using the extent of tunnels. We begin by exploring the notion of composition of tunnels, which is the basis for our proof of the triangle inequality for the dual propinquity:

\begin{theorem}[\cite{Latremoliere14}, Theorem 3.1]\label{tunnel-composition-thm}
Let $F$ be a permissible function. Let $(\A_1,\Lip_1)$, $(\A_2,\Lip_2)$ and $(\A_3,\Lip_3)$ be three {\Qqcms{F}s} and let $\tau_{12} = (\D_{12},\Lip_{12},\pi_1,\rho_1)$ and $\tau_{23}=(\D_{23},\Lip_{23},\pi_2,\rho_2)$ be two $F$ tunnels, respectively from $(\A_1,\Lip_1)$ to $(\A_2,\Lip_2)$ and from $(\A_2,\Lip_2)$ to $(\A_3,\Lip_3)$.

If $\varepsilon > 0$, then there exists an $F$-tunnel $\tau$ from $(\A_1,\Lip_1)$ to $(\A_3,\Lip_3)$ such that:
\begin{equation*}
\tunnelextent{\tau} \leq \tunnelextent{\tau_{12}} + \tunnelextent{\tau_{23}} + \varepsilon\text{.}
\end{equation*}
\end{theorem}

Keeping the notations of Theorem (\ref{tunnel-composition-thm}), let us briefly indicate what a possible tunnel $\tau$ would look like. Let $\D = \D_{12} \oplus \D_{23}$. We define for all $(d_1,d_2) \in \sa{\D} = \sa{\D_{12}} \oplus \sa{\D_{23}}$:
\begin{equation*}
\Lip (d_1,d_2) = \max\left\{ \Lip_{12}(d_1), \Lip_{23}(d_2), \frac{1}{\varepsilon} \|\rho_1(d_1) - \pi_2(d_2)\|_{\A_2} \right\}\text{.}
\end{equation*}
Set $\pi : (d_1,d_2)\in \D \mapsto \pi_1(d_1)\in\A_1$ and $\rho : (d_1,d_2) \in \D \mapsto \rho_2(d_2)\in\A_3$. In \cite{Latremoliere14}, we check that $(\D, \Lip, \pi, \rho)$ is indeed an $F$-tunnel with the desired extend. There are two comments which arise from this construction. First, even if $\D_{12} = \A_1\oplus\A_2$ and $\D_{23} = \A_2\oplus\A_3$, then $\D$ is not *-isomorphic to $\A_1\oplus\A_3$: thus, allowing for more general embeddings than just into the noncommutative direct sum for tunnels is essential to this construction. This is a key difference between our construction and all the earlier constructions of analogues of the Gromov-Hausdorff distance.

Second, we note that the map $N : (d_1,d_2)\in \D \mapsto \|\rho_1(d_1) - \pi_2(d_2)\|_{\A_2}$ satisfies a form of Leibniz inequality:
\begin{equation*}
N(d_1d_1', d_2 d_2') \leq \|d_1\|_{\D_1} N(d_1',d_2') + N(d_1,d_2)\|d_2\|_{\D_2}
\end{equation*}
for all $d_1,d_1'\in\D_1$, $d_2,d_2'\in\D_2$. Since $F$ is permissible, and thus:
\begin{equation}\label{perm-Leib-eq}
F(a,b,l_a,l_b) \geq a l_b + b l_a
\end{equation}
for all $a,b,l_a,l_b \geq 0$, we conclude that indeed, $\Lip$ is $F$-quasi-Leibniz. This is precisely for this observation that we required that permissibility includes the condition given by Inequality (\ref{perm-Leib-eq}).

Composing tunnels is the tool which we use to prove that our dual Gromov-Hausdorff propinquity satisfy the triangle inequality, so whatever restriction we may consider putting on tunnels later on, it must be compatible with some form of composition. More generally, we shall require the following compatibility conditions between a class of tunnels and a class of {\gQqcms s} so that we can carry on our construction.

\begin{definition}[\cite{Latremoliere14}, Definition 3.5]\label{appropriate-def}
Let $F$ be a permissible function. Let $\mathcal{C}$ be a nonempty class of {\Qqcms{F}s}. A class $\mathcal{T}$ of $F$-tunnels is \emph{appropriate} for $\mathcal{C}$ when: 
\begin{enumerate}
\item $\mathcal{T}$ is connected: For any $\mathds{A}, \mathds{B} \in \mathcal{C}$, there exists $\tau \in \mathcal{T}$ from $\mathds{A}$ to $\mathds{B}$,
\item $\mathcal{T}$ is symmetric: if $\tau = (\D,\Lip_\D,\pi,\rho) \in \mathcal{T}$ then $\tau^{-1} = (\D,\Lip_\D,\rho,\pi) \in \mathcal{T}$,
\item $\mathcal{T}$ is triangular: if $\tau,\tau' \in \mathcal{T}$ and if the domain of $\tau'$ is the codomain of $\tau$, then for all $\varepsilon > 0$ there exists $\tau''$ from the domain of $\tau$ to the codomain of $\tau'$ such that:
\begin{equation*}
\tunnelextent{\tau''} \leq \tunnelextent{\tau} + \tunnelextent{\tau'} + \varepsilon\text{.}
\end{equation*}
\item $\mathcal{T}$ is specific: if $\tau \in \mathcal{T}$ then the domain and codomain of $\tau$ lies in $\mathcal{C}$,
\item $\mathcal{T}$ is definite: for any $(\A,\Lip_\A), (\B,\Lip_\B) \in \mathcal{C}$, if there exists an isometric *-isomorphism $h : \A\rightarrow\B$ then both $(\A, \Lip_\A, \mathrm{id}_\A, h^{-1})$ and $(\B,\Lip_\B,h,\mathrm{id}_\B)$ belong to $\mathcal{T}$, where $\mathrm{id}_E$ is the identity map of the set $E$ for any set.
\end{enumerate}
\end{definition}

\begin{example}
Let $F$ be a permissible function and let. The class $\mathfrak{TQQCMS}$ of all $F$-tunnels is appropriate for the class  $\mathfrak{QQCMS}$ of all {\Qqcms{F}s}. Most assertions from Definition (\ref{appropriate-def}) are trivially check, and Theorem (\ref{tunnel-composition-thm}) ensures that the triangularity property is satisfied.
\end{example}

The following notation will prove useful:
\begin{notation}
Let $F$ be a permissible function. Let $\mathcal{C}$ be a nonempty class of {\Qqcms{F}s} and let $\mathcal{T}$ be a $\mathcal{C}$-appropriate class of $F$-tunnels. If $(\A,\Lip_\A)$ and $(\B,\Lip_\B)$ are in $\mathcal{C}$, then the class of all tunnels from $(\A,\Lip_\A)$ to $(\B,\Lip_\B)$ which belong to $\mathcal{T}$ is denoted by:
\begin{equation*}
\tunnelset{\A,\Lip_\A}{\B,\Lip_\B}{\mathcal{T}}\text{.}
\end{equation*}
\end{notation}

We now define the main object of our research, a new noncommutative analogue of the Gromov-Hausdorff distance adapted to the C*-algebraic setting which we call the dual Gromov-Hausdorff propinquity.

\begin{definition}[\cite{Latremoliere13c}, Definition 3.21,\cite{Latremoliere14}, Definition 3.6]\label{propinquity-def}
Let $F$ be an admissible function. Let $\mathcal{C}$ be a nonempty class of {\Qqcms{F}s} and let $\mathcal{T}$ be a $\mathcal{C}$-appropriate class of $F$-tunnels. Let $(\A,\Lip_\A)$ and $(\B,\Lip_\B)$ in $\mathcal{C}$.  The \emph{$\mathcal{T}$-dual Gromov-Hausdorff propinquity} $\propinquity{\mathcal{T}}\left((\A,\Lip_\A),(\B,\Lip_\B)\right)$ between $(\A,\Lip_\A)$ and $(\B,\Lip_\B)$ is defined as:
\begin{equation*}
\propinquity{\mathcal{T}}\left((\A,\Lip_\A),(\B,\Lip_\B)\right) = \inf\left\{ \tunnelextent{\tau} : \tau \in \tunnelset{\A,\Lip_\A}{\B,\Lip_\B}{\mathcal{T}} \right\}\text{.}
\end{equation*} 
\end{definition}

\begin{notation}
When working with the class of all {\Lqcms}, if $\mathcal{T}$ is the class of all Leibniz tunnels, then $\propinquity{\mathcal{T}}$ is simply denoted by $\propinquity{}$. If $F$ is some permissible function, and we work with the class $\mathfrak{TQQCMS}$ of all $F$-tunnels, then $\propinquity{\mathfrak{TQQCMS}}$ is simply denoted by $\propinquity{F}$.
\end{notation}

By default, the dual Gromov-Hausdorff propinquity refers to the distance $\propinquity{}$ on the class of {\Lqcms s} using all possible Leibniz tunnels. Yet, many results from our construction apply to the various forms the propinquity can take.

To begin with, we observe that:

\begin{proposition}[\cite{Latremoliere13c}, Proposition 3.24]
Let $F$ be a permissible function. Let $\mathcal{C}$ be a class of {\Qqcms{F}s} and let $\mathcal{T}$ be a $\mathcal{C}$-appropriate class of $F$-tunnels. Then:
\begin{equation*}
\propinquity{\mathcal{T}}((\A,\Lip_\A),(\B,\Lip_\B)) < \infty\text{.}
\end{equation*}

Moreover, if $\mathcal{T}$ is the class of all $F$-tunnels, then:
\begin{equation*}
\propinquity{\mathcal{T}}((\A,\Lip_\A),(\B,\Lip_\B)) \leq \max\{\diam{\StateSpace(\A)}{\Kantorovich{\Lip_\A}}, \diam{\StateSpace(\B)}{\Kantorovich{\Lip_\B}} \}\text{.}
\end{equation*}
\end{proposition}

Theorem (\ref{tunnel-composition-thm}) ensures that some class of tunnels are triangular --- in particular, the class of all tunnels for a given choice of the Leibniz property. In turn, this property of appropriate classes of tunnels allows us to prove:

\begin{theorem}[\cite{Latremoliere13c}, \cite{Latremoliere14}, Theorem 3.7]
Let $F$ be a permissible function. Let $\mathcal{C}$ be a nonempty class of {\Qqcms{F}s} and let $\mathcal{T}$ be a $\mathcal{C}$-appropriate class of tunnels. For all $(\A,\Lip_\A)$, $(\B,\Lip_\B)$ and $(\D,\Lip_\D)$ we have:
\begin{equation*}
\propinquity{\mathcal{T}}\left((\A,\Lip_\A),(\D,\Lip_\D)\right) \leq \propinquity{\mathcal{T}}\left((\A,\Lip_\A),(\B,\Lip_\B)\right) + \propinquity{\mathcal{T}}\left((\B,\Lip_\B),(\D,\Lip_\D)\right)
\end{equation*}
and
\begin{equation*}
\propinquity{\mathcal{T}}\left((\A,\Lip_\A),(\B,\Lip_\B)\right) = \propinquity{\mathcal{T}}\left((\B,\Lip_\B),(\A,\Lip_\A)\right)\text{.}
\end{equation*}
\end{theorem}

We now turn to a core motivation of our construction: the dual Gro\-mov-\-Haus\-dorff propinquity is, in fact, a metric up to isometric *-isomorphism.

\subsubsection{Coincidence Property}

We established in \cite{Latremoliere13c} the main theorem that the dual Gromov-Hausdorff propinquity is, in fact, a metric up to *-isomorphism: thus our metric genuinely captures the C*-algebraic structure.

\begin{theorem}[\cite{Latremoliere13c}, Theorem 4.16]\label{coincidence-thm}
Let $F$ be a permissible function. Let $\mathcal{C}\not=\emptyset$ be a class of {\Qqcms{F}s} and let $\mathcal{T}$ be a $\mathcal{C}$-appropriate class of tunnels. For all $(\A,\Lip_\A)$ and $(\B,\Lip_\B)$ in $\mathcal{C}$, the following two assertions are equivalent:
\begin{enumerate}
\item $\propinquity{\mathcal{T}}\left((\A,\Lip_\A), (\B,\Lip_\B)\right) = 0$,
\item there exists a *-isomorphism $\theta : \A \xlongrightarrow{\simeq} \B$ such that for all $a\in\A$ we have $\Lip_\B\circ \theta(a) = \Lip_\A(a)$.
\end{enumerate}
\end{theorem}

The proof of this important theorem relies on an interpretation of tunnels as a form of morphisms, akin to a correspondence between metric spaces. Theorem (\ref{tunnel-composition-thm}) showed that tunnels may be composed, although not in a unique manner. The conditions of Definition (\ref{appropriate-def}) could be read as the description of a structure modeled after a category, with tunnels for morphisms, albeit in a loose sense. Now, we push this analogy somewhat further.

We begin by defining the image of an element by a tunnel. Such an image is of course a set, and again depends on an additional choice of a real number. We call this image the target set of an element.

\begin{definition}[\cite{Latremoliere13c}, Definition 4.1]\label{targetset-def}
Let $F$ be a permissible function, $\mathcal{C}\not=\emptyset$ be a class of {\Qqcms{F}s} and $\mathcal{T}$ a $\mathcal{C}$-appro\-priate class of $F$-tunnels. Let:
\begin{equation*}
\tau=(\D,\Lip_\D,\pi_\A,\pi_\B) \in \tunnelset{\A,\Lip_\A}{\B,\Lip_\B}{\mathcal{T}}\text{.}
\end{equation*}

For any $a\in\sa{\A}$, and any $l \geq \Lip_\A(a)$, we define the \emph{lift set} of $a$ by:
\begin{equation*}
\liftsettunnel{\tau}{a}{l} = \left\{ d\in\sa{\D} : d\in\sa{\D}, \pi_\A(d) = a \text{ and } \Lip_\D(d)\leq l \right\}
\end{equation*}
and the \emph{target set} of $a$ by:
\begin{equation*}
\targetsettunnel{\tau}{a}{l} = \pi_\B\left(\liftsettunnel{\tau}{a}{l}\right)\text{.}
\end{equation*}
\end{definition}

Target sets are compact and nonempty under the assumptions of Definition (\ref{targetset-def}) by \cite[Lemma 4.2]{Latremoliere13c}.

One of most important result about the extent of a tunnel is akin to a statement about the continuity of a tunnel as a generalized morphism.

\begin{proposition}[\cite{Latremoliere13c}, Proposition 4.4]\label{fundamental-extent-prop}
Let $F$ be a permissible function, $\mathcal{C}\not=\emptyset$ be a class of {\Qqcms{F}s} and $\mathcal{T}$ a $\mathcal{C}$-appropriate class of $F$-tunnels. Let:
\begin{equation*}
\tau=(\D,\Lip_\D,\pi_\A,\pi_\B) \in \tunnelset{\A,\Lip_\A}{\B,\Lip_\B}{\mathcal{T}}\text{.}
\end{equation*}

Let $a\in\dom{\Lip_\A}$ and $l \in \R$ with $\Lip_\A(a)\leq l$. If $d\in\liftsettunnel{\tau}{a}{l}$, then:
\begin{equation*}
\|d\|_\D \leq \|a\|_\A + l \tunnelextent{\tau}\text{.}
\end{equation*}

Consequently, if $b\in\targetsettunnel{\tau}{a}{l}$ then:
\begin{equation*}
\|b\|_\B \leq \|a\|_\A + l \tunnelextent{\tau}\text{.}
\end{equation*}
\end{proposition}

As we wish to see tunnels as generalized morphisms, we should naturally connect the underlying algebraic structures of the Jordan-Lie algebras of the domain of Lip-norms with target sets. These generalized algebraic morphisms notions are given by the following:

\begin{proposition}[\cite{Latremoliere13c}, Corollary 4.5, Proposition 4.8]\label{generalized-morphism-prop}
Let $F$ be a permissible function, $\mathcal{C}\not=\emptyset$ be a class of {\Qqcms{F}s} and $\mathcal{T}$ a $\mathcal{C}$-appropriate class of $F$-tunnels. Let:
\begin{equation*}
\tau=(\D,\Lip_\D,\pi_\A,\pi_\B) \in \tunnelset{\A,\Lip_\A}{\B,\Lip_\B}{\mathcal{T}}\text{.}
\end{equation*}

If $a,a'\in\dom{\Lip_\A}$ and $l \in \R$ with $\max\{\Lip_\A(a),\Lip_\A(a')\}\leq l$, then for all $b\in\targetsettunnel{\tau}{a}{l}$ and $b'\in\targetsettunnel{\tau}{a'}{l}$:
\begin{enumerate}
\item for all $t\in\R$, we have:
\begin{equation*}
b + t b' \in \targetsettunnel{\tau}{a+ta'}{(1+|t|)l}\text{,}
\end{equation*}
\item we have:
\begin{equation*}
\Jordan{b}{b'} \in \targetsettunnel{\tau}{\Jordan{a}{a'}}{F(\|a\|_\A+2\tunnelextent{\tau}, \|a'\|_\A+2\tunnelextent{\tau}, l, l)}
\end{equation*}
and
\begin{equation*}
\Lie{b}{b'} \in \targetsettunnel{\tau}{\Lie{a}{a'}}{F(\|a\|_\A+2\tunnelextent{\tau}, \|a'\|_\A+2\tunnelextent{\tau}, l, l)}\text{.}
\end{equation*}
\end{enumerate}
\end{proposition}

Propositions (\ref{fundamental-extent-prop}) and (\ref{generalized-morphism-prop}) complete our picture of tunnels as generalized morphisms and are the key to Theorem (\ref{coincidence-thm}). Moreover, a key consequence of these two propositions, at the center of our proof of Theorem (\ref{coincidence-thm}), is that the diameters of target sets are controlled by the extent, or equivalently, by the lengths of tunnels:

\begin{corollary}[\cite{Latremoliere13c}, Corollary 4.5]\label{diameter-corollary}
Let $F$ be a permissible function, $\mathcal{C}$ be a nonempty class of {\Qqcms{F}s} and $\mathcal{T}$ a $\mathcal{C}$-appropriate class of $F$-tunnels. Let:
\begin{equation*}
\tau=(\D,\Lip_\D,\pi_\A,\pi_\B) \in \tunnelset{\A,\Lip_\A}{\B,\Lip_\B}{\mathcal{T}}\text{.}
\end{equation*}

If $a,a'\in\dom{\Lip_\A}$ and $l \in \R$ with $\max\{\Lip_\A(a),\Lip_\A(a')\}\leq l$, then for all $b\in\targetsettunnel{\tau}{a}{l}$ and $b'\in\targetsettunnel{\tau}{a'}{l}$:
\begin{equation*}
\|b-b'\|_\B \leq \|a-a'\|_\A + 2 l \tunnelextent{\tau}\text{.}
\end{equation*}
In particular:
\begin{equation*}
\diam{\targetsettunnel{\tau}{a}{l}}{\|\cdot\|_\B} \leq 2 l \tunnelextent{\tau}\text{.}
\end{equation*}
\end{corollary}

The strategy to prove Theorem (\ref{coincidence-thm}) in \cite{Latremoliere13c} consists, therefore, in proving that if we are given a sequence of tunnels whose extent converge to zero, then we can find a subsequence of tunnels which, seen as generalized morphisms, converge to an actual isometric Jordan-Lie morphism which is also continuous for the Lip-norms. Then, using the fact that tunnels are always invertible, one hope to build an inverse morphism at the limit (possibly extracting another subsequence), and concludes with Theorem (\ref{coincidence-thm}). This strategy proves, indeed, successful. Central to the construction of the convergent subsequence of tunnels is the fact that the unit ball for Lip-norms of {\gQqcms s} are totally bounded modulo scalars by Theorem (\ref{Rieffel-thm}), and hence compact modulo scalar, since our Lip-norms are always lower semi-continuous. Pushing our analogy between tunnels and morphisms one step further, we could claim that our proof in \cite{Latremoliere13c} includes a form of the Arz{\'e}la-Ascoli theorem for tunnels.

We thus constructed a metric on {\Lqcms s}, and more generally on {\Qqcms{F}s} for any choice of a permissible $F$ (we note that we must choose $F$ first and then get a metric via our construction; we do not get a metric on the class of  all {\gQqcms s}). Our efforts, of course, were motivated by finding an analogue of the Gromov-Hausdorff distance in noncommutative geometry, and the next section shows that this goal was achieved as well.

\subsubsection{Comparison with Gromov-Hausdorff and other Metrics}

The dual propinquity can be compared to three important objects. First is the quantum Gro\-mov-Haus\-dorff distance \emph{$\dist_q$} \cite{Rieffel00}, which is a pseudo-metric on the class of quantum compact metric spaces and was the first noncommutative analogue of the Gromov-Hausdorff distance. Second is the proximity \emph{$\prox$}, a modified version of the quantum Gromov-Hausdorff distance introduced by Rieffel in \cite{Rieffel10c} to deal with compact C*-metric spaces, which are a type of {\Lqcms s}.  Yet the proximity is not known to satisfy the triangle inequality. Our metric takes its name from the proximity. Last, of course, we wish to compare our new metric to the Gromov-Hausdorff distance \emph{$\GH$}, when working with classical metric spaces.

The following two theorems summarize our results:

\begin{theorem}[\cite{Latremoliere13c}, Theorem 5.5]\label{ncpropinquity-comparison-thm}
Let $F$ be a permissible function and $\mathcal{C}$ be a non-empty class of {\Qqcms{F}s}. Let $(\A,\Lip_\A)$ and $(\B,\Lip_\B)$ be in $\mathcal{C}$, and let $\mathcal{T}$ be a $\mathcal{C}$-appropriate class of tunnels. Then:
\begin{equation}\label{comparison-eq-1}
\mathrm{dist}_q((\A,\Lip_\A),(\B,\Lip_\B))\leq \propinquity{\mathcal{T}}((\A,\Lip_\A),(\B,\Lip_\B))  \text{.}
\end{equation}
If $\mathcal{T}\subseteq\mathcal{G}$ are two $\mathcal{C}$-appropriate classes of tunnels, then:
\begin{equation}\label{comparison-eq-2}
\propinquity{\mathcal{G}}((\A,\Lip_\A),(\B,\Lip_\B))\leq\propinquity{\mathcal{T}}((\A,\Lip_\A),(\B,\Lip_\B)) \text{.}
\end{equation}

Moreover, if $(\A,\Lip_\A)$ and $(\B,\Lip_\B)$ are both compact C*-metric spaces, then:
\begin{equation}\label{comparison-eq-3}
\propinquity{}((\A,\Lip_\A),(\B,\Lip_\B)) \leq \propinquity{\ast}((\A,\Lip_\A),(\B,\Lip_\B)) \leq \prox((\A,\Lip),(\B,\Lip_\B)) \text{,}
\end{equation}
where $\propinquity{\ast}$ is the specialized dual propinquity to the class of compact C*-metric spaces.
\end{theorem}

We note that in \cite{Latremoliere13c}, where we constructed the propinquity using lengths, the class of tunnels was compatible with, rather than appropriate for the choice of a class of {\Lqcms s}. The proof is however unaffected by this small change, and we will discuss the length construction in a later section.

\begin{theorem}[\cite{Latremoliere13c}, Corollary 5.7]
Let $(X,\mathsf{d}_X)$ and $(Y,\mathsf{d}_Y)$ be two compact metric spaces, and let $\mathsf{GH}$ be the Gro\-mov-Haus\-dorff distance \cite{Gromov}. Then:
\begin{equation*}
\propinquity{}((C(X),\Lip_X),(C(Y),\Lip_Y)) \leq \mathsf{GH}((X,\mathsf{d}_X),(Y,\mathsf{d}_Y))\text{,}
\end{equation*}
where $\Lip_X$ and $\Lip_Y$ are, respectively, the Lipschitz seminorms associated to $\mathsf{d}_X$ and $\mathsf{d}_Y$.

Thus, the topology induced by the dual Gromov-Hausdorff propinquity on the class of compact metric spaces agrees with the topology induced by the Gromov-Hausdorff distance.
\end{theorem}

Rieffel's quantum Gromov-Hausdorff distance is known to metrize the same topology as the Gromov-Hausdorff distance on the class of classical compact metric spaces; since the propinquity is between Rieffel's metric and Gromov's metric on this class, it also provides the same topology.

Thus, it is fair to see the dual Gromov-Hausdorff propinquity as a noncommutative analogue of the Gromov-Hausdorff distance. Theorem (\ref{coincidence-thm}) shows that our metric does remember the C*-algebraic distance and fix the coincidence property matter for the quantum Gromov-Hausdorff distance in the C*-algebra framework. By construction, the dual Gromov-Hausdorff propinquity allows to work entirely within the framework of {\Lqcms s}, unlike any other construction of noncommutative Gromov-Hausdorff distances, at least without sacrificing the triangle inequality. Moreover, one may adjust the construction to work within various sub-classes of {\Lqcms s}, or even choose a more lenient form of the Leibniz property, and still work with a well-behaved metric.

\subsubsection{Completeness}

The dual Gromov-Hausdorff propinquity shares another desirable property with the Gromov-Hausdorff distance: it is a complete metric:

\begin{theorem}[\cite{Latremoliere13c}, Theorem 6.27]\label{complete-thm}
Let $F$ be a \emph{continuous} permissible function. The dual Gromov-Hausdorff propinquity $\propinquity{F}$ is a complete metric.
\end{theorem}

The proof of completeness is quite technical. We shall only mention one aspect of the construction of the limit of a Cauchy sequence for the Gromov-Hausdorff propinquity, which significantly impact the structure of our proof of Theorem (\ref{complete-thm}).

The candidate for a limit is constructed as a quotient of a particular {\gQqcms}. However, quotient of Leibniz seminorms may not be Leibniz --- a difficulty which carries to the more general quasi-Leibniz situation. Thus, while we can obtain a quantum compact metric space as a limit, the proper Leibniz property requires quite some care. 

The idea is that any element $a$ of the quotients with a given Lip-norm admit, for any $\varepsilon > 0$, some lift with both Lip-norm and norm within $\varepsilon$  of the norm and Lip-norm of $a$. Assuming that the chosen permissible function is continuous, we then can obtain the desired quasi-Leibniz property at the limit. We refer to \cite[Section 6]{Latremoliere13c} and to \cite{Latremoliere15} for the proofs and a detailed account of these technical matters.

\subsection{Gromov's Compactness and Finite Dimensional Approximations}

\subsubsection{Compactness for the Dual Gromov-Hausdorff Propinquity}

Gromov's compactness theorem \cite{Gromov81} is a central tool when working with the Gromov-Hausdorff distance, and reads as follows:

\begin{theorem}[Gromov's Compactness Theorem]\label{Gromov-Compactness-thm}
A class $\mathcal{S}$ of compact metric spaces is totally bounded for the Gromov-Hausdorff distance if, and only if the following two assertions hold:
\begin{enumerate}
\item there exists $D\geq 0$ such that for all $(X,\mathrm{m}) \in \mathcal{S}$, the diameter of $(X,\mathrm{m})$ is less or equal to $D$,
\item there exists a function $G: (0,\infty)\rightarrow \N$ such that for every $(X,\mathrm{m}) \in \mathcal{S}$, and for every $\varepsilon > 0$, the smallest number $\mathrm{Cov}_{(X,\mathrm{m})}(\varepsilon)$ of balls of radius $\varepsilon$ needed to cover $(X,\mathrm{m})$ is no more than $G(\varepsilon)$. 
\end{enumerate}
Since the Gromov-Hausdorff distance is complete, a class of compact metric spaces is compact for the Gromov-Hausdorff distance if and only if it is closed and totally bounded.
\end{theorem}

Theorem (\ref{Gromov-Compactness-thm}) relates intimately to the matter of finite approximations for metric spaces: of course, every compact metric space is the Gromov-Hausdorff limit of its finite subsets. In particular, the covering number for a compact metric space is always finite, for any $\varepsilon > 0$. The situation is more complicated in the noncommutative setting, as we shall see in the next section. 

We now present our analogue of Theorem (\ref{Gromov-Compactness-thm}) for the dual propinquity. We shall need a few regularity conditions on our choice of a quasi-Leibniz property:

\begin{definition}[\cite{Latremoliere15}, Definition 3.4]
A function $F : [0,\infty)^4 \rightarrow [0,\infty)$ is \emph{strongly permissible} when:
\begin{enumerate}
\item $F$ is permissible,
\item $F$ is continuous,
\item for all $\lambda, \mu, x,y,l_x,l_y \in [0,\infty)$ we have:
\begin{equation*}
\lambda\mu F(x,y,l_x,l_y) = F(\lambda x, \mu y, \lambda l_x, \mu l_y)\text{,}
\end{equation*}
\item for all $x,y \in [0,\infty)$ we have $F(x,y,0,0) = 0$.
\end{enumerate}
\end{definition}

A first and important observation is that sets of finite dimensional {\gQqcms s} with bounded diameter form compact sets for the dual Gromov-Hausdorff propinquity.

\begin{notation}
Let $F$ be a permissible function. We let $\text{\calligra QuasiLeibniz}(F)$ be the class of all {\Qqcms{F}s} and we let $\finitedimclass{F}$ be the class of all finite dimensional {\Qqcms{F}s}.
\end{notation}

\begin{theorem}[\cite{Latremoliere15}, Theorem 3.6]
Let $F$ be a strongly permissible function. For all $d\in\N$ and $D\geq 0$, the class:
\begin{equation*}
\left\{ (\A,\Lip) \in \finitedimclass{F} : \dim_\C \A \leq d, \diam{\StateSpace(\A)}{\Kantorovich{\Lip}} \leq D \right\}
\end{equation*}
is compact for the dual Gromov-Hausdorff propinquity $\propinquity{F}$.
\end{theorem}

We now define an analogue of the covering number in our setting, inspired by the hypothesis of Theorem (\ref{Gromov-Compactness-thm}). 

\begin{definition}[\cite{Latremoliere15}, Definition 4.1]
Let $F$ be a permissible function. Let $(\A,\Lip)$ be a {\Qqcms{F}} and $\varepsilon > 0$. The \emph{$F$-covering number} of $(\A,\Lip)$ is:
\begin{equation*}
\covn{F}{\A,\Lip_\A}{\varepsilon} = \min\left\{ \dim_\C(\B,\Lip_\B) : \begin{array}{l}
\exists (\B,\Lip_\B)\in \text{\calligra QuasiLeibniz}(F)\\
\propinquity{F}((\A,\Lip_\A),(\B,\Lip_\B)) \leq \varepsilon
\end{array}
\right\}\text{.}
\end{equation*}
\end{definition}

Without additional requirement, the covering number $\covn{F}{\A,\Lip}{\varepsilon}$ of a given {\Qqcms{F}} $(\A,\Lip_\A)$ may well be infinite, i.e. there may be no finite dimensional approximations, at least for small values of $\varepsilon > 0$.

However, when the covering number is indeed finite, we get the following analogue of Theorem (\ref{Gromov-Compactness-thm}):

\begin{theorem}[\cite{Latremoliere15}, Theorem 4.2]
Let $F$ be a strongly permissible function. Let $\mathcal{A}$ be a nonempty subclass of the closure of $\finitedimclass{F}$ for the dual Gromov-Hausdorff propinquity $\propinquity{F}$. The following two assertions are equivalent:
\begin{enumerate}
\item the class $\mathcal{A}$ is totally bounded for the dual Gromov-Hausdorff propinquity $\propinquity{F}$,
\item there exists a function $C : [0,\infty) \rightarrow \N$ and $D\geq 0$ such that, for all $(\A,\Lip) \in \mathcal{A}$, we have:
\begin{itemize}
\item $\forall\varepsilon>0\quad \covn{F}{\A,\Lip}{\varepsilon} \leq C(\varepsilon)$,
\item $\diam{\A}{\Lip} \leq D$\text{.}
\end{itemize}
\end{enumerate}
In particular, since $\propinquity{F}$ is complete, compact classes of {\Qqcms{F}s} are the closed, totally bounded classes for $\propinquity{F}$.
\end{theorem}

We now must address the question of which {\gQqcms s} are limits of finite dimensional {\gQqcms s}. This matter was in fact the key motivation for the introduction of {\gQqcms s} in our theory, which we originally \cite{Latremoliere13,Latremoliere13b,Latremoliere13c,Latremoliere14} developed for {\Lqcms s}.

\subsubsection{Finite dimensional Approximations}

The field of C*-algebras is quite rich in notions of finite-dimensional approximations in a quantum topological sense: nuclearity, exactness, quasi-diagonality, and AF algebras are important examples. It is natural to ask: is there a way to connect some form of quantum topological finite dimensional approximation with quantum metric finite dimensional approximations? 

Our own research gave us some results in this direction. An appropriate notion of topological finite approximations which we propose is modeled after quasi-diagonality together with nuclearity.

\begin{definition}[\cite{Latremoliere15}, Definition 5.1]
A unital C*-algebra $\A$ is A \emph{pseudo-diagonal} when, for all finite subset $\alg{F}$ of $\A$ and for all $\varepsilon > 0$, there exists a finite dimensional C*-algebra $\B$ and two positive, unital maps $\psi: \A\rightarrow\A$ and $\varphi:\A\rightarrow\B$ such that:
\begin{enumerate}
\item for all $a\in\alg{F}$ we have $\|a-\varphi\circ\psi(a)\|_\A \leq\varepsilon$,
\item for all $a,b\in\alg{F}$ we have:
\begin{equation*}
\|\psi(a)\psi(b)-\psi(ab)\|_\B \leq \varepsilon\text{.}
\end{equation*}
\end{enumerate}
\end{definition}

Pseudo-diagonality does not involve completely positive maps, but it involves unital maps. Our concept is inspired by the characterization of nuclear quasi-diagonal C*-algebras of Blackadar and Kirchberg \cite{Blackadar97}:
\begin{theorem}
A C*-algebra $\A$ is nuclear, quasi-diagonal if and only if for every $\varepsilon > 0$ and for every finite set $F$ of $\A$, there exists a finite dimensional C*-algebra $\B$ and two completely positive contractions $\varphi : \A\rightarrow\B$ and $\psi: \B\rightarrow\A$ such that:
\begin{enumerate}
\item for all $a\in F$ we have $\|a-\psi\circ\varphi(a)\|_\A \leq \varepsilon$,
\item for all $a, b\in F$ we have $\|\varphi(ab) - \varphi(a)\varphi(b)\|_\B \leq \varepsilon$.
\end{enumerate}
\end{theorem}

A little work allows us to prove that:

\begin{theorem}[\cite{Latremoliere15}, Corollary 5.5]
A unital nuclear quasi-diagonal C*-algebra $\A$ is pseudo-diagonal.
\end{theorem}

The importance of pseudo-diagonal {\Lqcms s} is that they admit finite dimensional approximations for the dual Gromov-Hausdorff propinquity, albeit at the cost of relaxing the Leibniz property a little. Formally, we introduce a small variation on the Leibniz property, whose role in our approximation theorem was the motivation to extend the dual Gromov-Hausdorff propinquity to {\gQqcms s}.

\begin{notation}\label{Qqcms-CD-notation}
Let $C\geq 1$ and $D\geq 0$. Let:
\begin{equation*}
F_{C,D} : x,y,l_x,l_y \in [0,\infty) \mapsto C(x l_y + y l_x) + D l_x l_y \text{.}
\end{equation*}
A {\Qqcms{F_{C,D}}} is called a {\Qqcms{(C,D)}}.
\end{notation}

We proved the following approximation result:

\begin{theorem}[\cite{Latremoliere15}, Theorem 5.7]\label{approx-thm}
Let $C\geq 1$ and $D\geq 0$. If $(\A,\Lip)$ is a {\Qqcms{(C,D)}} and $\A$ is pseudo-diagonal, then for any $\varepsilon > 0$, there exists a sequence $\left((\A_n,\Lip_n)\right)_{n\in\N}$ of {\Qqcms{(C+\varepsilon, D+\varepsilon)}s} such that:
\begin{enumerate}
\item for all $n\in\N$, the C*-algebra $\A$ is finite dimensional,
\item we have:
\begin{equation*}
\lim_{n\rightarrow\infty} \propinquity{(C+\varepsilon, D+\varepsilon)} \left((\A_n,\Lip_n), (\A,\Lip)\right) = 0\text{.}
\end{equation*}
\end{enumerate}
\end{theorem}

An important observation is that if $C=1$, $D=0$ then Theorem (\ref{approx-thm}) gives finite dimensional approximations of {\Lqcms s} by {\Qqcms{(1+\varepsilon, \varepsilon)}s} for any $\varepsilon > 0$, but not for $\varepsilon = 0$ in general. The difficulty which we encountered occurred was to define a Lip-norm on the finite dimensional approximations provided by the pseudo-diagonal property: a simple quotient would not work, as we would have no control over then Leibniz property. Our new approach gives a better result, as we find approximations which are ``as close to Leibniz'' as possible, though not Leibniz, by constructing our Lip-norms on the finite dimensional algebras thanks to the maps from pseudo-diagonality, in a slightly tricky way.

We shall see later in this document that for some specific examples, such as quantum tori, we do manage to obtain finite dimensional {\Lqcms s} approximations.

We now turn to a natural question: how does one construct tunnels? A special form of tunnels is in fact the basis for our original construction of the quantum propinquity, which can now be seen as a specialization of the dual propinquity.

\subsection{The Quantum Gromov-Hausdorff Propinquity}

\subsubsection{Bridges and Treks}

The quantum Gromov-Hausdorff propinquity \cite{Latremoliere13} is a specialization of the dual Gromov-Hausdorff propinquity \cite{Latremoliere13c}, although we discovered it first, and it plays an important role in the proof of the convergence of several examples of {\Lqcms s} for the dual propinquity. At the core of the quantum Gromov-Hausdorff propinquity is the concept of a \emph{bridge}, which is a natural source of Leibniz Lip-norms.

Indeed, a mean to get seminorms with the Leibniz property is to use derivations, as for instance in Example (\ref{module-LP-ex}). The quantum propinquity specifically uses bounded inner derivations in C*-algebras. The key ingredient is the notion of a bridge.

\begin{figure}[h]\label{bridge-fig}
\begin{equation*}
\xymatrix{
 & (\D,\omega \in \D)  & \\
\A \ar@{^{(}->}[ur]^{\pi_\A} & & \B \ar@{_{(}->}[ul]_{\pi_\B}
}
\end{equation*}
\caption{A bridge}
\end{figure}
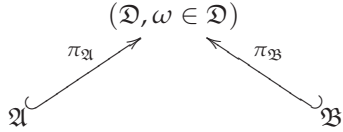

\begin{definition}[\cite{Latremoliere13}, Definition 3.1]
Let $\D$ be a unital C*-algebra and $\omega\in\D$. The $1$-level set of $\omega$ is:
\begin{equation*}
\StateSpace(\A|\omega) = \left\{ \varphi \in \StateSpace(\A) \middle\vert \forall d \in \D \quad \begin{cases}
\varphi((\unit_\A-\omega)^\ast (\unit_\A-\omega)) = 0\\
\varphi((\unit_\A-\omega)(\unit_\A-\omega)^\ast) = 0
 \end{cases} \right\}\text{.}
\end{equation*}
\end{definition}

\begin{lemma}[\cite{Latremoliere13}, Lemma 3.4]
Let $\D$ be a unital C*-algebra. If $\omega \in \D$ then:
\begin{equation*}
\StateSpace(\A|\omega) = \left\{ \varphi\in\StateSpace(\D) : \forall d\in\D \quad \varphi(d) = \varphi(d\omega)=\varphi(\omega d) \right\}\text{.}
\end{equation*}
\end{lemma}

\begin{definition}[\cite{Latremoliere13}, Definition 3.6]
Let $\A$ and $\B$ be two unital C*-algebras. A \emph{bridge} $\gamma = (\D,\omega,\pi_\A,\pi_\B)$ is given by a unital C*-algebra, two unital *-mono\-mor\-phi\-sms $\pi_\A : \A \hookrightarrow \D$ and $\pi_\B : \B \hookrightarrow \D$, and $\omega \in \D$ such that the $1$-level set $\StateSpace(\A|\omega)$ of $\omega$ is not empty.
\end{definition}

\begin{notation}
When $\gamma = (\D,\omega,\pi,\rho)$ is a bridge, $\omega$ is called the pivot of $\gamma$, the domain of $\pi$ is the called the domain $\dom{\gamma}$ of $\gamma$ while the domain of $\rho$ is called the codomain $\codom{\gamma}$ of $\gamma$.
\end{notation}

The role of the pivot is illustrated in our next section on convergence for the quantum tori. When pivot are restricted to always be the unit, then our construction would lead to the unital nuclear distance \cite{kerr09}; our metric is however more flexible to work with.

Bridges will allow us to construct tunnels. Of course, we wish to be able to compute the length of such tunnels from the data provided by the bridge. As with tunnels, we associate two natural numbers to bridges between quantum compact metric spaces, which will then allow us to compute their length.

The first of these quantities measure how far the domain and the codomain of a bridge are apart, using a metric given by the inner derivation defined by the pivot. Formally, we define:

\begin{definition}[\cite{Latremoliere13}, Definition 3.10]
Let $\A$ and $\B$ be two unital C*-algebras and let $\gamma = (\D,\omega,\pi_\A,\pi_\B)$ be a bridge from $\A$ to $\B$. The \emph{bridge seminorm} $\bridgenorm{\gamma}{\cdot}$ of $\gamma$ is given for all $a\in\A$ and $b\in \B$ by:
\begin{equation*}
\bridgenorm{\gamma}{a,b} = \|\pi_\A(a)\omega - \omega \pi_\B(b)\|_\D\text{.}
\end{equation*}
\end{definition}

We can now define the reach of a bridge in terms of the bridge seminorm (which, we note, is indeed a seminorm on $\A\oplus\B$).

\begin{definition}[\cite{Latremoliere13}, Definition 3.14]
Let $(\A,\Lip_\A)$ and $(\B,\Lip_\B)$ be two quantum compact metric spaces and:
\begin{equation*}
\gamma = (\D,\omega,\pi_\A,\pi_\B)
\end{equation*}
be a bridge from $(\A,\Lip_\A)$ to $(\B,\Lip_\B)$. The \emph{reach} $\bridgereach{\gamma}{\Lip_\A,\Lip_\B}$ of $\gamma$ with respect to $\Lip_\A,\Lip_\B$ is the non-negative real number:
\begin{multline*}
\max\{ \sup\{ \inf \{ \bridgenorm{\gamma}{a,b} : b\in\sa{\B}, \Lip_\B(b)\leq 1\} : a\in\sa{\A}, \Lip_\A(a) \leq 1 \} ,\\
\sup \{ \inf \{ \bridgenorm{\gamma}{a,b} : a \in \sa{\A}, \Lip_\A(a)\leq 1\} : b \in\sa{\B}, \Lip_\B(b) \leq 1 \} \}\text{.}
\end{multline*}
\end{definition}

Thus the reach is the Hausdorff distance between Lipschitz balls for the seminorm $\bridgenorm{\gamma}{\cdot,\cdot}$ on $\A\oplus\B$. This quantity is always finite by \cite[Lemma 3.15]{Latremoliere13}.

The reach of a bridge represents, metaphorically, the length of its span between its domain and codomain. However, one must first get on the bridge and then, once the span crossed, get off the bridge: in some sense, the span measures how far the images of the $1$-level set of the pivot in the state spaces of the domain and the codomain are from each other, and we must now measure how far these images are from being the entire state spaces in each of the domain and codomain. This number is provided by the height of the bridge:

\begin{definition}[\cite{Latremoliere13}, Definition 3.16]
Let $(\A,\Lip_\A)$ and $(\B,\Lip_\B)$ be two quantum compact metric spaces and:
\begin{equation*}
\gamma = (\D,\omega,\pi_\A,\pi_\B)
\end{equation*}
be a bridge from $(\A,\Lip_\A)$ to $(\B,\Lip_\B)$. The \emph{height} $\bridgeheight{\gamma}{\Lip_\A,\Lip_\B}$ of $\gamma$ with respect to $\Lip_\A, \Lip_\B$ is the non-negative real number:
\begin{multline*}
\max\{ \Haus{\Kantorovich{\Lip_\A}}\left(\StateSpace(\A), \pi_\A^\ast\left(\StateSpace(\D|\omega)\right)\right), \Haus{\Kantorovich{\Lip_\B}}\left(\StateSpace(\B), \pi_\B^\ast\left(\StateSpace(\D|\omega)\right)\right) \}\text{.}
\end{multline*}
\end{definition}

We now bring these two quantities together:

\begin{definition}[\cite{Latremoliere13}, Definition 3.17]
Let $(\A,\Lip_\A)$ and $(\B,\Lip_\B)$ be two quantum compact metric spaces and:
\begin{equation*}
\gamma = (\D,\omega,\pi_\A,\pi_\B)
\end{equation*}
be a bridge from $(\A,\Lip_\A)$ to $(\B,\Lip_\B)$. The \emph{length} $\bridgelength{\gamma}{\Lip_\A,\Lip_\B}$ of $\gamma$ with respect to $\Lip_\A, \Lip_\B$ is the non-negative real number:
\begin{equation*}
\max\{ \bridgereach{\gamma}{\Lip_\A,\Lip_\B}, \bridgeheight{\gamma}{\Lip_\A,\Lip_\B} \}\text{.}
\end{equation*}
\end{definition}

A bridge is a mean to define a special kind of tunnel, which is very useful in practice \cite{Rieffel11, Latremoliere13b, Rieffel15}. The bridge itself has no quantum metric structure, which is why its reach and height must be decorated with the metric structures of its domain and codomain. This is in sharp contrast with tunnels, which do carry their own Lip-norms. Moreover, Figure (\ref{bridge-fig}) is, in a sense, backward if compared to Figure (\ref{tunnel-fig}) --- it is not the dual picture to Figure (\ref{GH-fig})! This inversion may appear counter-intuitive, though the next theorem should clarify this matter. The dual relationship from bridges to tunnels (which has no known inverse) justifies the name \emph{dual} Gromov-Hausdorff propinquity.

\begin{theorem}[\cite{Latremoliere13}, Theorem 6.3]\label{tunnel-from-bridge-thm}
Let $F$ be a permissible function, and let $(\A,\Lip_\A)$ and $(\B,\Lip_\B)$ be two {\Qqcms{F}s}. Let:
\begin{equation*}
\gamma = (\D,\omega,\pi_\A,\pi_\B)
\end{equation*}
be a bridge from $\A$ to $\B$. For all $\lambda > \bridgelength{\gamma}{\Lip_\A,\Lip_\B}$, and for all $a\in\sa{\A}$ and $b\in\sa{\B}$, we define:
\begin{equation*}
\Lip_\lambda (a,b) = \max\left\{ \Lip_\A(a), \Lip_\B(b), \frac{1}{\lambda} \bridgenorm{\gamma}{a,b} \right\}\text{.}
\end{equation*}
Let $\iota_\A : \A\oplus\B \twoheadrightarrow \A$ and $\iota_\B : \A\oplus\B \twoheadrightarrow \B$ be the two canonical surjections. Then $(\A\oplus\B, \Lip_\lambda, \iota_\A, \iota_\B)$ is an $F$-tunnel from $(\A,\Lip_\A)$ to $(\B,\Lip_\B)$ of length at most $\lambda$. 
\end{theorem}

It is therefore natural to define, for any permissible function $F$, the class $\mathcal{T}_F$ of all $F$-tunnels obtained from the construction in Theorem (\ref{tunnel-from-bridge-thm})  and use it to construct a specialized version of the dual Gromov-Hausdorff propinquity. Unfortunately, $\mathcal{T}_F$ is not triangular. 

This difficulty is in fact quite important from the perspective of the development of a noncommutative analogue of the Gromov-Hausdorff distance. Indeed, Theorem (\ref{tunnel-composition-thm}) allows to compose tunnels, but we see immediately that this approach, when applied to two tunnels in $\mathcal{T}_F$, will not lead to a tunnel in $\mathcal{T}_F$. A simple observation is that indeed, we would require taking a quotient in order to return the composed tunnel in $\mathcal{T}_F$. Yet this procedure would, in general, destroy the Leibniz property. This exact difficulty has prevented many earlier metrics \cite{kerr02, Rieffel10c} to be well-behaved with respect to the Leibniz property because the triangle inequality may fail.

We however developed another approach to the construction of both the quantum propinquity and the dual propinquity which allows us to still obtain well-behaved metrics over classes of tunnels which may not be appropriate, such as $\mathcal{T}_F$. These constructions rely on the notion of a finite path between quantum compact metric spaces, consisting of bridges or tunnels. In \cite{Latremoliere13}, we originally introduced treks:

\begin{definition}[\cite{Latremoliere13}, Definition 3.20]
Let $F$ be a permissible function. Let $\mathcal{C}$ be a nonempty class of {\Qqcms{F}s} and let $(\A,\Lip_\A), (\B,\Lip_\B) \in \mathcal{C}$. A \emph{$\mathcal{C}$-trek} from $(\A,\Lip_\A)$ to $(\B,\Lip_\B)$ is a finite family:
\begin{equation*}
\Gamma = \left( \A_j,\Lip_k, \gamma_j, \A_{j+1}, \Lip_{j+1} : j = 1,\ldots,n \right)
\end{equation*}
where:
\begin{enumerate}
\item for all $j\in\{1,\ldots,n+1\}$ we have $(\A_j,\Lip_j) \in \mathcal{C}$,
\item we have $(\A_1,\Lip_1) = (\A,\Lip_\A)$ and $(\A_{n+1}, \Lip_{n+1}) = (\B,\Lip_\B)$,
\item for all $j\in \{1,\ldots,n\}$, we are given a bridge $\gamma_j$ from $(\A_j,\Lip_j)$ to $(\A_{j+1},\Lip_{j+1})$.
\end{enumerate}
\end{definition}

Since each bridge of a trek gives rise to a tunnel, we have a natural notion of a journey as well. 

\begin{definition}[\cite{Latremoliere13c}, Definition 3.18]
Let $F$ be a permissible function, $\mathcal{C}$ be a nonempty class of {\Qqcms{F}s} and $\mathcal{T}$ be a nonempty class of tunnels. Let $(\A,\Lip_\A)$ and $(\B,\Lip_\B)$ in $\mathcal{C}$. 

A \emph{$\mathcal{T}$-journey} from  $(\A,\Lip_\A)$ to $(\B,\Lip_\B)$ is a finite family:
\begin{equation*}
\Upsilon = \left( \A_j,\Lip_j,\tau_j,\A_{j+1},\Lip_{j+1} : j = 1,\ldots, n \right)
\end{equation*}
where:
\begin{enumerate}
\item for all $j\in\{1,\ldots,n+1\}$ we have $(\A_j,\Lip_j) \in \mathcal{C}$,
\item we have $(\A_1,\Lip_1) = (\A,\Lip_\A)$ and $(\A_{n+1}, \Lip_{n+1}) = (\B,\Lip_\B)$,
\item for all $j\in \{1,\ldots,n\}$, we are given a tunnel:
\begin{equation*}
\gamma_j\in\tunnelset{(\A_j,\Lip_j)}{(\A_{j+1},\Lip_{j+1})}{\mathcal{T}}\text{.}
\end{equation*}
\end{enumerate}
\end{definition}

In order for the construction of the dual propinquity based on the length of tunnels to lead to a metric, we require some compatibility condition on what classes of tunnels may be used to build journeys. The difference between this notion of compatibility and the notion of an appropriate class is that we do not require the class to be triangular: the introduction of journeys and treks provide a natural notion of composition from which the triangle inequality will follow. We also relax the notion of connectedness.

\begin{definition}[\cite{Latremoliere13c}, Definition 3.11]
Let $F$ be a permissible function, $\mathcal{C}$ be a nonempty class of {\Qqcms{F}s}. A class $\mathcal{T}$ of tunnels is \emph{$\mathcal{C}$-compatible} when:
\begin{enumerate}
\item $\mathcal{T}$ is weakly connected: For any $\mathds{A}, \mathds{B} \in \mathcal{C}$, there exists a $\mathcal{T}$-journey from $\mathds{A}$ to $\mathds{B}$,
\item $\mathcal{T}$ is symmetric: if $\tau = (\D,\Lip_\D,\pi,\rho) \in \mathcal{T}$ then $\tau^{-1} = (\D,\Lip_\D,\rho,\pi) \in \mathcal{T}$,
\item $\mathcal{T}$ is specific: if $\tau \in \mathcal{T}$ then the domain and codomain of $\tau$ lies in $\mathcal{C}$,
\item $\mathcal{T}$ is definite: for any $(\A,\Lip_\A), (\B,\Lip_\B) \in \mathcal{C}$, if there exists an isometric *-isomorphism $h : \A\rightarrow\B$ then both $(\A, \Lip_\A, \mathrm{id}_\A, h^{-1})$ and $(\B,\Lip_\B,h,\mathrm{id}_\B)$ belong to $\mathcal{T}$, where $\mathrm{id}_E$ is the identity map of the set $E$ for any set.
\end{enumerate}
\end{definition}

In particular, for any permissible function $F$, if $\mathcal{C}_F$ is the class of all {\Qqcms{F}s}, then the class $\mathcal{T}_F$ of all tunnels build by Theorem (\ref{tunnel-from-bridge-thm}) from bridges between arbitrary {\Qqcms{F}s} is $\mathcal{C}_F$-compatible.

Treks and journeys can be composed by concatenation. Now, we define the length of both these types of paths between quantum compact metric spaces:

\begin{definition}[\cite{Latremoliere13},  Definition 3.22; \cite{Latremoliere13c}, Definition 3.20]
The \emph{length} of a trek:
\begin{equation*}
\Gamma = \left(\A_j,\Lip_j,\gamma_j,\A_{j+1},\Lip_{j+1} : j=0,\ldots,n\right)
\end{equation*}
is the non-negative real number:
\begin{equation*}
\treklength{\Gamma} = \sum_{j=1}^n \bridgelength{\gamma_j}{\Lip_j,\Lip_{j+1}} \text{.}
\end{equation*}
Similarly, the \emph{length} of a journey:
\begin{equation*}
\Upsilon = \left(\A_j,\Lip_j,\gamma_j,\A_{j+1},\Lip_{j+1} : j=0,\ldots,n\right)
\end{equation*}
is the non-negative real number:
\begin{equation*}
\journeylength{\Upsilon} = \sum_{j=1}^n \tunnellength{\tau_j} \text{.}
\end{equation*}
\end{definition}

With these concepts defined, we now introduce the quantum propinquity.

\subsubsection{The Quantum Propinquity}

\begin{notation}
Let $\mathcal{C}$ be a nonempty class of {\Qqcms{F}s} for some permissible function $F$. The set of all $\mathcal{C}$-treks from $(\A,\Lip_\A) \in \mathcal{C}$ to $(\B,\Lip_\B) \in \mathcal{C}$ is denoted by:
\begin{equation*}
\trekset{\A,\Lip_\A}{\B,\Lip_\B}{\mathcal{C}}\text{.}
\end{equation*}
\end{notation}

\begin{definition}[\cite{Latremoliere13}, Definition 4.2]\label{quantum-propinquity-def}
Let $F$ be a permissible function. Let $\mathcal{C}$ be a nonempty class of {\Qqcms{F}s}. The \emph{quantum Gromov-Hausdorff $\mathcal{C}$-propinquity} $\qpropinquity{\mathcal{C}}((\A,\Lip_\A),(\B,\Lip_\B))$ between $(\A,\Lip_\A)\in \mathcal{C}$ and $(\B,\Lip_\B)\in \mathcal{C}$ is:
\begin{equation*}
\inf\left\{ \treklength{\Gamma} : \Gamma \in \trekset{\A,\Lip_\A}{\B,\Lip_\B}{\mathcal{C}} \right\}\text{.}
\end{equation*}
\end{definition}

The main results in \cite{Latremoliere13} are that the quantum propinquity is indeed a metric up to isometric *-isomorphism, that it induces the same topology as the Gromov-Hausdorff distance on classical metric spaces, and dominates Rieffel's quantum Gromov-Hausdorff distance.

\begin{theorem}[\cite{Latremoliere13}, Proposition 4.6, Proposition 4.7, Theorem 5.13]
Let $F$ be a permissible function. Let $\mathcal{C}$ be a nonempty class of {\Qqcms{F}s}. Then, for all $(\A,\Lip_\A), (\B,\Lip_\B), (\D,\Lip_\D) \in \mathcal{C}$, we have:
\begin{enumerate}
\item $\qpropinquity{\mathcal{C}}((\A,\Lip_\A),(\B,\Lip_\B)) \leq \max\{\diam{\StateSpace(\A)}{\Kantorovich{\Lip_\A}}, \diam{\StateSpace(\B)}{\Kantorovich{\Lip_\B}} \}$,
\item $\qpropinquity{\mathcal{C}}((\A,\Lip_\A),(\B,\Lip_\B)) = \qpropinquity{\mathcal{C}}((\B,\Lip_\B),(\A,\Lip_\A))$,
\item we have:
\begin{equation*}
\qpropinquity{\mathcal{C}}((\A,\Lip_\A),(\B,\Lip_\B)) \leq \qpropinquity{\mathcal{C}}((\A,\Lip_\A),(\D,\Lip_\D)) + \qpropinquity{\mathcal{C}}((\D,\Lip_\D),(\B,\Lip_\B))\text{,}
\end{equation*}
\item $\qpropinquity{\mathcal{C}}((\A,\Lip_\A),(\B,\Lip_\B)) = 0$ if and only if there exists a *-isomorphism $\pi : \A\rightarrow \B$ such that $\Lip_\B\circ h = \Lip_\A$.
\end{enumerate}
\end{theorem}

The comparison theorem for the quantum propinquity is given as:

\begin{theorem}[\cite{Latremoliere13}, Corollary 6.4, Theorem 6.6]
Let $F$ be a permissible function. Let $\mathcal{C}$ be a nonempty class of {\Qqcms{F}s}. Then, for all $(\A,\Lip_\A), (\B,\Lip_\B) \in \mathcal{C}$, we have:
\begin{equation*}
\dist((\A,\Lip_\A),(\B,\Lip_\B)) \leq 2 \qpropinquity{\mathcal{C}}((\A,\Lip_\A),(\B,\Lip_\B)) \text{.}
\end{equation*}
Moreover, if $(X,\mathrm{d}_X)$ and $(Y,\mathrm{d}_Y)$ are compact metric spaces, and if $\mathrm{Lip}_X$ and $\mathrm{Lip}_Y$ are the two Lipschitz seminorms associated, respectively, to $\mathrm{d}_X$ and $\mathrm{d}_Y$, then:
\begin{equation*}
\qpropinquity{\mathcal{C}}((C(X),\mathrm{Lip}_X,C(Y),\mathrm{Lip}_Y)) \leq \GH((X,\mathrm{d}_X), (Y,\mathrm{d}_Y)) \text{.}
\end{equation*}
In particular, the topology induced by the quantum Gromov-Hausdorff propinquity on the class of compact metric spaces agrees with the topology of the Gromov-Hausdorff distance.
\end{theorem}

As we shall discuss in the next section, examples of convergence for the dual propinquity come often from convergence in the sense of the quantum propinquity. Indeed, these two metrics are comparable --- as Theorem (\ref{tunnel-from-bridge-thm}) strongly suggests. As we mentioned earlier, the issue is to define the dual Gromov-Hausdorff propinquity so that we may use compatible, rather than appropriate classes of tunnels. Indeed, using journeys, one may propose an alternative form of the dual Gromov-Hausdorff propinquity (and its original form in \cite{Latremoliere13c}), which is really an equivalent metric whenever both constructions can be carried out:

\begin{definition}[\cite{Latremoliere13c}]\label{propinquity-alt-def}
Let $F$ be a permissible function, $\mathcal{C}$ be a nonempty class of {\Qqcms{F}s} and $\mathcal{T}$ be a $\mathcal{C}$-compatible class of tunnels. If $(\A,\Lip_\A),(\B,\Lip_\B) \in \mathcal{C}$ then the \emph{dual propinquity} (alternate version) $\propinquity{\mathcal{T},\mathrm{alt}}((\A,\Lip_\A),(\B,\Lip_\B))$ is the non-negative real number:
\begin{equation*}
\inf\left\{ \journeylength{\Upsilon} : \Upsilon \in \journeyset{\A,\Lip_\A}{\B,\Lip_\B}{\mathcal{T}} \right\}\text{.}
\end{equation*}
\end{definition}

We proved directly in \cite{Latremoliere13c} that $\propinquity{\mathrm{alt}}$ is also a metric, and discussed the equivalence of our constructions in \cite{Latremoliere14}:

\begin{theorem}[\cite{Latremoliere13c}, Theorem 4.16, Theorem 4.17, \cite{Latremoliere14}]\label{alt-propinquity-thm}
Let $F$ be a permissible function. Let $\mathcal{C}$ be a nonempty class of {\Qqcms{F}s} and let $\mathcal{T}$ be a $\mathcal{C}$-compatible class of tunnels. Then, for all $(\A,\Lip_\A), (\B,\Lip_\B), (\D,\Lip_\D) \in \mathcal{C}$, we have:
\begin{enumerate}
\item We have:
\begin{multline*}
\propinquity{\mathcal{T},\mathrm{alt}}((\A,\Lip_\A),(\B,\Lip_\B)) \\ \leq \max\{\diam{\StateSpace(\A)}{\Kantorovich{\Lip_\A}}, \diam{\StateSpace(\B)}{\Kantorovich{\Lip_\B}} \}\text{,}
\end{multline*}
\item $\propinquity{\mathcal{T},\mathrm{alt}}((\A,\Lip_\A),(\B,\Lip_\B)) = \propinquity{\mathcal{T},\mathrm{alt}}((\B,\Lip_\B),(\A,\Lip_\A))$,
\item we have:
\begin{equation*}
\propinquity{\mathcal{T},\mathrm{alt}}((\A,\Lip_\A),(\B,\Lip_\B)) \leq \propinquity{\mathcal{T},\mathrm{alt}}((\A,\Lip_\A),(\D,\Lip_\D)) + \propinquity{\mathcal{T},\mathrm{alt}}((\D,\Lip_\D),(\B,\Lip_\B))\text{,}
\end{equation*}
\item $\propinquity{\mathcal{T},\mathrm{alt}}((\A,\Lip_\A),(\B,\Lip_\B)) = 0$ if and only if there exists a *-isomorphism $\pi : \A\rightarrow \B$ such that $\Lip_\B\circ h = \Lip_\A$.
\item if $\mathcal{T}_{\mathcal{C}}$ is the class of all tunnels between elements of $\mathcal{C}$:
\begin{equation*}
\dist((\A,\Lip_\A),(\B,\Lip_\B)) \leq \propinquity{\mathcal{T}_{\mathcal{C},\mathrm{alt}}}((\A,\Lip_\A),(\B,\Lip_\B)) \leq \qpropinquity{\mathcal{C}}((\A,\Lip_\A,\B,\Lip_\B))\text{,}
\end{equation*}
\item If moreover, $\mathcal{T}$ is $\mathcal{C}$-appropriate, i.e. it is also triangular, then:
\begin{equation*}
\propinquity{\mathcal{T},\mathrm{alt}} ((\A,\Lip_\A),(\B,\Lip_\B)) \leq \propinquity{\mathcal{T}}((\A,\Lip_\A),(\B,\Lip_\B)) \leq 2\propinquity{\mathcal{T},\mathrm{alt}}((\A,\Lip_\A),(\B,\Lip_\B))\text{.}
\end{equation*}
\end{enumerate}
\end{theorem}

Thus, the dual Gromov-Hausdorff propinquity can be defined for compatible classes of tunnels. In particular, for any permissible $F$, we can check that if $\mathcal{T}_F$ is the class of all $F$-tunnels obtained from bridges via Theorem (\ref{tunnel-from-bridge-thm}), which is compatible with the class of all {\Qqcms{F}s}, then:
\begin{equation*}
\qpropinquity{F} = \propinquity{\mathcal{T}_F}\text{.}
\end{equation*}

\bigskip

In summary, the dual Gromov-Hausdorff propinquity is a complete metric on the class of {\Lqcms s}, and even {\Qqcms{F}s} for a given permissible function $F$, i.e. an appropriate choice of a quasi-Leibniz relation. Moreover, the construction of the dual Gromov-Hausdorff propinquity can be specialized to various subcategories of {\Lqcms s} or {\Qqcms{F}s}. Among these specializations, the most important is also the original metric which we introduced, the quantum Gromov-Hausdorff propinquity. The quantum propinquity answers long standing questions regarding the computation of upper bounds for Rieffel's Gromov-Hausdorff distance when given structures we call bridges, which are a useful source of Leibniz Lip-norms. Moreover, the very specific form of the Leibniz Lip-norms coming from bridges allow for algebraic manipulations, which proved useful when extending certain convergence result to matrix algebras over convergent sequences of {\Lqcms s} \cite{Rieffel15}. Notably, the dual Gromov-Hausdorff propinquity admits an alternative construction, up to equivalence and up to a mild restriction on the possible choices of tunnels one may invoke. This alternative construction using extents instead of length, is an elegant way to handle the difficulties attached to the triangle inequality. However, it is not compatible with the construction of the quantum propinquity. The main lesson from these many constructions is that our approach is, in fact, very flexible and thus more likely to provide the framework for future research about noncommutative analogues of the Gromov-Hausdorff distance.

We now turn to examples of convergence for our new metrics, starting with the fundamental example of quantum tori. This example has played a central role in our work and the general development of the theory of noncommutative metric geometry, and even noncommutative geometry.

\subsection{Quantum Tori}

\subsubsection{Background}

This preliminary subsection contains a brief summary of the various facts and notations we will use in our work with quantum and fuzzy tori. 

\begin{theorem-definition}
Let $G$ be an Abelian discrete group, $\T = \{z\in\C :|z|=1\}$ and $\sigma : G \rightarrow \T$ be a $2$-cocyle over $G$, i.e. for all $x, y, z \in G$ we have:
\begin{equation*}
\sigma(x,y)\sigma(x+y,z) = \sigma(x,y+z) \sigma(y,z)\text{.}
\end{equation*}
For any two functions $f ,g : G\rightarrow \C$ with finite support, we define:
\begin{equation*}
f \ast_\sigma g : x\in G \mapsto \sum_{y \in G} f(y) g(x-y) \sigma(y, x-y)\text{.}
\end{equation*}
The vector space $C_c(G)$ of $\C$-valued functions over $G$ with finite support is an associative $*$-algebra for the multiplication $\ast_\sigma$ and the adjoint operation defined, for all $f \in C_c(G)$ and $x\in G$ by:
\begin{equation*}
f^\ast (x) = \overline{f(-x)}\text{.}
\end{equation*}
\end{theorem-definition}

\begin{notation}
For any set $E$, the Hilbert space $\ell^2(E)$ is the space $\{ (x_g)_{g\in E} : \sum_{g \in E} |x_e|^2 < \infty\}$ equipped with the pointwise addition and scalar multiplication. Note that the sum notation is meant for the notion of a summable family.

For each $g \in E$ we denote by $\delta_g$ the element of $\ell^2(G)$ defined as the family which is zero for all indices in $E$ except at $g$, where it is one.
\end{notation}

\begin{theorem-definition}\label{twisted-cstar-def}
Let $G$ be an Abelian discrete group and let $\sigma$ be a multiplier over $G$. For every $g \in G$, we define:
\begin{equation*}
U^g : \delta_h = \sigma(h,g-h) U^{g-h}\text{.}
\end{equation*}
Then $g \in G \mapsto U^g$ is a $\sigma$-projective representation of $G$ on $\ell^2(G)$, i.e.
\begin{equation*}
U^g U^h = \sigma(g,h) U^{g+h}\text{.}
\end{equation*}

The map $\pi : f \in C_c(G) \mapsto \sum_{g \in G} f(g) U^g$ is a faithful *-representation of $(C_c(G),\ast_\sigma)$ on $\ell^2(G)$.

The twisted group C*-algebra $C^\ast(G,\sigma)$ of discrete Abelian group $G$ for a multiplier $\sigma$  of $G$ is the norm closure of $\pi(C_c(G))$.

The C*-algebra $C^\ast(G,\sigma)$ enjoys two important properties:
\begin{enumerate}
\item\emph{Universality}: if $g\in G\mapsto V^g$ is a $\sigma$-projective representation of $G$ on some Hilbert space $\Hilbert$, then there exists a *-morphism $\rho : C^\ast(G,\sigma) \rightarrow C^\ast(\{V^g : g\in G\})$ such that for all $g\in G$ we have $\rho(U^g) = V^g$,
\item if $\eta$ is a multiplier of $G$ which is cohomologous to $\sigma$ then $C^\ast(G,\eta)$ and $C^\ast(G,\sigma)$ are *-isomorphic.
\end{enumerate}
\end{theorem-definition}
Whenever convenient, we will identity $C_c(G)$ with a dense subset of $C^\ast(G,\sigma)$ for any discrete Abelian group $G$ and multiplier $\sigma$ of $G$.

\begin{theorem-definition}
Let $G$ be a compact Abelian group and let $\widehat{G}$ be its Pontryagin dual group. Let $\sigma$ be a multiplier of $\widehat{G}$. For all $f \in C_c(\widehat{G})$, $g \in G$ and $\chi \in \widehat{G}$, we define:
\begin{equation*}
\alpha^g(f) (\chi) = \chi(g) f(\chi)\text{.}
\end{equation*}
The action $\alpha$ of $G$ extends to $C^\ast(\widehat{G},\sigma)$ to a strongly continuous action by *-automorphisms, called the \emph{dual action} of $G$ on $C^\ast(\widehat{G},\sigma)$.
\end{theorem-definition}

Therefore, we now have all the needed ingredients to define a structure of {\Lqcms} on twisted group C*-algebras of Abelian discrete groups, using Example (\ref{ergodic-LP-ex}) and Theorem (\ref{Rieffel-ergo-thm}):

\begin{theorem-definition}\label{twisted-lip-def}
Let $G$ be a compact Abelian group endowed with a continuous length function $l$, and $\sigma$ a multiplier of $\widehat{G}$. For all $a\in C^\ast(\widehat{G},\sigma)$ we define:
\begin{equation*}
\Lip(a) = \sup\left\{\frac{\|a-\alpha^g(a)\|_{C^\ast(\widehat{G},\sigma)}}{l(g)} : g\in G\setminus\{1\} \right\}\text{,}
\end{equation*}
where $\alpha$ is the dual action of $G$ on $C^\ast(\widehat{G},\sigma)$ and where $\Lip$ may assume the value $\infty$.

The pair $(C^\ast(G,\sigma),\Lip)$ is a {\Lqcms}.
\end{theorem-definition}

\subsubsection{Finite Dimensional Approximations of Quantum Tori}

We established in \cite{Latremoliere05} the following fundamental example of convergence for Rieffel's quantum Gromov-Hausdorff distance:

\begin{theorem}[\cite{Latremoliere05}, Theorem 3.13]\label{qt-main-1}
Let $G_\infty$ be a compact Abelian group endowed with a continuous length function $l$. Let $(G_n)_{n\in\N}$ be a sequence of compact subgroups of $G_\infty$ converging, for the Hausdorff distance defined by $l$, to $G_\infty$.

If, for all $n\in\N$, we let $\sigma_n$ be a skew-bicharacter of $\widehat{G_\infty}$ which induces a skew-bicharacter of $\widehat{G_n}$ (also denoted by $\sigma_n$) and such that the sequence $(\sigma_n)_{n\in\N}$ converges pointwise to some skew bicharacter $\sigma_\infty$ of $G_\infty$, then:
\begin{equation*}
\lim_{n\rightarrow\infty} \dist((C^\ast(\widehat{G_n}),\sigma_n), \Lip_n), (C^\ast(\widehat{G_\infty},\sigma_\infty),\Lip_\infty)) = 0\text{,}
\end{equation*}
where $\Lip_n$ is, for all $n\in\N\cup\{\infty\}$, the Lip-norm induced on $C^\ast(\widehat{G_n},\sigma_n)$ by the length function $l$ and the dual action of $G_n$.
\end{theorem}

This theorem answers the question raised in the introduction to this chapter: we can approximate quantum tori by other quantum tori or even by matrix algebras. However, the quantum propinquity is a stronger metric, and it is desirable to strengthen Theorem (\ref{qt-main-1}) so that it holds for our new metrics, which are better adapted to the C*-algebraic structures. 

The proof of Theorem (\ref{qt-main-1}) is very involved, and its enhancement even more so. We rather briefly sketch the general idea on how we obtained this result. Our goal is to indicate some of the ideas involved in proving such theorems.

Our proof begins with a first approximation theorem:

\begin{notation}
Let $G$ be an Abelian compact group and $\psi : G\rightarrow\R$ be continuous. Let $\alpha$ be the dual action of $G$ on $C^\ast(\widehat{G},\sigma)$ for some skew-bicharacter $\sigma$ of the Pontryagin dual $\widehat{G}$ of $G$. We define:
\begin{equation*}
\alpha^\phi : a \in C^\ast\left(\widehat{G},\sigma\right) \mapsto \int_{G} \phi(\omega)\alpha^\omega(a)\,d\lambda(\omega)\in C^\ast\left(\widehat{G},\sigma\right)\text{,}\
\end{equation*}
where $\lambda$ is the Haar probability measure on $G$.
\end{notation}

\begin{theorem}[\cite{Latremoliere05}, proof of Proposition 3.8] \label{Fejer-approx-thm}
Let $G_\infty$ be a compact Abelian group, $l$ be a continuous length function on $G_\infty$, and let $(G_n)_{n\in\N}$ be a sequence of compact subgroups of $G_\infty$ which converge for the Hausdorff distance induced by $l$ to $G_\infty$. We denote by $\widehat{G_n}$ the Pontryagin dual groups of $G$ and $G_n$ for all $n\in\Nbar$.

For any $n\in\N\cup\{\infty\}$ and any skew-bicharacter $\sigma$ of $\widehat{G}_n$, we denote the norm of $C^\ast(\widehat{G_n},\sigma)$ by $\|\cdot\|_{n,\sigma}$, and we denote the dual action of $G_n$ on $C^\ast(\widehat{G_n},\sigma)$ by $\alpha_{n,\sigma}$. The Lip-norm defined by Theorem (\ref{twisted-lip-def}) by the action $\alpha_{n,\sigma}$ on $C^\ast(\widehat{G}_n,\sigma)$ and the restriction of $l$ to $G_n$ is denoted by $\Lip_{n,\sigma}$.

If $\varepsilon > 0$, then there exists a positive, continuous function $\phi :G_\infty \rightarrow \R$ and $N\in\N$ such that:
\begin{enumerate}
\item For all $n\geq N$, all skew-bicharacters $\sigma$ of $\widehat{G_n}$ and for all $a\in \sa{\C^\ast(\widehat{G_n},\sigma)}$ we have $\|a-\alpha_{n,\sigma}^\phi (a)\|_{n,\sigma} \leq \varepsilon \Lip_{n,\sigma}(a)$ and $\Lip_{n,\sigma}(\alpha_{n,\sigma}^\phi(a))\leq\Lip_{n,\sigma}(a)$,
\item There exists a finite subset $S$ of $\widehat{G_\infty}$ with $0\in S$ such that, for all $n\geq N$ and any skew-bicharacter $\sigma$ of $\widehat{G_n}$, the restriction of the canonical surjection $q_n: \widehat{G}\rightarrow \widehat{G_n}$ is injective on $S$ and the range of $\alpha_n^\phi$ is the span of $\{ u_{n,\sigma}^p : p \in q_c(S) \}$ where for all $p \in \widehat{G_n}$, the unitary $u_{n,\sigma}^p$ is defined in Theorem-Definition (\ref{twisted-cstar-def}).
\end{enumerate}
\end{theorem}

A key feature of Theorem (\ref{Fejer-approx-thm}) is that, for $n$ large enough and using the same notation as in Theorem (\ref{Fejer-approx-thm}), the range of $\alpha^\phi$ is, in some sense, always the set finite dimensional vector subspace of $C_c(G)$ (neglecting the identifications of that space within each twisted group C*-algebra). Thus, a natural path to explore is to use this common space as a pivot space. Unfortunately, such a space will not be closed under the Jordan or Lie multiplication --- it is only an order-unit subspace. However, Rieffel's distance was defined on the class of order-unit subspaces with Lip-norms, so computations can be carried forward in this setting.

This strategy is indeed the one we applied in \cite{Latremoliere05}. We thus have a finite dimensional space to work on, with a sequence of norms and a sequence of Lip-norms. As it were, the hypothesis of Theorem (\ref{qt-main-1}) precisely ensure that these sequences have the desired convergence properties, thanks to a continuous field structure argument. Namely:

\begin{theorem}[\cite{Latremoliere05}, Corollary 2.9]\label{continuity-thm}
Let $G_\infty$ be an Abelian compact group endowed with a length function $l$, and let $(G_n)_{n\in\N}$ be a sequence of compact subgroups of $G_\infty$ converging to $G_\infty$ for the Hausdorff distance defined by $l$. For all $n\in\N$, let $\sigma_n$ be a skew bicharacter of $\widehat{G_\infty}$ which induces a skew bicharacter on $\widehat{G_n}$, and such that the sequence $(\sigma_n)_{n\in\N}$ converges pointwise to some skew bicharacter $\sigma_\infty$ on $\widehat{G_\infty}$.

Let $\mathscr{G} = \prod_{n\in\N\cup\{\infty\}} \widehat{G_n}$ be endowed with the groupoid structure given by declaring that $(g,g') \in \mathscr{G}^{(2)}$, i.e. $(g,g') \in \mathscr{G}^2$ is composable, if and only if there exists $n\in\N\cup\{\infty\}$ such that $g,g' \in \widehat{G_n}$, in which case, of course, the product of $g$ and $g'$ is simply $gg' \in \widehat{G_n}$. Last, let:
\begin{equation*}
\gamma  : (g,g') \in \mathscr{G}^{(2)} \mapsto \sigma_n(g,g') \text{ if $g \in \widehat{G_n}$ for some $n\in \N\cup\{\infty\}$} \text{.}
\end{equation*}

Then $\left(\left(C^\ast\left(\widehat{G_n},\sigma_n\right) : n \in \N\cup\{\infty\} \right), C^\ast(\mathscr{G},\gamma)\right)$ is a continuous field of C*-algebras.
\end{theorem}

Now, let us use the notations of Theorem (\ref{qt-main-1}) and Theorem (\ref{Fejer-approx-thm}). Let $\varepsilon > 0$. We thus have an $N\in \N$ and finite dimensional subspace $V$ of $\sa{C^\ast(\widehat{G}_\infty,\sigma_\infty)}$, such that we may regard $V$ as a finite dimensional subspace of $\sa{C^\ast(\widehat{G_n},\sigma_n)}$ for $n\geq N$, and thus $V$ comes equipped, for all $n\geq N$, with a norm $\|\cdot\|_{n}$ and a Lip-norm $\Lip_n$. Moreover, from the fact that for all $a\in \sa{C^\ast(\widehat{G_n},\sigma_n)}$ with $\Lip_n(a)\leq 1$, we have:
\begin{equation*}
\|a-\alpha_n^\phi(a)\|\leq \varepsilon\text{,}
\end{equation*}
we infer with a little bit of work that $\dist(C^\ast(\widehat{G_n},\sigma_n), (V,\Lip_n)) \leq \varepsilon$, where we must stress that the notation $(V,\Lip_n)$ must be understood as looking at $V$ with both the norm $\|\cdot\|_n$ and the Lip-norm $\Lip_n$.

From the continuity Theorem (\ref{continuity-thm}), we can deduce that for all $f\in V$, the sequences $(\Lip_n(f))_{n\geq N}$ and $(\|f\|_n)_{n\geq N}$ converge to $\Lip_\infty(f)$ and $\|f\|_\infty$. From this, it is then possible to conclude that $\lim_{n\rightarrow\infty}\dist((V,\Lip_n),(V,\Lip_\infty)) = 0$. To prove this last result, we require the construction of Lip-norms on $V\oplus V$ which arise from continuous fields of states --- they possess no natural connection with any multiplicative structure of the underlying C*-algebras. Each step is quite technically involved.

While all these efforts do allow us to establish Theorem (\ref{qt-main-1}), they are not quite enough to conclude a stronger result about finite dimensional approximations of quantum tori for the quantum propinquity. Indeed, as we saw, they rely on a pivot space which is not a C*-algebra, and Lip-norms which are not possibly Leibniz even if extended somehow to the underlying C*-algebras. Much effort must be done to fix these issues.

One approach, proposed by Li \cite{li03}, uses a rather abstract construction about continuous field subtrivialization. For a continuous field $(\A_x : x\in X ; \Gamma)$ of nuclear C*-algebras over some compact space $X$ and with structure algebra $\Gamma$ \cite{Dixmier}, Blanchard \cite{Blanchard97} proved that there exists a Hilbert space $\Hilbert$ and, for all $x\in X$, a faithful *-representation $\pi_x$ of $A_x$ on $\Hilbert$ such that, for all $\gamma \in \Gamma$, the map $x\in X\mapsto \pi_x(\gamma(x))$ is actually continuous in norm. This very strong result then allows us to proceed from Theorems (\ref{Fejer-approx-thm}) and Theorems (\ref{continuity-thm}) and its corollaries regarding continuity of fields of Lip-norms to prove that bridges of the form:
\begin{equation*}
(\mathscr{B}(\Hilbert), \unit_{\mathscr{B}}, \pi_n, \pi_\infty)
\end{equation*}
have lengths which converge to $0$, where the *-representations $\pi_n$ of $C^\ast(\widehat{G_n},\sigma_n)$ (for $n\in \N\setminus\{\infty\}$) are obtained by applying Blanchard's subtrivialization theorem to the continuous field given by Theorem (\ref{continuity-thm}), and where $\mathscr{B}(\Hilbert)$ is the C*-algebra of all bounded linear operators on $\Hilbert$. We are thus led to:

\begin{theorem}[\cite{Latremoliere13}, Theorem 6.8]\label{qt-main-2}
Let $G_\infty$ be a compact Abelian group endowed with a continuous length function $l$. Let $(G_n)_{n\in\N}$ be a sequence of compact subgroups of $G_\infty$ converging, for the Hausdorff distance defined by $l$, to $G_\infty$.

If, for all $n\in\N$, we let $\sigma_n$ be a skew-bicharacter of $\widehat{G_\infty}$ which induces a skew-bicharacter of $\widehat{G_n}$ (also denoted by $\sigma_n$) and such that the sequence $(\sigma_n)_{n\in\N}$ converges pointwise to some skew bicharacter $\sigma_\infty$ of $G_\infty$, then:
\begin{equation*}
\lim_{n\rightarrow\infty} \qpropinquity{}((C^\ast(\widehat{G_n}),\sigma_n), \Lip_n), (C^\ast(\widehat{G_\infty},\sigma_\infty),\Lip_\infty)) = 0\text{,}
\end{equation*}
where $\Lip_n$ is, for all $n\in\N\cup\{\infty\}$, the Lip-norm induced on $C^\ast(\widehat{G_n},\sigma_n)$ by the length function $l$ and the dual action of $G_n$.
\end{theorem}

We note that, in addition to quantum tori, Theorem (\ref{qt-main-2}) may be applied to show that the family of noncommutative solenoids is continuous as a function from the solenoid group to the twisted C*-algebras of $\Z\left[\frac{1}{p}\right]\times \Z\left[\frac{1}{p}\right]$ --- where elements of the solenoids give rise to skew bicharacters in a natural manner, as discussed in \cite{Latremoliere11c}.

We propose a different and more explicit proof of Theorem (\ref{qt-main-2}) where $G_\infty = \Z^d$ for some $d\in \N\setminus\{0,1\}$ in \cite{Latremoliere13b} based on using the left regular representations given by Theorem (\ref{twisted-cstar-def}), instead of the subtrivialization representations, which may be less natural. Our bridges in this setting are quite different since their pivot are trace class operators. 

While the proof in \cite{Latremoliere13b} is too technical to be summarized effectively, we wish to provide a sense of the use of the pivot element of bridges in the definition of the quantum propinquity. Indeed, in \cite{Latremoliere13b}, we use the pivot element to promote a convergence in the strong operator topology to a convergence in norm.

We begin with some notations taken from \cite{Latremoliere13b}.

\begin{notation}
Let $\N_\ast = \N \setminus \{0,1\}$. Let $\Nbar = \N_\ast\cup\{\infty\}$ be the one-point compactification of $\N_\ast$. For any $d \in \N_\ast$ and $k = (k_1,k_2,\ldots,k_d) \in \Nbar^d$, we set:
\begin{equation*}
k\Z^d = \prod_{j=1}^d k_j \Z \quad\text{ and }\quad\Z^d_k = \Mod{\Z^d}{k\Z^d} \text{,}
\end{equation*} 
with the convention that $\infty\Z = \{0\}$, so that $\Z^d_{(\infty,\ldots,\infty)} = \Z^d$. The Pontryagin dual of $\Z^d_k$ is denoted by $\U^d_k$. In particular, if $k\in \N^d$ then $\Z^d_k$ is finite and thus self-dual. However, we shall always consider $\U_k^d$ as a compact subgroup of the $d$-torus $\U^d = \U^d_{(\infty,\ldots,\infty)}$, where $\U = \{ z \in \C : |z| = 1 \}$ is the unitary group of $\C$.
\end{notation}

Our bridges between quantum tori, and more generally twisted group C*-al\-ge\-bras of finite products of cyclic groups, will be of the form $\left(\mathscr{B}(\ell^2(\Z^d)),\omega,\pi,\rho\right)$ where $\pi$ and $\rho$ will be non-degenerate faithful representations constructed from left regular representations of these algebras. More formally:

\begin{notation}
Let $d \in \N_\ast$ and $k = (k_1,\ldots,k_d) \in \Nbar$. Let:
\begin{equation*}
I_k = \prod_{j=1}^d \left\{ \left\lfloor \frac{1-k_j}{2} \right\rfloor, \left\lfloor \frac{1-k_j}{2} \right\rfloor+1,\ldots,\left\lfloor \frac{k_j-1}{2} \right\rfloor\right\}\text{.}
\end{equation*}
\end{notation}

We observe that, by construction, the set:
\begin{equation}\label{partition-eq}
\mathscr{P}_k = \left\{ I_k + n : n \in k\Z^d \right\}
\end{equation}
is a partition of $\Z^d$. This is not the partition of $\Z^d$ consisting of the translates of the usual standard domain of $\Z^d$ by $k\Z^d$, but we will find it a bit more convenient (though one could, at the expense of worse notations later on, work with the standard partition of $\Z^d$ in cosets of $k\Z^d$).

Fix $d\in\N_\ast$ and $k \in \Nbar^d$. The canonical surjection $q_k : \Z^d \rightarrow \Z^d_k$ restricts to a bijection from $I_k$ onto $\Z^d_k$. We thus can define an isometric embedding $\vartheta_k : \ell^2\left(\Z^d_k\right)\rightarrow \ell^2(\Z^d)$ by setting for all $\xi \in \ell^2\left(\Z^d_k\right)$:
\begin{equation}\label{vartheta-embedding-def}
\vartheta_k(\xi) : n\in \Z^d \longmapsto \begin{cases}\xi(q_k(n)) \text{ if $n\in I_k$}\\0 \text{ otherwise.}\end{cases}
\end{equation}
Since $\vartheta_k$ is an isometry by construction, $\vartheta_k^\ast \vartheta_k$ is the identity of $\ell^2\left(\Z^d_k\right)$. Therefore, for all skew bicharacter $\sigma$ of $\Z^d_k$, the map $\vartheta_k \pi_{k,\sigma}(\cdot)\vartheta_k^\ast$ is a \emph{non-unital} *-representation of $C^\ast\left(\Z^d_k,\sigma\right)$ on $\ell^2(\Z^d)$. To construct a non-degenerate representation (or, equivalently, unital *-monomorphisms), we proceed as follows. Since $\mathscr{P}_k$, defined by Equation (\ref{partition-eq}), is a partition of $\Z^d$, we have the following decomposition of $\ell^2(\Z^d)$ in a Hilbert direct sum:
\begin{equation}\label{l2-decomposition-k-eq}
\ell^2(\Z^d) = \bigoplus_{n \in k\Z^d} \overline{\operatorname{span} \{ e_j : j \in I_k + n \}}
\end{equation}
with $(e_j)_{j\in\Z^d}$ the canonical basis of $\ell^2(\Z^d)$ given by $e_m(n)\in\{0,1\}$ and $e_m(n) = 1$ if and only if $n=m$, for all $m,n\in\Z^d$. Note that the range of $\vartheta_k$  is $\overline{\operatorname{span} \{ e_j : j \in I_k  \}}$.

For all $n \in k\Z^d$, let 
\begin{equation*}
u_n : \overline{\operatorname{span}\{e_j : j \in I_k\}} \longrightarrow \overline{\operatorname{span}\{e_j : j \in I_k + n\}}
\end{equation*}
be the unitary defined by extending linearly and continuously the map:
\begin{equation*}
e_j \in \{ e_m : m \in I_k \} \longmapsto e_{j+n}\text{.}
\end{equation*}

We now define:

\begin{notation}[\cite{Latremoliere13b}, Notation 4.1.2]\label{rep}
Let $d\in \N_\ast$ and $k \in \Nbar^d$. Let $\sigma$ be a skew-bicharacter of $\Z^d_k$, and $\rho_{k,\sigma}$ the representation of $\left(\ell^1\left(\Z^d_k\right),\conv{k,\sigma},\cdot^\ast\right)$ on $\ell^2(\Z^d)$ defined by Theorem (\ref{twisted-cstar-def}). 

Let $\xi \in \ell^2(\Z^d)$, and write $\xi = \sum_{j \in k\Z^d} \xi_j$ with $\xi_j \in \overline{\operatorname*{span}\{e_m : m \in I_k + j \}}$. Such a decomposition is unique by Equation (\ref{l2-decomposition-k-eq}). Define for all $a\in C^\ast(\Z^d_k,\sigma)$:
\begin{equation}\label{rep-eq}
\pi_{k,\sigma}(a)\xi = \sum_{j\in k\Z^d} u_j\vartheta_k \rho_{k,\sigma}(a)\vartheta_k^\ast u_j^\ast \xi_j \text{,}
\end{equation}
which is well-defined since $\|u_n\vartheta_k \rho_{k,\sigma}(a)\vartheta_k^\ast u_j^\ast \xi_j\|_2\leq\|a\|_{k,\sigma} \|\xi_j\|_2$ for all $j\in k\Z^d$, and $\sum_{j\in kZ^d}\|\xi_j\|^2_2 = \|\xi\|_2^2 < \infty$ by definition of $\xi$.

It is easy to check that $\pi_{k,\sigma}$ thus defined is a faithful, non-degenerate (i.e. unital) *-representation of $C^\ast(\Z_k^d,\sigma)$ on $\ell^2(\Z^d)$, which acts ``diagonally'' in the decomposition of $\ell^2(\Z^d)$ given by Equation (\ref{l2-decomposition-k-eq}).
\end{notation}

The representations $\pi_{k,\sigma}$ for $k\in\Nbar^d$ and $\sigma$ a skew bicharacter of $\Z^d_k$ will be the maps used to defined our bridges, whose ambient space will always be $\mathscr{B}(\ell^2(\Z^d))$. 

Let $B$ be the space of skew bicharacters of $\Z^d$ with the topology of pointwise convergence. We note that, for any $f \in C_c(\Z^d)$, the map $\sigma\in B \mapsto \pi_{\infty^d, \sigma}(f)$ is not continuous in norm, though it is continuous for the weak operator topology. Thus, if $\omega$ is a trace class operator on $\ell^2(\Z^d)$, then $\sigma\in B\mapsto \pi_{\infty^d,\sigma}(f) \omega$ becomes norm continuous. This motivates us to choose a trace class pivot.

However, the bridge norm will be of the form $\opnorm{\pi_{k,\sigma}(\cdot)\omega - \omega\pi_{k',\sigma'}(\cdot)}$, and $\sigma\in B\mapsto \omega\pi_{\infty^d,\sigma}(f)$ is not continuous in norm either in general, even if $\omega$ is trace class. Thus, we wish to commute our pivot with one of the representation. We thus begin with the following theorem:

\begin{notation}
Let $d\in\N_\ast$. Let $(\lambda_n)_{n\in\Z^d}$ be a bounded family of complex numbers indexed by $\Z^d$. The operator $\Diag{\lambda_n}{n\in\Z^d}$ on $\ell^2(\Z^d)$ is defined by setting for all $n \in \Z^d$:
\begin{equation*}
\Diag{\lambda_n}{n\in\Z^d}e_n = \lambda_n e_n\text{,}
\end{equation*}
where $(e_n)_{n\in\Z^d}$ is the canonical Hilbert basis of $\ell^2(\Z^d)$.
\end{notation}

\begin{notation}
For any $d\in\N_\ast$ and any $n = (n_1,\ldots,n_j)\in\Z^d$, we define:
\begin{equation*}
|n| = \sum_{j=1}^d |n_j|\text{.}
\end{equation*}
We note that $|\cdot|$ thus defined is the length  function on $\Z^d$ associated with the canonical generators of $\Z^d$. Thus, in particular, for any $n,m\in\Z^d$ we have $| |n| - |m| | \leq |n-m| \leq |n| + |m|$. 
\end{notation}

\begin{notation}\label{wedge-notation}
For any $d\in\N_\ast$, and for any $k = (k_1,\ldots,k_d)\in \Nbar^d$, we denote by $\wedge k$ the element of $\Nbar$ defined as:
\begin{equation*}
\wedge k = \min \{ |n| : n \not\in I_k \} = \min\left\{ \left\lceil\frac{k_j-1}{2}\right\rceil : j=1,\ldots,d \right\} + 1\text{.}
\end{equation*}
\end{notation}

\begin{notation}\label{weight-notation}
Let $M,N\in\N_\ast$ be given. We define, for all $d\in\N_\ast$:
\begin{equation*}
w_{N,M} : n \in \Z^d \longmapsto \begin{cases}
1 &\text{if $|n| \leq N$,}\\
\frac{M + N - |n|}{M}&\text{ if $N\leq |n| \leq M + N$,}\\
0 &\text{otherwise.}
\end{cases}
\end{equation*}
\end{notation}

\begin{theorem}[\cite{Latremoliere13b}, Theorem 5.1.5]\label{qt-commutator-thm}
Let $d\in\N_\ast$, $k\in\Nbar^d$, and $\sigma$ a skew-bicharacter of $\Z^d_k$. Let $N,M \in \N_\ast$ such that $N+M < \wedge k$. Define, using Notation (\ref{weight-notation}):
\begin{equation*}
\omega_{N,M} = \Diag{ w_{N,M}(n) }{n\in\Z^d}\text{.}
\end{equation*}
Then $\omega_{N,M}$ is a finite rank operator such that, for all $m \in I_k\subseteq \Z^d$, we have:
\begin{equation*}
\left\| \left[\omega_{N,M},\pi_{k,\sigma}\left(\delta_{q_k(m)}\right)\right]\right\|_{\B^d} \leq \frac{|m|}{M} \text{,}
\end{equation*}
where $q_k : \Z^d\rightarrow \Z^d_k$ is the canonical surjection and $\pi_{c,\theta}$ is given by Notation (\ref{rep}) for all $(c,\theta)\in\Xi^d$.
\end{theorem}

Now, we will choose a pivot of the form given in Theorem (\ref{qt-commutator-thm}), because this very theorem will allow us to bound the reach of the resulting bridge. To bound the height, however, requires another tool, given by our next lemma.

\begin{lemma}[\cite{Latremoliere13b}] \label{F_N-lemma}
Let $\mathfrak{L}_1^{1+}$ be the set of all positive trace class operators on $\ell^2(\Z^d)$ of trace $1$. For any $A\in\mathfrak{L}_1^{1+}$, we define:
\begin{equation*}
\psi_A : T \in \B(\ell^2(\Z^d)) \longmapsto \tr(AT)\text{.}
\end{equation*}
Let $\sigma$ be a skew-bicharacter of $\Z^d$ and $l$ be a continuous length function on $\U^d$. Let $\varepsilon>0$. There exists $N\in\N$ and a finite set $\mathfrak{F}_N$ of $\mathfrak{L}_1^+$ such that:
\begin{equation*}
\Haus{\Kantorovich{\Lip_{l,\infty^d,\sigma}}}(\StateSpace(\A_{\infty^d,\sigma}),\{ \psi_A\circ\pi_{\infty^d,\sigma} : A\in \mathfrak{F}_N \}) \leq \varepsilon
\end{equation*}
and
\begin{equation*}
\forall A \in \mathfrak{F}_N\quad P_N A P_N = P_N A = A P_N = A
\end{equation*}
where $P_N$ is the projection of $\ell^2\left(\Z^d\right)$ onto the span of $\{e_n : |n|\leq N\}$, with $(e_n)_{n\in\Z^d}$ the canonical Hilbert basis of $\ell^2\left(\Z^d\right)$.
\end{lemma}

Now, we proceed, informally, as follows. For a given $\varepsilon > 0$ and a given skew bicharacter $\sigma$ of $\Z^d$, we use Theorem (\ref{Fejer-approx-thm}) to get a Fejer kernel $\phi : \Z^d \rightarrow \R$ and a neighborhood $U$ of $\infty^d$ in $\Nbar$ such that the maps $\alpha_{k,\sigma}^\phi$ have a common, finite dimensional range when restricted to the self-adjoint parts of $C^\ast(\Z^d_k,\eta)$ for $k\in U$ and $\eta$ any possible skew bicharacter.

We also use Lemma (\ref{F_N-lemma}) to obtain a projection $P$ so that the set:
\begin{equation*}
 \{\mathrm{tr}(A\pi_{\infty^d,\sigma}(\cdot)) : A\text{ is trace class}, AP=PA=A \}
\end{equation*}
is $\varepsilon$-dense in $(\StateSpace(C^\ast(\Z^d,\sigma)), \Kantorovich{\Lip_{\infty^d,\sigma}})$. 

We now use the combination of these two observations and the construction of continuous fields of states so that:
\begin{equation*}
 \{\mathrm{tr}(A\pi_{\k,\eta}(\cdot)) : A\text{ is trace class}, AP=PA=A \}
\end{equation*}
is $6\varepsilon$-dense in $(\StateSpace(C^\ast(\Z^d_k,\eta)), \Kantorovich{\Lip_{k,\eta}})$, for $k$ in some neighborhood of $\infty^d$ and $\eta$ any skew-bicharcater.

We then use the fact that our pivot commutes with $P$ to show the height of our bridges is no more than $6\varepsilon$. On the other hand, to compute the reach, we use again the continuity of the norm and Lip-norms, the fact that our pivot is trace class --- which allows us to exploit weak operator convergence and turn into norm convergence as explained above, thanks to the commutation property of Theorem (\ref{qt-commutator-thm}). All put together, we obtain the desired estimates and conclude again that Theorem (\ref{qt-main-2}) holds for quantum tori, albeit with more natural representations than the general technique.

While very involved, this method actually produces fairly explicit objects. The effort needed to develop the techniques in \cite{Latremoliere13b} are motivated by two goals: first, that these methods can be extended to many other cases. Second, as our research moves toward understanding modules over {\Lqcms s} which converge to some limit for the quantum propinquity, the explicit construction in \cite{Latremoliere13b} gives us hopes that we may carry some computations in a relatively concrete setting. Another example of such large efforts dedicated to obtain fairly explicit proofs of convergence in the metric sense can be found in \cite{Rieffel09,Rieffel10c,Rieffel11}.

\subsection{Matrix Converging to the Sphere}

Another fundamental example of convergence in noncommutative metric geometry is given by matrix approximations of the C*-algebra $C(S^2)$, where $S^2$ is the $2$-sphere $\{ (x,y,z)\in\R^3 : x^2+y^2+z^2 = 1 \}$, as studied by Rieffel \cite{Rieffel01, Rieffel08, Rieffel09, Rieffel10, Rieffel10c, Rieffel15}. In \cite{Rieffel01}, $C(S^2)$, with an appropriate metric, is shown to be the limit of matrix algebras for Rieffel's quantum Gromov-Hausdorff distance. In subsequent works, motivated by his work on vector bundles over $S^2$ in \cite{Rieffel08}, Rieffel explored the problem of modifying his original construction so that convergence would only involve strong Leibniz Lip-norms. Eventually, this work motivated our own \cite{Latremoliere13, Latremoliere13b, Latremoliere13c}. Now, our metric plays an interesting role in the most current developments on finite dimensional approximations of $C(S^2)$, which can be made sense of using the quantum Gromov-Hausdorff propinquity, which the benefit of extending to matrix algebras over $C(S^2)$ \cite{Rieffel15}. In this section, we follow \cite{Rieffel15} and summarize the construction of finite dimensional approximations of $C(S^2)$.

Rieffel's setup begins with a compact group $G$ and an irreducible unitary representation $U$ of $G$ on some Hilbert space $\Hilbert$, necessarily finite dimensional \cite{Dixmier}. If $\B$ is the C*-algebra of all linear operators on $\Hilbert$, and if we set $\alpha^g (T) = U^g T U^{g^{-1}}$ for all $T \in \B$ and $g\in G$, then $\alpha$ is an ergodic action of $G$ on $\B$ --- since $U$ is irreducible.

Thus, if we choose a continuous length function $l$ on $G$, then by Theorem (\ref{Rieffel-ergo-thm}), we may define a Lip-norm $\Lip_U$ on $\B$ via $\alpha$. In this section, we will assume that $l$ is invariant via conjugation.

Now, let $P \in \B$ be a rank one projection, and let:
\begin{equation*}
H = \left\{ g \in G : \alpha^g(P) = P \right\}\text{.}
\end{equation*}
The C*-algebra $\A = C\left(\bigslant{G}{H}\right)$ of the continuous functions on the homogeneous space $\bigslant{G}{H}$ is endowed with a Lip-norm $\Lip_\A$ obtained via the natural action of $G$ on $\bigslant{G}{H}$, and the length function $l$ on $G$, again as in Theorem (\ref{Rieffel-ergo-thm}).  Alternatively, $\Lip_\A$ is the usual Lipschitz seminorm induced on $\A$ via the quotient metric of $l$ on $\bigslant{G}{H}$.

Thus, we have two {\Lqcms s}: $(\A,\Lip_\A)$, which is a commutative space with a classical structure, and $(\B,\Lip_\B)$, which is a finite dimensional {\Lqcms}.

 The \emph{Berezin symbol} of $T\in \B$ is the function:
\begin{equation*}
\sigma_T : g \in \bigslant{G}{H} \mapsto \mathrm{tr}(T \alpha^g(P))
\end{equation*}
where, by abuse of notation, for any $g \in \bigslant{G}{H}$, we denote by $\alpha^g(P)$ the value of $\alpha^h(P)$ for any $h\in G$ such that $h H = g$. We use the notation $\mathrm{tr}$ to refer to the usual trace on $\B$ whose value on the identity of $\B$ is $\dim_\C \Hilbert$. The Berezin symbol is a positive linear map of norm at most $1$ and equivariant from the action $\alpha$ to the action of $G$ by left translation on $\bigslant{G}{H}$.

If we endow $\B$ with the inner product $T,R \in \B \mapsto \frac{1}{\dim \Hilbert}\mathrm{tr}(R^\ast T)$, and if we regard $\sigma$ as a continuous linear operator from $\B$ to $L^2\left(\bigslant{G}{H}, \mu\right)$, where $\mu$ is the $G$-invariant probability measure on $\bigslant{G}{H}$ for the action of $G$ by left translations, then $\sigma$ has an adjoint denoted by $\breve{\sigma}$ in \cite{Rieffel01}. An explicit formula for this quantization map is:
\begin{equation*}
\breve{\sigma} : f\in L^2\left(\bigslant{G}{H},\mu\right) \mapsto \dim\Hilbert \int_{\bigslant{G}{H}} f(g) \alpha^g(P) \, d\mu(g)\text{,}
\end{equation*}
again with the same abuse of notation as before for $\alpha^g(P)$ where $g\in \bigslant{G}{H}$. 

Thanks to the equivariance of $\sigma$ and $\breve{\sigma}$, we note that for all $T \in \B$ we have $\Lip_\A(\sigma_T) \leq \Lip_B(T)$ and for all $f\in L^2\left(\bigslant{G}{H},\mu \right)$ we have $\Lip_\B(\breve{\sigma}(f))\leq \Lip_\A(f)$. Thus it would be natural to build a bridge using the maps $\sigma$ and $\breve{\sigma}$ --- though these maps are not multiplicative, so this requires some additional work.

None the less, Rieffel used $\sigma$ and $\breve{\sigma}$ in \cite{Rieffel01} to derive estimates on how far $(\A,\Lip_\A)$ and $(\B,\Lip_\B)$ are for the quantum Gromov-Hausdorff distance.

We now see how to use the above framework to build approximations of $C(S^2)$: we would consider $G = SU(2)$ above. More generally, we assume henceforth that $G$ is a semisimple compact Lie group endowed with a conjugation invariant length function $l$, and we begin by choosing $U_1$, as above, an irreducible unitary representation of $G$ on some Hilbert space $\Hilbert_1$. Let $\xi$ be a normalized vector of highest weight associated with $U_1$ and $P$ be projection $P_1$ on the space $\C\xi$. We let $(\B_1,\Lip_{\B_1})$ be the {\Lqcms} constructed as above on the matrix algebra $\B_1$ of all linear operators on $\Hilbert_1$.

For every $n\in\N$, $n\geq 1$, we now let $\xi_n = \xi^{\otimes n} \in \Hilbert^{\otimes n}$, and we denote by $U_n$ the irreducible unitary representation of $G$ obtained by restricting $U_1^{\otimes n}$ to the $U_1^{\otimes n}$ invariant subspace $\Hilbert_n$ generated by $\xi_n$. We note that this new setup also matches our general description above, so we may carry out the same construction, obtaining a {\Lqcms} $(\B_n,\Lip_{\B_n})$ where $\B_n$ is the C*-algebra of all operators on $\Hilbert_n$, and we choose for our projection $P_n$ the projection on $\C\xi_n$. A key observation is that the stabilizer subgroup of $P_1$ and $P_n$ are the same $H$ for all $n\in\N$. Thus $\A = C\left(\bigslant{G}{H}\right)$, in the above construction, is always the same classical space.

Rieffel proved the following in \cite{Rieffel10c}:

\begin{theorem}[\cite{Rieffel10c}, Theorem 9.1]
Using the construction described in this section, we have:
\begin{equation*}
\lim_{n\rightarrow\infty} \prox((\A,\Lip_\A),(\B_n,\Lip_{\B_n}))  = 0\text{,}
\end{equation*}
where $\prox$ is Rieffel's proximity.
\end{theorem}
Consequently, by \cite[Theorem 5.5]{Latremoliere13c}, we also have:
\begin{equation}\label{sphere-cv-eq}
\lim_{n\rightarrow\infty} \propinquity{}((\A,\Lip_\A),(\B_n,\Lip_{\B_n})) = 0\text{.}
\end{equation}

We also note that in \cite{Rieffel09}, Rieffel showed that the above sequence $(\B_n,\Lip_{\B_n})_{n\in\N}$ is Cauchy for $\prox$ --- as $\prox$ is not known to satisfy the triangle inequality, this fact does not follow from \cite[Theorem 9.1]{Rieffel10c}. However, more is actually proven in \cite{Rieffel09}: indeed, Rieffel in fact proved that $(\B_n,\Lip_{\B_n})_{n\in\N}$ is Cauchy for the quantum propinquity, though not in these words as we had yet to introduce our metric. In fact, Rieffel points out that the sort of constructions he carried out in \cite{Rieffel09} did not give a priori easy estimates on the quantum Gromov-Hausdorff distance. This matter is resolved with our work on the quantum propinquity \cite{Latremoliere13}. Since the dual propinquity is complete and dominated by the quantum propinquity, we thus have another proof of the limit in Expression (\ref{sphere-cv-eq}).

In \cite[Theorem 6.8]{Rieffel15}, Rieffel proves that in fact, $\left(C\left(\bigslant{G}{H}\right),\Lip_\A\right)$ is the limit of $(\B_n,\Lip_{\B_n})_{n\in\N}$ for the quantum propinquity. More importantly, Rieffel extends this convergence to matrix algebras over $C\left(\bigslant{G}{H}\right)$ in \cite[Theorem 6.10]{Rieffel15}, in the following sense. If $(\A,\Lip)$ is any {\Lqcms} then there is a natural extension of the notion of Lip-norms to matrix algebras $M_k(\A)$ over $\A$, by applying $\Lip$ to every matrix entry of an element in $\M_k(\A)$. In \cite{Rieffel15}, the notions of a bridge and its length are extended to this setting, and is shown to converge to $0$ when applied to the example described in this section, for a fixed $k \in \N$.

Thus the quantum propinquity appears as a natural tool in the study of convergence of modules, and future work will hopefully carry this project to fruition.

\subsection{New Results on Perturbations of the quantum metrics and the Quantum Propinquity}

\subsubsection{A simple perturbation lemma}

A simple application of the quantum Gromov-Hausdorff propinquity is to provide a framework for discussing perturbations of the metric structure of {\Lqcms s}. This section presents results concerned with perturbation of the metrics, and \emph{which are new to this survey}. We present a new lemma which simplify some computations for the quantum propinquity, and then, three new examples of applications: continuity for conformal deformations of spectral triples, continuity for another type of perturbation of spectral triples, and last, a generalization to the quantum propinquity of our result on dimensional collapse for quantum tori \cite{Latremoliere05}.

\begin{lemma}\label{perturb-lemma}
Let $F$ be a permissible function. Let $(\A,\Lip_\A)$ and $(\B,\Lip_\B)$ be two {\Qqcms{F}s}. If there exists a bridge $\gamma  = (\D,\pi_\A,\pi_\B,\omega)$ from $(\A,\Lip_\A)$ to $(\B,\Lip_\B)$ with $\|\omega\|_\D\leq 1$ and $\delta > 0$ such that:
\begin{enumerate}
\item for all $a\in\dom{\Lip_\A}$ there exists $b\in \sa{\B}$ such that:
\begin{equation*}
\max \left\{ \|\pi_\A(a)\omega - \omega\pi_\B(b)\|_\D, |\Lip_\A(a) - \Lip_\B(b)| \right\} \leq \delta\Lip_\A(a) \text{,}
\end{equation*}
\item for all $b\in\dom{\Lip_\B}$, there exists $a\in\sa{\A}$ such that:
\begin{equation*}
\max \left\{ \|\pi_\A(a)\omega - \omega\pi_\B(b)\|_\D, |\Lip_\A(a) - \Lip_\B(b)| \right\} \leq \delta\Lip_\B(b) \text{,}
\end{equation*}
\end{enumerate}
then:
\begin{multline*}
\qpropinquity{F}((\A,\Lip_\A),(\B,\Lip_\B)) \leq \\ \max\left\{\delta\left(1 + \max\left\{\diam{\StateSpace(\A)}{\Kantorovich{\Lip_\A}}, \diam{\StateSpace(\B)}{\Kantorovich{\Lip_\B}}\right\} \right), \right. \\ \left. \bridgeheight{\gamma}{\Lip_\A,\Lip_\B} \right\} \text{.}
\end{multline*}
\end{lemma}

\begin{proof}
Let $R > \max\{\diam{\StateSpace(\A)}{\Lip_\A},\diam{\StateSpace(\B)}{\Lip_\B}\}$. Let $a\in \sa{\A}$ with $\Lip_\A(a)\leq 1$. Then there exists $b \in \sa{\B}$ with $\|\pi_\A(a)\omega - \omega \pi_\B(b)\|_\D\leq\delta$ and $\Lip_\B(b)\leq 1 + \delta$. Thus:
\begin{equation*}
\Lip_\B\left(\frac{1}{\delta+1}b\right)\leq 1\text{,}
\end{equation*}
so by \cite[Proposition 2.2]{Rieffel99}, we conclude that there exists $t\in \R$ such that:
\begin{equation*}
\left\|\frac{1}{\delta+1}b - t\unit_\B \right\|_\B \leq R\text{.}
\end{equation*}
Thus:
\begin{equation*}
\begin{split}
\left\|\pi_\A(a)\omega - \omega \pi_\B\left(\frac{1}{1+\delta}b + \delta t \unit_\B \right) \right\|_\D &\leq \|\pi_\A(a)\omega - \omega \pi_\B(b)\|_\D \\ 
&\quad + \delta\left\|\frac{1}{\delta+1}b - t \unit_\B\right\|_\B \|\omega\|_\D\\
&\leq \delta(1+R)\text{.}
\end{split}
\end{equation*}
while $\Lip_\B\left(\frac{1}{1+\delta}b + t \delta \unit_\B\right) \leq 1$.
The result is symmetric in $(\A,\Lip_\A)$ and $(\B,\Lip_\B)$.

Thus by Definition (\ref{quantum-propinquity-def}), our lemma is proven.
\end{proof}

In particular, a consequence of Lemma (\ref{perturb-lemma}) is:

\begin{proposition}\label{dil-prop}
Let $(\A,\Lip_\A)$ and $(\B,\Lip_\B)$ be two {\Qqcms{F}} for some permissible function $F$. If $\alpha : \A \rightarrow \B$ is a $\delta$-bi-Lipschitz *-isomorphism for some $\delta \geq 1$, i.e.:
\begin{equation*}
\delta^{-1} \Lip_\B \circ \pi \leq \Lip_\A \leq \delta \Lip_\B\circ \pi\text{,}
\end{equation*}
then:
\begin{multline*}
  \qpropinquity{}\left((\A,\Lip_\A), (\B,\Lip_\B)\right) \\ \leq \left(\delta-1 \right)\left(1 + \max\{\diam{\StateSpace(\A)}{\Kantorovich{\Lip_\A}},\diam{\StateSpace(\A)}{\Kantorovich{\Lip_\A}}\} \right)\text{.}
\end{multline*}
\end{proposition}

\begin{proof}
Indeed, we simply consider the bridge $(\B, \mathrm{id}_\B, \pi, \unit_\B)$ and note that for all $b\in \dom{\Lip_\B}$ we have $\|b-\pi(\pi^{-1}(b))\|_\B = 0$ and $\Lip_\A(\pi^{-1}(b)) \leq \delta \Lip_\B(b)$ so:
\begin{equation*}
|\Lip_\B(b) - \Lip_\A(\pi^{-1}(b))| \leq \left( \delta - 1 \right)\Lip_\B(b)\text{.}
\end{equation*}
The computation is symmetric in $\A$ and $\B$, and thus our result follows from our Lemma (\ref{perturb-lemma}). 
\end{proof}

We note that the Leibniz property does not play any role, but the lower semicontinuity is used in translating the bi-Lipschitz property in terms of the Lip-norms.

\subsubsection{Perturbation of the metrics from spectral triples}

We now propose a couple of other examples inspired by the noncommutative geometry literature. We begin with small perturbations of a conformal type, as in \cite{Connes08, Ponge14}, which leads to twisted spectral triples. This result borrows from Example (\ref{Connes-LP-ex}) and Example (\ref{module-LP-ex}).

\begin{theorem}
  Let $\A$ be a unital C*-algebra, $\pi$ a faithful unital *-representation of $\A$ on some Hilbert space $\Hilbert$ and $D$ be a not necessarily bounded self-adjoint operator on a dense subspace $\dom{D}$ of $\Hilbert$, such that if
  \begin{equation*}
    \dom{\Lip} = \left\{ a\in\sa{\A} : \pi(a)\dom{D}\subseteq\dom{D}, [D,\pi(a)]\text{ is bounded} \right\}\text,
  \end{equation*}
  and if:
  \begin{equation*}
    \Lip : a\in\sa{\A} \longmapsto \opnorm{[D,\pi(a)]} \text,
  \end{equation*}
  then $(\A,\Lip)$ is a {\Lqcms}. 

  Let $\mathrm{GLip}(\A)$ be the set of all invertible elements $h$ in $\dom{\Lip}$ with $\Lip(h) < \infty$.  For any $h\in\mathrm{GLip}(\A)$, we define $D_h = \pi(h)D\pi(h)$, $\sigma_h : a\in\A\mapsto h^2 a h^{-2}$ and:
\begin{equation*}
\Lip_h : a\in \sa{\A} \longmapsto \opnorm{D_h\pi(a) - \pi(\sigma_h(a))D_h} \text{.}
\end{equation*}

Then $(\A,\Lip_h)$ is a {\Qqcms{\left(\|h^2\|_\A\|h^{-2}\|_\A, 0 \right)}} and moreover, if $(h_n)_{n\in\N}$ is a sequence in $\mathrm{GLip}(\A)$ which converges to $h \in \mathrm{GLip}$ in $\A$, and such that:
\begin{equation*}
\lim_{n\rightarrow\infty} \Lip(h_n^{-1}h) = \lim_{n\rightarrow\infty} \Lip(h_n h^{-1}) = 0\text{,} 
\end{equation*}
then:
\begin{equation*}
\lim_{n\rightarrow\infty} \qpropinquity{M,0}\left((\A,\Lip_{h_n}),(\A,\Lip_h)\right) = 0 \text{,}
\end{equation*}
where $M\geq \sup_{n\in\N} \|h_n^2\|_\A\|h_n^{-2}\|_\A$.
\end{theorem}

\begin{proof}
Fix $h, w\in \mathrm{GLip}$ and denote $\pi(h)$ by $k$ and $\pi(w)$ by $m$.

To simplify notation, for all $a\in\A$, we write:
\begin{equation*}
[D_h, \pi(a)]_h = D_h \pi(a) - \pi\circ\sigma_h(a)D_h\text{.}
\end{equation*}
We note that $\Lip_h$ is well-defined on the subalgebra
\begin{equation*}
  \D = \left\{ a \in \A : \pi(a)\dom{D}\subseteq\dom{D}, [D,\pi(a)]\text{ is bounded } \right\}
\end{equation*}
which is dense in $\A$, which follows easily since $\D$ contains $\dom{\Lip}$. It can also be checked that $\D$ is actually a *-subalgebra of $\A$ (since $D$ is self-adjoint, and elements of $\D$ have bounded commutators with $D$).

Moreover, if $a,b \in \D$, then:
\begin{equation*}
\begin{split}
\Lip_h(ab) &= \opnorm{[D_h,\pi(ab)]_h} \\
&= \opnorm{D_h \pi(a)\pi(b) - \pi(\sigma_h(a))\pi(\sigma_h(b))D_h}\\
&\leq \opnorm{D_h \pi(a)\pi(b) - \pi(\sigma_h(a))D_h\pi(b)} \\
&\quad + \opnorm{\pi(\sigma_h(a))D_h\pi(b) - \pi(\sigma_h(a))\pi(\sigma_h(b))D_h}\\
&\leq \Lip_h(a)\|b\|_\A + \|\sigma_h(a)\|_\A\Lip_h(b) \text{.}
\end{split}
\end{equation*}
Thus, in particular, $\Lip_h$ is quasi-Leibniz for the permissible function:
\begin{equation*}
F: (x,y,l_x,l_y) \in [0,\infty)^4\longmapsto \|h^2\|_\A\|h^{-2}\|_\A x l_y + y l_x 
\end{equation*}
where we note that $\|h^2\|_\A\|h^{-2}\|_\A \geq 1$. Of  course, the same holds for $\Lip_w$ (the quasi-Leibniz relation depends on the choice of $h$ and $w$; however we can find a uniform quasi-Leibniz property applicable to any sequence satisfying our theorem, as we shall see later on).

% Last, we note that for all $a\in\mathrm{GLip}(\A)$ we have:
% \begin{equation*}
% \begin{split}
% 0 &= [D,\unit_\A] = [D,\pi(a)\pi(a)^{-1}]\\
% &= [D,\pi(a)]\pi(a)^{-1} + \pi(a)[D,\pi(a)^{-1}]
% \end{split}
% \end{equation*}
% thus $[D,\pi(a)^{-1}] = - \pi(a)^{-1}[D,\pi(a)]\pi(a)^{-1}$, and therefore:
% \begin{equation*}
% \Lip(a^{-1})\leq \|a^{-1}\|_\A^2 \Lip(a)\text{.}
% \end{equation*}
% Thus $\Lip$ is a \emph{strong Leibniz} Lip-norm \cite{Rieffel10c}, which we will need later on. 

We begin with a simple computation for all $a\in\dom{\Lip}$:
\begin{equation}\label{conformal-pert-eq1}
\begin{split}
[D_h, \pi(a)]_h &= [kDk , \pi(a)]_h \\
&= k D k \pi(a) - k^2\pi(a)k^{-2} k D k\\
&= k(D k\pi(a)k^{-1} - k\pi(a)k^{-1} D)k\\
&= km^{-1}(mDm m^{-1}k\pi(a)k^{-1}m - mk\pi(a)k^{-1}m^{-1} mDm)m^{-1}k\\
&= km^{-1} [mDm, m^{-1}k\pi(a)k^{-1}m]_w m^{-1}k \\
&= km^{-1} [D_m, \pi(w^{-1}h a h^{-1}w)]_w m^{-1}k \text{.}
\end{split}
\end{equation}

In particular, for all $a\in \dom{\Lip}$, we have
\begin{align*}
  \opnorm{[D, \pi(a)]}
  &=\opnorm{m^{-1} [mDm, m^{-1} \pi(a) m]_w m^{-1}}\\
  &\leq \|w^{-1}\|_\A^2 \Lip_w(w^{-1} a w) \text.
\end{align*}
Therefore, $\opnorm{[D,\pi(w a w^{-1})]} \leq \|w^{-1}\|_\A^2 \Lip_w(a)$ for all $a\in\dom{\Lip}$. Thus, if $\Lip_w(a) =0$, then $w a w^{-1} \in \C\unit_\A$, which implies $a\in\C\unit_\A$.

Now, let $\mu\in\StateSpace(\A)$ be a fixed state of $\A$. Suppose $(a_n)_{n\in\N}$ is a sequence in $\sa{\A}$ such that $\Lip_w(a_n)\leq 1$, and $\mu(a_n) = 0$, for all $n\in\N$. Therefore, $\opnorm{[D,w a w^{-1}]} \leq 1$ for all $n\in\N$. Now,
\begin{align*}
  1
  &\geq \opnorm{[D,(w a w^{-1})]^\ast} \\
  &= \opnorm{(w a w^{-1})^\ast D^\ast - D^\ast (w a w^{-1})^\ast} \\
  &= \opnorm{D(w a w^{-1})^\ast - (w a w^{-1})^\ast D} \text.
\end{align*}

Denoting, for any $b\in \A$, the real part $\frac{b+b^\ast}{2}$ by $\Re b$, and the imaginary part $\frac{b-b^\ast}{2}$ of $b \in \A$ by $\Im b$, we observe that $\Lip(\Re(w a_n w^{-1})) \leq 1$ and $\Lip(\Im (w a_n w^{-1})) \leq 1$. Thus, both $(\Re(w a _n w^{-1}))_{n\in\N}$ and $(\Im(w a _n w^{-1}))_{n\in\N}$ have convergent subsequences, since $\Lip$ is a Lip-norm; therefore, so does $(a_n)_{n\in\N}$. Thus, $\Lip_w$ is also a Lip-norm.

Let $r_\xi = \diam{\StateSpace(\A)}{\Kantorovich{\Lip_\xi}}$ for all $\xi \in \mathrm{GLip}(\A)$ and let $a\in\sa{\A}$ with $\Lip(a)<\infty$. By \cite[Proposition 2.2]{Rieffel99}, for all $a\in\sa{\A}$, there exists $t\in \R$ such that $\|a-t\unit_\A\|_\A \leq r_w \Lip_w(a)$. Thus:
\begin{equation*}
\begin{split}
\opnorm{[D_h , \pi(a)]_h} &= \opnorm{[kDk, \pi(a-t\unit_\A)]_h}\\
&=\opnorm{km^{-1} [mDm, m^{-1}k\pi(a-t\unit_\A)k^{-1}m]_w m^{-1} k}\\
&\leq \|w^{-1} h\|_\A^2 \Lip_w(w^{-1}h (a-t\unit_\A) h^{-1}w) \\
&\leq \|w^{-1} h\|_\A^2 \left( \Lip_w(w^{-1}h(a-t\unit_\A))\|h^{-1}w\|_\A \right.\\
&\quad  \left.+ \|w^2\|_\A\|w^{-2}\|_\A \|w^{-1}h\|_\A\|a-t\unit\|_\A \Lip_w(h^{-1}w)\right) \\
&\leq   \|w^{-1} h\|_\A^2 \left(\Lip_w(w^{-1} h)\|a-t\unit_\C\|_\A \|h^{-1}w\|_\A \right.\\
&\quad + \|\sigma_w(w^{-1} h)\|_\A \Lip_w(a) \|h^{-1}w\|_\A\\
&\quad  \left.+  \|w^2\|_\A\|w^{-2}\|_\A \|w^{-1}h\|_\A\|a-t\unit\|_\A \Lip_w(h^{-1}w) \right) \\
&\leq \Lip_w(a) \|w^{-1}h\|^2 \left( r_w \Lip_w(w^{-1}h)\|h^{-1}w\|_\A\right. \\
&\quad  + \|\sigma_w(w^{-1} h)\|_\A\|h^{-1}w\|_\A \\
&\quad \left. +  r_w \|w\|_\A^2\|w^{-2}\|_\A \Lip_w(h^{-1}w)\|w^{-1}h\|_\A\right)\text{.}
\end{split}
\end{equation*}

Define
\begin{multline*}
  f(w,h) = \|w^{-1}h\|^2 ( r_w \Lip_w(w^{-1}h)\|h^{-1}w\|_\A
\\ + \|\sigma_w(w^{-1} h)\|_\A\|h^{-1}w\|_\A  +  r_w \|w\|_\A^2\|w^{-2}\|_\A \Lip_w(h^{-1}w)\|w^{-1}h\|_\A)\text{.}
\end{multline*}

We thus have shown that, by symmetry,
\begin{equation*}
  f(h,w)^{-1} \Lip_w(a) \leq \Lip_h(a) \leq f(w,h) \Lip_w(a) \text.
\end{equation*}

Now, assume $(\omega_n)_{n\in\N}$ is sequence in $\mathrm{GLip}(\A)$ such that $\lim_{n\rightarrow\infty} \|\omega_n - h\|_{\A} = \lim_{n\rightarrow\infty} \Lip(\omega_n h^{-1}) = 0$. Note that $\left(\|\omega_n h^{-1}\}\right)_{n\in\N}$ converges to $1$, as does $\left(\|\omega_n h^{-1}\}\right)_{n\in\N}$. Similarly, the sequence $(\|\omega_n\|_\A)_{n\in\N}$ and $(\|\omega_n^{-1}\|)_{n\in\N}$ are bounded.

Now
\begin{equation*}
  \Lip_{\omega_n}( \omega_n^{-1} h )\leq f(\omega_n,\unit_\A) \Lip(\omega_n^{-1} h) 
\end{equation*}
so $\lim_{n\rightarrow\infty} \Lip_{\omega_n}(\omega_n^{-1} h) = 0$.

On the other hand, since $h\omega_n^{-1}$ converges in norm to $\unit_\A$, we also have
\begin{equation*}
  \|\sigma_{\omega_n}(\omega_n^{-1} h)\|_\A = \|\omega_n h \omega_n^{-2}\|_\A \xrightarrow 1
\end{equation*}

Therefore, $\lim_{n\rightarrow\infty} f(\omega_n,h) = 1$. By Proposition (\ref{dil-prop}), our result is proven.
\end{proof}

Another approach to metric fluctuations is given by the following example.

\begin{proposition}
  Let $\A$ be a unital C*-algebra, $\pi$ a unital faithful *-representation of $\A$ on some Hilbert space $\Hilbert$, and $D$ a self-adjoint, possibly unbounded operator on a dense subspace $\dom{D}$ of $\Hilbert$, such that if:
  \begin{equation*}
    \dom{\Lip} = \left\{ a \in \sa{\A} : \pi(a)\dom{D}\subseteq\dom{D}, [D,\pi(a)] \text{ is bounded } \right\}
  \end{equation*}
  and
  \begin{equation*}
    \Lip : a\in\A \longmapsto \opnorm{[D,\pi(a)]}
  \end{equation*}
  then $(\A,\Lip)$ is a {\Lqcms}. Let $r = \diam{\A}{\Lip}$ be the diameter of $(\StateSpace(\A),\Kantorovich{\Lip_D})$.

  Let $\B$ be the C*-algebra of all bounded linear operators on $\Hilbert$. 

  For any $\omega \in \sa{\B}$ on $\Hilbert$, we define:
  \begin{equation*}
    D_\omega = D+\omega\text{ and }\Lip_\omega : a\in\dom{\Lip} \mapsto \opnorm{[D_\omega, \pi(a)]}\text{.}
  \end{equation*}
  
  If $\|\omega\|_{\B} < \frac{1}{2r}$, then the pair $(\A,\Lip_\omega)$ is a {\Lqcms} and, moreover:
\begin{equation*}
\omega \in \{ \xi \in \sa{\B} : \|\xi\|_{\B} < \frac{1}{2r} \} \longmapsto (\A,\Lip_\omega)
\end{equation*}
is continuous for the quantum Gromov-Hausdorff propinquity $\qpropinquity{}$.
\end{proposition}

\begin{proof}
For all $a\in\dom{\Lip}$ and for all $t\in\R$, we have
\begin{align*}
  \left|\Lip_\omega(a)-\Lip_{\omega'}(a)\right|
  & \leq \|[\omega-\omega',a]\|_{\B} = \|[\omega-\omega', a-t\unit_\B]\|_{\B}  \\
  & \leq 2 \|\omega-\omega'\|_{\B} \|a-t\unit_\A\|_{\A}\text.
\end{align*}
Since $\Lip$ is a Lip-norm, there exists $t\in\R$ such that $\|a-t\unit_\A\|_{\A} \leq r \Lip(a)$. Thus we conclude
\begin{equation*}
  \forall a \in \dom{\Lip} \quad \left|\Lip(a)-\Lip_\omega(a)\right| \leq 2 r \|\omega\|_{\B} \Lip(a) \text.
\end{equation*}
Thus
\begin{equation*}
  \forall a \in \dom{\Lip} \quad (1-2 r \|\omega\|_{\B}) \Lip(a) \leq \Lip_\omega(a) \leq (1 + 2 r \|\omega\|_{\B}) \Lip(a) \text.
\end{equation*}

Since $\|\omega\|_{\B}<\frac{1}{2r}$, using \cite[Lemma 1.10]{Rieffel98a}, we conclude that, indeed, $(\A,\Lip_\omega)$ is a {\qcms}. Moreover, we note that the diameter $\diam{\A}{\Lip_\omega}$ of $(\StateSpace(\A),\Kantorovich{\Lip_\omega})$ is, at most, $r_\omega = \frac{r}{1 - 2 r \|\omega\|_\B}$.

We then have
\begin{equation*}
   \forall a \in \dom{\Lip} \quad \left|\Lip(a)_{\omega'}-\Lip_\omega(a)\right| \leq 2 \min\{ r_\omega \|\omega'-\omega\|_{\B} \Lip_\omega(a), r_{\omega'} \|\omega-\omega'\|_\B \Lip_{\omega'}(a) \} \text.
\end{equation*}

Now, a direct application of Lemma (\ref{perturb-lemma}) shows that for all $\omega,\omega' \in \sa{\B}$:
\begin{multline*}
\qpropinquity{}((\A,\Lip_\omega),(\A,\Lip_{\omega'})) \\
\begin{aligned}
&\leq 2 \min\{ r_\omega \|\omega'-\omega\|_{\B} , r_{\omega'} \|\omega-\omega'\|_\B \}\left(1+ r \max\left\{\frac{1}{1-2 r\|\omega\|_\B},\frac{1}{1-2 r\|\omega'\|_\B}\right\} \right)
\end{aligned}
\end{multline*}
Our proposition follows immediately.
\end{proof}

\subsubsection{Perturbation of the length functions for Lip-norms from ergodic actions}

In \cite{Latremoliere05}, we showed that we can collapse quantum tori of the form $C^\ast(\Z^d,\sigma)$ to quantum tori of the form $C^\ast(\Z^r,\eta)$ where $r < d$, i.e. observe a noncommutative form of dimensional collapse. This phenomenon is also present with the quantum Gromov-Hausdorff propinquity. The following proof borrows from \cite[Theorem 4.4]{Latremoliere05} and uses our Lemma (\ref{perturb-lemma}), while illustrating another form of metric perturbation. It should be observed that all the Lip-norms in this result are Leibniz.

\begin{theorem}
Let $\alpha$ be a strongly continuous action of a compact group $G$ on a unital C*-algebra $\A$ such that:
\begin{equation*}
\left\{ a\in\sa{\A} : \forall g \in G \quad \alpha^g(a) = a \right\} = \C\unit_\A\text{.}
\end{equation*}
For all $n\in\N$, let $l_n$ be a continuous length function on $G$ and $M > 0$ such that:
\begin{equation*}
\sup_{n\in\N} \diam{G}{l_n} \leq M\text{.}
\end{equation*}

Let $H$ be a closed normal subgroup of $G$ and $K = \bigslant{G}{H}$ be the compact quotient group. Let $l'_\infty$ be a function on $G$ such that:
\begin{equation*}
\left\{ g \in H : l'_\infty(g) = 0 \right\} = H \text{.}
\end{equation*}

Let:
\begin{equation*}
\A_K = \left\{ a \in \A : \forall g \in H \quad \alpha^g(a) = a \right\}
\end{equation*}
be the fixed point of $\alpha$ restricted to $H$, and denote the quotient map from $G$ to $K$ by $g in \G \mapsto [g] \in K$.  Let $\alpha_K$ be the action of $K$ induced by $\alpha$ on $\A_K$ via:
\begin{equation*}
\alpha_K^{[g]}(a) = \alpha^g(a)
\end{equation*} 
for all $a\in\A_K$.

Note that if we set $l_\infty([g]) = l'_\infty(g)$ for all $g \in G$ we define a continuous length function on $K$. Let $\Lip_\infty$ be the Lip-norm induced on $\A_K$ by $\alpha_K$ and $l_\infty$ via Theorem (\ref{Rieffel-ergo-thm}). Similarly, for all $n\in\N$ let $\Lip_n$ be the Lip-norm on $\A$ induced by $\alpha$ and $l_n$ via Theorem (\ref{Rieffel-ergo-thm}).

Denoting the unit of $G$ by $e$, if:
\begin{equation*}
\lim_{n\rightarrow\infty} \sup \left\{ \left| \frac{l_\infty(g)}{l_n(g)} - 1 \right| : g \in G \setminus \{e\} \right\} = 0
\end{equation*}
then:
\begin{equation*}
\lim_{n\rightarrow\infty} \qpropinquity{}((\A,\Lip_n), (\A_K, \Lip_\infty)) = 0 \text{.}
\end{equation*}
\end{theorem}

\begin{proof}
We begin by checking that $l_\infty$ is indeed a continuous length on $K$. If $[g] = [g']$ for two $g,g' \in G$ then there exists $h\in H$ such that $gh = g'$. By assumption:
\begin{equation*}
\begin{split}
l'_\infty(g') &= l'_\infty(gh) \leq l'_\infty(g) + l'_\infty(h) = l'_\infty(g)\\
&= l'_\infty(g'h^{-1}) \leq l'_\infty(g') + l'_\infty(h^{-1}) = l'_\infty(g')\text{.}
\end{split}
\end{equation*}
and thus $l_\infty'(g) = l_\infty'(g')$. Thus $l_\infty$ is well-defined. It is then easy to check that $l_\infty$ is a length function on $K$.

By assumption, for all $n\in\N$ we have:
\begin{equation*}
\| l_n - l_\infty'\|_{C(G)} \leq \sup_{g\in G\setminus\{e\}} l_n(g) \left| 1- \frac{l'_\infty(g )}{l_n(g)}\right| \leq \sup_{k\in\N}\diam{G}{l_k} \left| 1- \frac{l'_\infty(g )}{l_n(g)} \right|\text{,}
\end{equation*}
we conclude by assumption that $(l_n)_{n\in\N}$ converges to $l'_\infty$ uniformly on $G$. Thus $l_\infty'$ is continuous on $G$. Consequently, $l_\infty$ is continuous on $K$ (see \cite[Lemma 4.1]{Latremoliere05} as well).

We define the expected value:
\begin{equation*}
\mathds{E} : a\in\A \longmapsto \int_H \alpha^g(a) \, d\mu(g)
\end{equation*}
for all $a\in\A$, where $\mu$ is the Haar probability measure of $G$. By construction, since $\mathds{E}(\alpha^g(a)) = \mathds{E}(a)$ for all $a\in\A$ and $g\in H$, we conclude that $\mathds{E}$ is valued in $\A_K$. It is also easy to check that $\mathds{E}$ is a unital, positive linear map.

Moreover, for all $a\in\sa{\A}$ and $n\in\N$:
\begin{equation*}
\begin{split}
\|a-\mathds{E}(a)\|_\A &\leq \int_H \|a-\alpha^g(a)\|_\A \, d\mu(g) \\
&\leq \int_H l_n(g) \frac{\|a-\alpha^g(a)\|_\A}{l_n(g)} \, d\mu(g)\\
&\leq \diam{H}{l_n} \Lip_n(a)\text{.}
\end{split}
\end{equation*}

Since $(l_n)_{n\in\N}$ converges uniformly to $l_\infty'$ which is null on $H$, we conclude that $\diam{H}{l_n}$ converges to $0$.

Let $\varepsilon > 0$ and let $N\in\N$ such that for all $n\geq N$:
\begin{itemize}
\item $\diam{H}{l_n} \leq \varepsilon$,
\item $\sup\left\{\left|\frac{l'_\infty(g)}{l_n(g)}-1\right| : g \in G\setminus\{e\} \right\} \leq \varepsilon$.
\end{itemize}
Let $n\geq N$.

Let $a\in\sa{\A}$ with $\Lip(a)_n < \infty$. We have:
\begin{equation*}
\begin{split}
|\Lip_\infty(\mathds{E}(a)) - \Lip_n(a)| &= \left|\sup\left\{ \frac{\|a-\alpha^g(a)\|_\A}{l_\infty([g])} : g\in G\setminus\{e\} \right\} - \Lip_n(a)\right|\\
&\leq\sup\left\{ \left|\frac{l_n(g)}{l'_\infty(g)}-1\right| \frac{\|a-\alpha^g(a)\|_\A}{l_n([g])} : g\in G\setminus\{e\} \right\}\\
&\leq \varepsilon \sup\left\{ \frac{\|a-\alpha^g(a)\|_\A}{l_n([g])} : g\in G\setminus\{e\} \right\}\\
&\leq \varepsilon \Lip_n(a) \text{.}
\end{split}
\end{equation*}

Hence, for all $a\in\sa{\A}$ there exists $b = \mathds{E}(a)$ in $\sa{\A_K}$ with:
\begin{equation*}
\| a - \mathds{E}(a) \|_\A \leq \varepsilon\Lip_n(a) \text{ and }|\Lip_\infty(\mathds{E}(a)) - \Lip_n(a)|\leq \varepsilon \Lip_n(a)\text{.}
\end{equation*}

On the other hand, if $a\in\sa{\A_K}$ then $a = \mathds{E}(a)$ and thus, by a similar computation:
\begin{equation*}
\| a - a \|_\A = 0 \text{ and }|\Lip_\infty(a) - \Lip_n(a)| \leq \varepsilon \Lip_\infty(a)\text{.}
\end{equation*}

We thus conclude by Lemma (\ref{perturb-lemma}) that for all $n\geq N$:
\begin{equation*}
\qpropinquity{}((\A,\Lip_n),(\A_K,\Lip_\infty)) \leq \varepsilon\text{.}
\end{equation*}

This concludes our theorem.
\end{proof}

Lemma (\ref{perturb-lemma}) thus provides a convenient tool to simplify certain computations related to relatively simple modifications of the Lip-norms.

\section{A Gromov-Hausdorff hypertopology for quantum proper metric spaces}

We propose in \cite{Latremoliere14b} a new topology on proper quantum metric spaces, which extends both Gromov's topology on proper metric spaces and the topology of the Gromov-Hausdorff propinquity. This work is quite technical, so this section will focus on the key ideas.

\subsection{Gromov-Hausdorff topology}

We begin with a few notations, and refer to \cite{Latremoliere14b} for details:

\begin{notation}
Let $(X,\mathsf{d})$ be a metric space, $x_0\in X$ and $r\geq 0$. The closed ball:
\begin{equation*}
\left\{ x\in X : \mathsf{d}(x,x_0)\leq r \right\}
\end{equation*}
is denoted by $\cBall{X}{\mathsf{d}}{x_0}{r}$. When the context is clear, we simply write $\cBall{X}{}{x_0}{r}$ for $\cBall{X}{\mathsf{d}}{x_0}{r}$.
\end{notation}

When working with Gromov-Hausdorff distance, we will often use the following notion of approximate inclusion:

\begin{notation}
Let $A,B \subseteq Z$ be subsets of a metric space $(Z,\mathsf{d})$. We write $B\almostsubseteq{Z,\mathsf{d}}{\varepsilon} A$ when:
\begin{equation*}
B\subseteq \cBall{Z}{\mathsf{d}}{A}{\varepsilon} = \bigcup_{a\in A}\cBall{Z}{\mathsf{d}}{a}{\varepsilon}\text{.}
\end{equation*}
When the context is clear, we may simply write $B\subseteq_\varepsilon A$ for $B\almostsubseteq{Z,\mathsf{d}}{\varepsilon} A$.
\end{notation}

Gromov defined in \cite{Gromov81} a topology on the class of pointed proper metric spaces as follows. We first define a local form of the Hausdorff distance.

\begin{definition}\label{delta-r-def}
Let $X,Y\subseteq Z$ be two subsets of a metric space $(Z,\mathsf{d})$ and let $x_0\in X$, $y_0\in Y$. For any $r > 0$, we define:
\begin{equation*}
\delta_r^{(Z,\mathsf{d})}( (X,x_0) ,(Y,y_0) ) = \inf\Set{\varepsilon > 0}{
\mathrm{d}(x_0,y_0) &\leq \varepsilon\text{ and }\\
\cBall{X}{}{x_0}{r} &\almostsubseteq{Z,\mathsf{d}}{\varepsilon} Y\text{,}\\
\cBall{Y}{}{y_0}{r}&\almostsubseteq{Z,\mathsf{d}}{\varepsilon} X \text{.}\\
}
\end{equation*}
\end{definition}

We note in \cite[Theorem 2.1.6]{Latremoliere14b}, we show that:
\begin{equation}\label{alt-exp-eq0}
\delta_r^{(Z,\mathsf{d})}((X,x_0),(Y,y_0)) = \min\Set{\varepsilon > 0}{
\mathsf{d}(x_0,y_0)&\leq\varepsilon \text{ and }\\
\cBall{X}{}{x_0}{r}&\subseteq_\varepsilon \cBall{Y}{}{y_0}{r+2\varepsilon},
\\ \cBall{Y}{}{y_0}{r}&\subseteq_\varepsilon \cBall{X}{}{x_0}{r+2\varepsilon}}
\end{equation}
with the notations of Definition (\ref{delta-r-def}).

Now, we use Definition (\ref{delta-r-def}) to define an intrinsic notion of convergence.

\begin{definition}\label{Delta-r-def}
Let $\mathds{X} = (X,\mathsf{d}_X,x_0)$ and $\mathds{Y} = (Y,\mathsf{d}_Y,y_0)$ be two pointed proper metric spaces and $r > 0$. We define $\Delta_r(\mathds{X},\mathds{Y})$ as:
\begin{equation*}
\inf \Set{ \delta_r^{(Z,\mathsf{d})}((\iota_X(X)&,\iota_X(x_0)),\\
&(\iota_Y(Y),\iota_Y(y_0)))}{ (Z,\mathsf{d}) &\text{ metric space,}\\
&\iota_X:X\hookrightarrow Z, \iota_Y:Y\hookrightarrow Z,\\
&\iota_X \text{ isometry from $(X,\mathsf{d}_X)$ into $(Z,\mathsf{d}_Z)$,}\\
&\iota_Y \text{ isometry from $(Y,\mathsf{d}_Y)$ into $(Z,\mathsf{d}_Z)$}
}\text{.}
\end{equation*}
\end{definition}

We thus may define convergence of pointed proper metric spaces as:

\begin{definition}\label{GH-cv-def}
A net $(X_j,\mathsf{d}_j,x_j)_{j\in J}$ of pointed proper metric spaces \emph{converges in the sense of Gromov-Hausdorff} to a pointed proper metric space $(X,\mathsf{d},x)$ when:
\begin{equation*}
\forall r > 0 \quad \lim_{j\in J} \Delta_r((X_j,\mathsf{d}_j,x_j),(X,\mathsf{d},x)) = 0\text{.}
\end{equation*}
\end{definition}

Remarkably, a net of compact metric spaces converge for the Gromov-Hausdorff distance described in the previous section if and only if it converges in the sense of Definition (\ref{GH-cv-def}), for any appropriate choice of base points.

There is, in fact, a distance associated with the convergence in the sense of Gromov-Hausdorff, which was the original Gromov-Hausdorff distance in\-tro\-du\-ced in \cite{Gromov81}:

\begin{theorem}
  We let:
  \begin{equation*}
    \GHl((X,x_0,Y,y_0)) = \max\left\{ \inf \Set{r>0}{\Delta_{\frac{1}{r}}((X,x_0),(Y,y_0)) < r} , \frac{1}{2}\right\} \text{,}
  \end{equation*}

  A net $(X_j,\mathsf{d}_j,x_j)_{j\in J}$ of pointed proper metric spaces converges to a pointed proper metric space $(X,\mathsf{d},x)$ in the Gromov-Hausdorff sense if and only if:
\begin{equation*}
  \lim_{j\in J}\GHl((X_j,\mathsf{d},x_j),(X,\mathsf{d},x)) = 0\text{.}
\end{equation*}
\end{theorem}

These results are well-known but not always given in details, so we offer a detailed survey in \cite[Section 2]{Latremoliere14b}. 

Now, we provide in \cite{Latremoliere14b} a generalization of the notion of Gromov-Hausdorff convergence for \emph{quantum} proper metric spaces.

\subsection{Quantum Proper Metric Spaces}

We begin with the definition of a quantum equivalent to the notion of a proper metric space. Proper spaces are complete, so our approach also begins the exploration of the notion of a complete {\lcqms}. A first surprise in \cite{Latremoliere14b} is that we wish to relate the topography and the Lip-norm of a {\lcqms} more tightly than in \cite{Latremoliere12b}:

\begin{definition}[\cite{Latremoliere14b}, Definition 3.2.3]
A {\lcqms} $(\A,\Lip,\M)$ is \emph{standard} when the set:
\begin{equation*}
\set{ m \in \sa{\M} } { \Lip(m) < \infty }
\end{equation*}
is dense in $\sa{\M}$, i.e. when $(\M,\Lip)$ is a Lipschitz pair (where we use the same notation for the restriction of $\Lip$ to $\sa{\M}$ and $\Lip$ itself).
\end{definition}

The motivation for the notion of a standard {\lcqms} is:

\begin{theorem}[\cite{Latremoliere14b}, Theorem 3.2.5]
Let $(\A,\Lip,\M)$ be a {\lcqms}. Then $(\A,\Lip,\M)$ is standard if and only if $(\M,\Lip,\M)$ is a {\lcqms}.
\end{theorem}

Thus, while every {\lcqms} comes equipped with a natural metric on its topography by Theorem (\ref{induced-metric-thm}), we will use another metric associated with standard {\lcqms}:
\begin{notation}
If $(\A,\Lip,\M)$ is a standard {\lcqms} then we denote the {\mongekant} associated with $(\M,\Lip,\M)$ by $\sigmaKantorovich{\Lip}$.
\end{notation}

The reason for this choice is simple: we wish to define a topology which is capable of distinguishing between {\pqms s}, which in particular means distinguishing topographies. This latter point forces us to have a stronger relationship between Lip-norms and topographies.

We need another notion to introduce the concept of a {\pqms}: elements which are, in effect, locally supported from the perspective of the topography:

\begin{notation}
Let $(\A,\Lip,\M)$ be a Lipschitz triple. Let $K\in\compacts{\M}$. The set $\sa{\A}\cap\indicator{K}\A\indicator{K}$ is denoted by $\lsa{\A}{\M}{K}$.

We define:
\begin{equation*}
\Loc{\A}{\M}{K} = \set{ a \in \lsa{\A}{\M}{K} }{ \Lip(a) < \infty }\text{.}
\end{equation*}
We also set:
\begin{equation*}
\Loc{\A}{\M}{\star} = \bigcup_{K \in \compacts{\M^\sigma}}\Loc{\A}{\M}{K}\text{.}
\end{equation*}
Moreover, if $\mu$ is a pure state of $\M$ and $r > 0$, then we denote:
\begin{equation*}
\Loc{\A}{\M}{\M^\sigma [ \mu, r ]_{\sigmaKantorovich{\Lip}}}
\end{equation*}
simply by $\Loc{\A}{\M}{\mu,r}$.
\end{notation}

With all these ingredients, we thus can define:

\begin{definition}[\cite{Latremoliere14b}, Definition 3.2.7]\label{pqms-def}
A {\lcqms} $(\A,\Lip,\M)$ is a \emph{\pqms} when:
\begin{enumerate}
\item $\A$ is separable,
\item $(\A,\Lip)$ is a Leibniz Lipschitz pair,
\item $\Lip$ is lower semi-continuous with respect to the norm topology on $\sa{\A}$,
\item $\Loc{\A}{\M}{\star}$ is norm dense in $\dom{\Lip}$,
\item the set $\Loc{\A}{\M}{\star}\cap\M_\A$ is dense in $\sa{\M_\A}$ (in particular, $(\A,\Lip,\M)$ is standard),
\item $(\M^\sigma,\sigmaKantorovich{\Lip})$ is a proper metric space.
\end{enumerate}
\end{definition}

Our notion of a {\pqms} is the weakest we could use to define our Gromov-Hausdorff hypertopology; however we believe the natural concept is as follows:

\begin{definition}[\cite{Latremoliere14b}, Definition 3.2.12]\label{strongly-proper-def}
A triple $(\A,\Lip,\M)$ is \emph{strongly proper} quantum metric space when:
\begin{enumerate}
\item $\Lip$ is defined on a dense subset of $\A$,
\item $\A$ is separable, and $\Lip$ is lower semi-continuous,
\item $(\A,\Lip,\M)$ is a standard {\lcqms} (where we identify $\Lip$ with its restriction to $\sa{\A}$),
\item for all $a,b\in\A$ we have:
\begin{equation*}
\Lip(ab)\leq\|a\|_\A \Lip(b) + \|b\|_\A\Lip(a)\text{,}
\end{equation*}
\item there exists an approximate unit $(e_n)_{n\in\N}$ in $\sa{\M}$ for $\A$ such that for all $n\in\N$, we have $\|e_n\|_\A\leq 1$ and $e_n \in \Loc{\A}{\M_\A}{\star}$, $\lim_{n\rightarrow\infty} \Lip(e_n) = 0$.
\end{enumerate}
\end{definition}

In general, we see that strongly proper implies proper:

\begin{proposition}[\cite{Latremoliere14b}, Proposition 3.2.14]\label{strongly-proper-implies-proper-prop}
A strongly proper quantum metric space $(\A,\Lip,\M)$ is a {\pqms}, and moreover for any $a\in\sa{\A}$ with $\Lip(a)<\infty$ there exists a sequence $(a_n)_{n\in\N}$ with $a_n\in\Loc{\A}{\M}{\star}$ for all $n\in\N$, converging to $a$ in norm and such that $\lim_{n\rightarrow\infty}\Lip(a_n) = \Lip(a)$. If moreover $a\in\M_\A$ then we can choose $a_n \in \M_\A$ for all $n\in\N$.
\end{proposition}

Last, we need a notion of a pointed space. We simply propose to pick a point in the spectrum of the topography. As a result, we get:

\begin{definition}[\cite{Latremoliere14b}, Definition 3.2.10]
A quadruple $(\A,\Lip,\M,\mu)$ is a \emph{\pqpms} when $(\A,\Lip,\M)$ is a {\pqms} and $\mu \in \M^\sigma$ is a pure state of $\M$. The state $\mu$, identified with a point in $\M^\sigma$, is called the \emph{base point} of $(\A,\Lip,\M,\mu)$.
\end{definition}

\subsection{Tunnels}

We follow the model of the construction of the Gromov-Hausdorff propinquity as described in this survey. However, the notion of tunnel is more subtle. In general, asking for isometric *-epimorphisms is too strict, and we relax the notion of tunnels somewhat --- although, in the case when we work with {\Lqcms s}, we recover our original concept.

A tunnel will be a special king of passage:

\begin{definition}[\cite{Latremoliere14b}, Definition 4.1.1]\label{passage-def}
Let:
\begin{equation*}
\mathds{A}_1=(\A_1,\Lip_1,\M_1,\mu_1)\text{ and }\mathds{A}_2=(\A_2,\Lip_2,\M_2,\mu_2)
\end{equation*}
be two {\pqpms s}. A \emph{passage}:
\begin{equation*}
(\D,\Lip_\D,\M_\D,\pi_1,\mathds{A}_1,\pi_2,\mathds{A}_2)
\end{equation*}
from $\mathds{A}_1$ to $\mathds{A}_2$ is a {\lcqms} $(\D,\M_\D,\Lip_\D)$ such that:
\begin{enumerate}
\item the Lip-norm $\Lip_\D$ is lower semi-continuous with respect to the norm $\|\cdot\|_\D$ of $\D$,
\item $\pi_\A$ and $\pi_\B$ are proper *-morphisms which map $\M_\D$ to, respectively, $\M_\A$ and $\M_\B$ (such maps are called \emph{topographic morphisms}).
\end{enumerate}
\end{definition}

\begin{notation}
Let $\mathds{A}$ and $\mathds{B}$ be two {\pqpms s}. If $\tau$ is a passage from $\mathds{A}$ to $\mathds{B}$, then the \emph{domain} $\dom{\tau}$ of $\tau$ is $\mathds{A}$ while the \emph{co-domain} $\codom{\tau}$ of $\tau$ is $\mathds{B}$.
\end{notation}

Now, we produce a form of local admissibility for a passage, which leads to our concept of tunnel. The key idea here is that a tunnel is a passage for which we can define a notion of local length. Thus, the quantity associated with tunnels, in this context, becomes part of the notion of tunnel itself. We refer to \cite{Latremoliere14b} for a detailed account of this matter.

To understand the notion of weak admissibility for tunnels, we first introduce the notion of a lift set and a tunnel set for a passage. 

\begin{definition}[\cite{Latremoliere14b}, Definition 4.1.3]\label{lift-set-def}
Let:
\begin{equation*}
\mathds{A} = (\A,\Lip_\A,\M_\A,\mu_\A)\text{ and }\mathds{B} = (\B,\Lip_\B,\M_\B,\mu_\B)
\end{equation*}
be two {\pqpms s}. Let:
\begin{equation*}
\tau = (\D,\Lip_\D,\M_\D,\pi_\A,\mathds{A},\pi_\B,\mathds{B})
\end{equation*}
be a passage from $\mathds{A}$ to $\mathds{B}$. If $K\in\compacts{\M_\D^\sigma}$, $l > 0$, $\varepsilon > 0$, $r > 0$ and $a\in\sa{\A}$ with $\Lip_\A(a)\leq l$, then the \emph{lift set} of $a$ for $\tau$ associated with $(l,r,\varepsilon,K)$ is:
\begin{equation*}
\liftsettunnel{\tau}{a}{l,r,\varepsilon,K} = \Set{d\in\Loc{\D}{\M_\D}{K}}{&\pi_\A(d) = a,\\ &\pi_\B(d) \in \Loc{\B}{\M_\B}{\mu_\B,r+4\varepsilon},\\ &\Lip_\D(d)\leq l}\text{,}
\end{equation*}
and the \emph{target set} of $a$ for $\tau$ associated with $(l,r,K,\varepsilon)$ is:
\begin{equation*}
\targetsettunnel{\tau}{a}{l,r,\varepsilon,K} = \pi_\B\left(\liftsettunnel{\tau}{a}{l,r,\varepsilon,K}\right)\text{.}
\end{equation*}
\end{definition}

Of course, as defined, a lift set for some passage may be empty. The key to our notion of tunnel is, indeed, related to when lift sets are not empty. In essence, our notion of admissibility for a passage, which depends on the choice of some radius $r > 0$, relies on whether one can lift elements locally supported on a ball centered at the base point and of radius $r$ for that passage. Formally:

\begin{definition}[\cite{Latremoliere14b}, Definition 4.1.4]\label{left-admissible-def}
Let:
\begin{equation*}
\mathds{A} = (\A,\Lip_\A,\M_\A,\mu_\A)\text{ and }\mathds{B} = (\B,\Lip_\B,\M_\B,\mu_\B)
\end{equation*}
be two {\pqpms s}. Let:
\begin{equation*}
\tau = (\D,\Lip_\D,\M_\D,\pi_\A,\mathds{A},\pi_\B,\mathds{B})
\end{equation*}
be a passage from $\mathds{A}$ to $\mathds{B}$. Let $r > 0$. A pair $(\varepsilon,K)$, where $\varepsilon > 0$ and $K\in\compacts{\M_\D^\sigma}$, is \emph{$r$-left admissible} when:
\begin{enumerate}
\item $\pi_\A^\ast\left(\StateSpace[\A|\mu_\A,r]\right) \subseteq \StateSpace[\D|K]$,
\item we have:
\begin{equation*}
\StateSpace[\D|K]\almostsubseteq{(\StateSpace(\D),\Kantorovich{\Lip_\B})}{\varepsilon} \pi_\A^\ast\left(\StateSpace[\A|\mu_\A,r+4\varepsilon]\right)\text{,}
\end{equation*}
\item for all $a\in\Loc{\A}{\M_\A}{\mu_\A,r}$, we have:
\begin{equation*}
\targetsettunnel{\tau}{a}{\Lip_\A(a),r,\varepsilon,K} \not= \emptyset\text{,}
\end{equation*}
\item for all $a\in\Loc{\A}{\M_\A}{\mu_\A,r}\cap \M_\A$, we have:
\begin{equation*}
\targetsettunnel{\tau}{a}{\Lip_\A(a),r,\varepsilon,K}\cap\M_\B\not=\emptyset\text{,}
\end{equation*}
\item for all $d\in\Loc{\D}{\M_\D}{K}$, we have $\Lip_\B(\pi_\B(d)) \leq \Lip_\D(d)$.
\end{enumerate}
\end{definition}

We then define our notion of admissibility by symmetrizing left-admissibility, and adding appropriate requirements related to the Leibniz property:

\begin{definition}\cite{Latremoliere14b}, Definition 4.1.8]\label{admissible-pair-def}
Let:
\begin{equation*}
\tau = (\D,\Lip_\D,\M_\D,\pi,\dom{\tau},\rho,\codom{\tau})
\end{equation*}
be a passage and let $r > 0$. 
\begin{enumerate}
\item A pair $(\varepsilon,K)$ is \emph{$r$-right admissible} when $(\varepsilon,K)$ is $r$-left admissible for $\tau^{-1} $ (which is the passage where we simply switch the domain and codomain of $\tau$),
\item A pair $(\varepsilon,K)$ is \emph{$r$-admissible} when:
\begin{itemize}
\item $(\varepsilon,K)$ is both $r$-left and $r$-right admissible for $\tau$,
\item if $\mu$ and $\nu$ are the respective base points of $\dom{\tau}$ and $\codom{\tau}$, then $\sigmaKantorovich{\Lip_\D}(\mu\circ\pi,\nu\circ\rho)\leq\varepsilon$,
\item for all $d,d'\in\Loc{\D}{\M_\D}{K}$, we have:
\begin{equation}\label{admissible-Jordan-eq}
\Lip_\D\left(\Jordan{d}{d'}\right)\leq \Lip_\D(d)\|d'\|_\D+\Lip_\D(d')\|d\|_\D\text{,}
\end{equation} 
\item for all $d,d'\in\Loc{\D}{\M_\D}{K}$, we have:
\begin{equation}\label{admissible-Lie-eq}
\Lip_\D\left(\Lie{d}{d'}\right)\leq \Lip_\D(d)\|d'\|_\D+\Lip_\D(d')\|d\|_\D\text{,}
\end{equation}
\end{itemize}
\end{enumerate}
\end{definition}

In order to associate a number to tunnels, we need to ensure that there exists at least one admissible number. 

\begin{definition}[\cite{Latremoliere14b}, Definition 4.1.10]\label{admissible-def}
A number $\varepsilon > 0$ is $r$-admissible when there exists a family $(K_t)_{t \in (0,r]}$ of compacts of $\M_\D^\sigma$ such that:
\begin{enumerate}
\item for all $t\leq t' \in (0,r]$, we have $K_t\subseteq K_{t'}$,
\item for all $t \in (0,r]$, the pair $(\varepsilon,K_t)$ is $t$-admissible.
\end{enumerate}
\end{definition}

\begin{notation}\label{admissible-notation}
The set of $r$-admissible numbers for some passage $\tau$ is denoted by $\Adm(\tau|r)$.
\end{notation}

The apparent complexity in the notion of an admissible number is justified in part by the following desirable monotonicity:

\begin{remark}\label{inclusion-adm-rmk}
By Definition (\ref{admissible-def}), if $R \geq r$ then:
\begin{equation*}
\Adm(\tau|R) \subseteq \Adm(\tau|r)\text{,}
\end{equation*}
and
\begin{equation*}
\Adm(\tau^{-1}|r) = \Adm(\tau|r)\text{.}
\end{equation*}
\end{remark}

An $r$-tunnel is a passage for which some number is $r$-admissible:

\begin{definition}[\cite{Latremoliere14b}, Definition 4.1.13]\label{tunnel-def}
Let $\mathds{A}$ and $\mathds{B}$ be two {\pqpms s} and $r > 0$. An \emph{$r$-tunnel} $\tau$ from $\mathds{A}$ to $\mathds{B}$ is a passage from $\mathds{A}$ to $\mathds{B}$ such that the set $\Adm(\tau|r)$ is not empty.
\end{definition}

\begin{remark}
If $\tau$ is an $r$-tunnel, then $\tau$ is a $t$-tunnel for all $t\in (0,r]$ and $\tau^{-1}$ is an $r$-tunnel as well by Remark (\ref{inclusion-adm-rmk}).
\end{remark}

Thus, we took a detour in defining tunnels in such a manner that we can now define their extent --- which is given by the infimum of possible admissible numbers:

\begin{definition}[\cite{Latremoliere14b}, Definition 4.1.15]\label{extent-def}
Let $\mathds{A}$ and $\mathds{B}$ be two {\pqpms s}, let $r > 0$, and let $\tau$ be an \emph{$r$-tunnel} from $\mathds{A}$ to $\mathds{B}$. The \emph{$r$-extent} $\tunnelextent{\tau}{r}$ is the non-negative real number:
\begin{equation*}
\tunnelextent{\tau|r} = \inf \Adm(\tau|r) \text{.}
\end{equation*}
When $\tau$ is an $r$-tunnel, we call $r$ a \emph{radius of admissibility}.
\end{definition}

\begin{remark}\label{extent-monotone-rmk}
For any $r$-tunnel $\tau$, we have $\tunnelextent{\tau^{-1}|r} = \tunnelextent{\tau|r}$, and if $t \in (0,r]$ then:
\begin{equation*}
\tunnelextent{\tau|t} \leq \tunnelextent{\tau|r}\text{,}
\end{equation*}
by Remark (\ref{inclusion-adm-rmk}).
\end{remark}

At this point, it may be quite obvious that this new approach to tunnels agrees with our standard approach for the Gromov-Hausdorff propinquity; nonetheless we proved:

\begin{proposition}[\cite{Latremoliere14b}, 4.1.17]\label{compact-tunnel-prop}
If $\mathds{A}_1 = (\A_1,\Lip_1,\M_1,\mu_1)$ and $\mathds{A}_2 = (\A_2,\Lip_2,\M_2,\mu_2)$ are two {\pqpms s}, if $r > 0$, and if:
\begin{equation*}
\tau = (\D,\Lip_\D,\M_\D,\pi_1,\mathds{A}_1,\pi_2,\mathds{A}_2)
\end{equation*}
 is an \emph{$r$-tunnel} from $\mathds{A}_1$ to $\mathds{A}_2$ such that, for some $j\in\{1,2\}$, we have:
\begin{equation*}
r \geq \diam{\M_j^\sigma}{\sigmaKantorovich{\Lip_j}}\text{,}
\end{equation*}
then:
\begin{enumerate}
\item $(\A_1,\Lip_1)$, $(\A_2,\Lip_2)$ and $(\D,\Lip_\D)$ are {\Lqcms s},
\item if $(\varepsilon,K)$ is $r$-admissible, then $K = \M_\D^\sigma$,
\item for all $a\in\sa{\A_j}$, we have:
\begin{equation*}
\Lip_j(a) = \inf\Set{\Lip_\D(d)}{\pi_j(d) = a}\text{,}
\end{equation*}
\item $\pi_j$ is a *-epimorphism.
\end{enumerate}
\end{proposition}

\subsection{A Gromov-Hausdorff Topology}

We prove in \cite[Theorem 4.2.1]{Latremoliere14b} that a tunnel composition process similar to Theorem (\ref{tunnel-composition-thm}) can be defined for our tunnels between {\pqms s}. The notions of lift sets and target sets are also extended in \cite{Latremoliere14}, and they enjoy various local versions of the properties of tunnels, seen as morphisms of sort, between {\Lqcms s}. These matters are however rather technical as they always involve working ``locally'', i.e. involving the topography quite explicitly.

None the less, we can formulate a local form of the propinquity as follows:

\begin{notation}
The set of all $r$-tunnels from $\mathds{A}$ to $\mathds{B}$ is denoted by:
\begin{equation*}
\tunnelset{\mathds{A}}{\mathds{B}}{(r)}\text{,}
\end{equation*}
for any two {\pqpms s} $\mathds{A}$ and $\mathds{B}$.This set may be empty.
\end{notation}
We should note that in \cite{Latremoliere14b}, we do introduce a notion of appropriate classes of tunnels, in the same spirit as in the compact setting; we avoid overloading this already involved section with more concepts and notations and refer to \cite{Latremoliere14b} for discussions on all these topics.

We thus can define:

\begin{definition}[\cite{Latremoliere14b}, Definition 5.1.1]\label{local-propinquity-def}
Let $\mathds{A}$ and $\mathds{B}$ be two {\pqpms s} and $r > 0$. The \emph{$r$-local propinquity} between $\mathds{A}$ and $\mathds{B}$ is:
\begin{equation*}
\mathsf{\Lambda}^{\#}_{r}(\mathds{A},\mathds{B}) = \inf\set{ \tunnelextent{\tau|r} }{\tau \in \tunnelset{\mathds{A}}{\mathds{B}}{(r)} }
\end{equation*}
with the usual convention that the infimum of the empty set is $\infty$.
\end{definition}

We then showed in \cite{Latremoliere14b} that the local propinquity enjoys properties which allow us to define an equivalent of the Gromov-Hausdorff distance:

\begin{definition}[\cite{Latremoliere14b}, Definition 5.2.1]
The \emph{topographic Gromov-Hausdorff Propinquity} $\mathsf{\Lambda}^\#(\mathds{A},\mathds{B})$ is the non-negative real number:
\begin{equation*}
\mathsf{\Lambda}^\#(\mathds{A},\mathds{B}) = \max\left\{\inf\set{\varepsilon > 0}{\mathsf{\Lambda}^\#_{\frac{1}{\varepsilon}}(\mathds{A},\mathds{B}) < \varepsilon},\frac{\sqrt{2}}{4}\right\}\text{.}
\end{equation*}
\end{definition}

Now, much work is involved in \cite{Latremoliere14b} to prove that, in fact, the topographic propinquity is an infra-metric with the property that:

\begin{theorem}[\cite{Latremoliere14b}, Theorem 5.3.7]\label{lc-coincidence-thm}
If:
\begin{equation*}
\propinquity{\mathcal{T}}(\mathds{A},\mathds{B}) = 0
\end{equation*}
for two {\pqpms s} $\mathds{A},\mathds{B}$, then there exists a pointed isometric isomorphism $\pi : \mathds{A} \longrightarrow \mathds{B}$.
\end{theorem}
We note that a pointed morphism is a morphism whose dual map associates one base point to another.

Thus, our topographic propinquity defines a Hausdorff topology on the class of all {\pqpms s}. 

Moreover, our new topology extends both the Gromov-Hausdorff propinquity topology on {\Lqcms s} and the Gromov-Hausdorff topology on proper metric spaces:

\begin{theorem}[\cite{Latremoliere14b}, Theorem 6.1.1]
If a sequence:
\begin{equation*}
(X_n,\mathsf{d}_n,x_n)_{n\in\N}
\end{equation*}
of pointed proper metric spaces converges to a pointed proper metric space $(Y,\mathsf{d}_Y,y)$ for the Gromov-Hausdorff distance, then the sequence:
\begin{equation*}
\left(C_0(X_n),\Lip_n,C_0(X_n),x_n\right)_{n\in\N}
\end{equation*}
converges to $(C_0(Y),\Lip,C_0(Y),y)$ for the topographic Gromov-Hausdorff propinquity, \\ where $\Lip_n$ and $\Lip$ are the Lipschitz seminorms on, respectively, $C_0(X_n)$ for $\mathsf{d}_n$ for any $n\in\N$, and $C_0(Y)$ for $\mathsf{d}_Y$.
\end{theorem}

\begin{theorem}[\cite{Latremoliere14b}, Theorem 6.2.1]
A sequence $(\A_n,\Lip_n)$ of {\Lqcms s} converges to a {\Lqcms} $(\A,\Lip)$ for the dual propinquity if, and only if it converges for the topographic Gromov-Hausdorff propinquity.
\end{theorem}

We thus propose that the topographic propinquity provides a possible avenue to discuss the convergence of {\pqpms s}.

\providecommand{\bysame}{\leavevmode\hbox to3em{\hrulefill}\thinspace}
\providecommand{\MR}{\relax\ifhmode\unskip\space\fi MR }
\providecommand{\MRhref}[2]{%
  \href{http://www.ams.org/mathscinet-getitem?mr=#1}{#2}
}
\providecommand{\href}[2]{#2}

\vfill

\end{document}